\documentclass[10pt,reqno]{amsart}

\usepackage{amsmath}
\usepackage{amsthm}
\usepackage{amssymb}
\usepackage{amsfonts}
\usepackage{amsxtra}
\usepackage{fullpage}
\usepackage{amscd}
\usepackage{color} 
\usepackage{bm} 
\usepackage{mathrsfs}
\usepackage{subfigure}
\usepackage{tikz}
\usepackage{tikz-cd}
\usetikzlibrary{cd}
\usepackage{enumerate}
\usepackage[all]{xy}
\usepackage[mathscr]{eucal}
\usepackage{verbatim}

\usepackage[normalem]{ulem}

\let\nc\newcommand
\let\renc\renewcommand

\theoremstyle{plain}
\newtheorem{thm}{Theorem}
\newtheorem{prop}[thm]{Proposition}
\newtheorem{cor}[thm]{Corollary}
\newtheorem{lem}[thm]{Lemma}

\theoremstyle{definition}

\newtheorem{example}[thm]{Example}

\newtheorem{remark}[thm]{Remark}
\newtheorem{rem}[thm]{Remark}
\numberwithin{thm}{section}

\makeatletter
\renewcommand{\subsection}{\@startsection{subsection}{2}{0pt}{-3ex
plus -1ex minus -0.2ex}{-2mm plus -0pt minus
-2pt}{\normalfont\bfseries}} \makeatother

\numberwithin{equation}{section}

\newcommand{\Lmod}[1]{#1\text{-}{\mathsf{mod}}}

\def\gbf#1{\mbox{$\mathbf{#1}$}}

\DeclareMathOperator{\res}{{\mathrm{res}}}

\DeclareMathOperator{\Proj}{\mathrm{Proj}}

\DeclareMathOperator{\Ext}{\mathrm{Ext}}

\DeclareMathOperator{\End}{\mathrm{End}}

\DeclareMathOperator{\rk}{\mathrm{rk}}

\DeclareMathOperator{\gr}{\mathrm{gr}}
\DeclareMathOperator{\Hilb}{{\mathrm{Hilb}}}

\DeclareMathOperator{\Tr}{\mathrm{Tr}}

\DeclareMathOperator{\Rep}{\mathrm{Rep}}

\newcommand{\bv}{\boldsymbol{v}}

\newcommand{\beq}{\begin{equation}\label}
\newcommand{\eeq}{\end{equation}}

\DeclareMathOperator{\Spec}{\mathrm{Spec}}

\newcommand{\iso}{{\;\stackrel{_\sim}{\to}\;}}

\DeclareMathOperator{\Hom}{\mathrm{Hom}}
\DeclareMathOperator{\GL}{\mathrm{GL}}

\nc{\Z}{\mathbb{Z}}
\newcommand{\N}{\mathbb{N}}
\newcommand{\Q}{\mathbb{Q}}

\newcommand{\C}{\mathbb{C}}
\newcommand{\h}{\mathfrak{h}}

\nc{\rank}{\textrm{rank} \,}
\nc{\ds}{\dots}

\let\mc\mathcal
\let\mf\mathfrak

\nc{\mbf}{\mathbf}
\nc{\Res}{\mathsf{Res} \, }
\nc{\Ind}{\mathsf{Ind} \, }
\nc{\cont}{\textrm{cont}}

\nc{\msf}{\mathsf}

\nc{\minusone}{-1}
\nc{\minustwo}{-2}
\nc{\Mod}{\mathrm{Mod} \,}
\nc{\ms}{\mathscr}
\nc{\Frac}{\mathrm{Frac} \,}
\nc{\ra}{\rightarrow}
\nc{\hra}{\hookrightarrow}
\nc{\lab}{\label}
\renc{\O}{\mc{O}}
\nc{\Tan}{\mc{T}}
\nc{\ul}{\underline}
\nc{\s}{\mathfrak{S}}
\nc{\g}{\mf{g}}
\nc{\pa}{\partial}
\nc{\tit}{\textit}
\nc{\Maxspec}{\mathrm{Maxspec} \, }
\nc{\gldim}{\mathrm{gl.dim}}
\nc{\rkm}{\mathrm{rk} \, (\mf{m})}
\nc{\sm}{\mathrm{sm}}
\nc{\PD}{\mathbb{PD}}
\nc{\hilb}{\textrm{Hilb}}
\nc{\<}{\langle}
\renc{\>}{\rangle}
\nc{\T}{\mathbb{T}}
\nc{\X}{\mathbb{X}}
\nc{\F}{\mathbb{F}}
\nc{\id}{\msf{id}}
\nc{\A}{\mathbb{A}}
\nc{\Grat}{\mc{Grat}}
\nc{\Squo}[1]{\A^{(#1)}}
\nc{\twist}{\mathrm{twist}}
\nc{\Cd}{\mc{C}}
\nc{\Span}{\mathrm{Span}}
\nc{\Grass}{\mathrm{Gr}}
\nc{\Supp}{\mathrm{Supp}}
\nc{\Irr}{\mathrm{Irr}}
\renc{\o}{\otimes}
\renc{\gr}{\mathsf{gr}}
\nc{\fin}{\mathrm{fin}}
\nc{\aff}{\mathrm{aff}}
\nc{\algD}{\mf{D}}
\nc{\hr}{\mf{h}_{\textrm{reg}}}
\nc{\D}{\mathscr{D}}
\nc{\PIdeg}{\mathrm{P.I.-degree}}
\nc{\ch}{\mathrm{ch}}
\nc{\ev}{\mathsf{ev}}
\nc{\Stab}{\mathrm{Stab}}
\nc{\Der}{\mathrm{Der}}
\nc{\rightsim}{\stackrel{\sim}{\longrightarrow}}
\nc{\HZ}{H_{\mbf{h},\Z}(\Z_m)}
\nc{\sing}{\mathrm{sing}}
\nc{\dd}{\mathscr{D}}
\nc{\bc}{\mathbf{c}}
\nc{\vc}{\underline{\mathbf{c}}}
\nc{\ba}{\mathbf{a}}
\nc{\reg}{\mathrm{reg}}
\nc{\Amp}{\mathrm{Amp}}
\nc{\Nef}{\mathrm{Nef}}
\nc{\SL}{\mathrm{SL}}
\nc{\Sp}{\mathrm{Sp}}
\nc{\Sym}{\mathrm{Sym}}
\nc{\Mov}{\mathrm{Mov}}
\nc{\Pic}{\mathrm{Pic}}
\nc{\Cs}{\C^{\times}}
\nc{\Nak}[3]{\mf{M}_{{#1}} ({#2},{#3}) }
\nc{\Naka}[2]{\mf{M}({#1},{#2}) }
\nc{\Mtheta}[1]{\mc{M}_{#1}}
\DeclareMathOperator{\Seshadri}{\mathrm{S}} 
\DeclareMathOperator{\head}{hd}
\DeclareMathOperator{\tail}{tl}
\nc{\bw}{\mathbf{w}}
\renewcommand{\bm}{\mathbf{m}}
\nc{\bn}{\mathbf{n}}
\nc{\CB}{\mathrm{CB}}
\nc{\GVect}{\Lambda}
\nc{\pZ}{\overline{Z}}
\nc{\Tang}{\mc{T}}

\newcommand{\one}{\ensuremath{(\mathrm{i})}}
\newcommand{\two}{\ensuremath{(\mathrm{ii})}}
\newcommand{\three}{\ensuremath{(\mathrm{iii})}}
\newcommand{\four}{\ensuremath{(\mathrm{iv})}}

\newcommand{\git}{\ensuremath{/\!\!/\!}}

\begin{document}

\title{Birational geometry of symplectic quotient singularities}

\author{Gwyn Bellamy}
\address{School of Mathematics and Statistics, University of Glasgow, University Place,
Glasgow G12 8QQ,
United Kingdom.}
\email{gwyn.bellamy@glasgow.ac.uk}
\urladdr{http://www.maths.gla.ac.uk/~gbellamy/}

\author{Alastair Craw} 
\address{Department of Mathematical Sciences, 
University of Bath, 
Claverton Down, 
Bath BA2 7AY, 
United Kingdom.}
\email{a.craw@bath.ac.uk}
\urladdr{http://people.bath.ac.uk/ac886/}

\begin{abstract}
For a finite subgroup $\Gamma\subset \SL(2,\C)$ and for $n\geq 1$, we use variation of GIT quotient for Nakajima quiver varieties to study the birational geometry of the Hilbert scheme of $n$ points on the minimal resolution $S$ of the Kleinian singularity $\C^2/\Gamma$. It is well known that $X:=\Hilb^{[n]}(S)$ is a projective, crepant resolution of the symplectic singularity $\C^{2n}/\Gamma_n$, where $\Gamma_n=\Gamma\wr\s_n$ is the wreath product. We prove that every projective, crepant resolution of $\C^{2n}/\Gamma_n$ can be realised as the fine moduli space of $\theta$-stable $\Pi$-modules for a fixed dimension vector, where $\Pi$ is the framed preprojective algebra of $\Gamma$ and $\theta$ is a choice of generic stability condition. Our approach uses the linearisation map from GIT to relate wall crossing in the space of $\theta$-stability conditions to birational transformations of $X$ over $\C^{2n}/\Gamma_n$. As a corollary, we describe completely the ample and movable cones of $X$ over $\C^{2n}/\Gamma_n$, and show that the Mori chamber decomposition of the movable cone is determined by an extended Catalan hyperplane arrangement of the ADE root system associated to $\Gamma$ by the McKay correspondence. 

In the appendix, we show that morphisms of quiver varieties induced by variation of GIT quotient are semismall, generalising a result of Nakajima in the case where the quiver variety is smooth. \end{abstract}

\maketitle
\tableofcontents

\section{Introduction}

 For a finite subgroup $\Gamma\subset \SL(2,\C)$, let $S\to \C^2/\Gamma$ denote the minimal resolution of the corresponding Kleinian singularity. The well-known paper by Kronheimer~\cite{Kronheimer} realises $S$ as a hyperk\"{a}hler quotient, describes the ample cone of $S$ as a Weyl chamber of the root system of type ADE associated to $\Gamma$ by the McKay correspondence, and constructs the simultaneous resolution of the semi-universal deformation of $\C^2/\Gamma$. 
 In the present paper we provide a natural generalisation of these results to higher dimensions by studying symplectic resolutions of the quotient singularity  $\C^{2n} / \Gamma_n$ for any $n\geq 1$, where $\Gamma_n=\Gamma\wr\s_n$ is the wreath product. We prove that every projective crepant resolution of $\C^{2n} / \Gamma_n$ can be realised as a Nakajima quiver variety, generalising the description of the Hilbert scheme $X:=\Hilb^{[n]}(S)$ of $n$-points on $S$ by Kuznetsov~\cite{Kuznetsov07} (established independently by both Haiman and Nakajima). We also obtain a complete understanding of the birational geometry of $X$ over $\C^{2n} / \Gamma_n$ by describing explicitly the movable cone of $X$ over $\C^{2n} / \Gamma_n$ in terms of an extended Catalan hyperplane arrangement determined by the ADE root system associated to $\Gamma$. Finally, we construct, using quiver GIT, the simultaneous resolution of the universal Poisson deformation of $\C^{2n}/\Gamma_n$.
 
\subsection{Quiver varieties and the linearisation map}
 For $n\geq 1$ and for a finite subgroup $\Gamma\subset \SL(2,\C)$, the symmetric group $\s_n$ acts on the direct product $\Gamma^n$ by permuting the factors, and the wreath product is defined to be the semidirect product $\Gamma_n:= \Gamma^n\rtimes \s_n\subset \Sp(2n,\C)$. Throughout the introduction, we assume $n > 1$ and $\Gamma$ is non-trivial (see Section~\ref{sec:n=1}, Remark~\ref{rem:Gammatrivial} and Proposition~\ref{prop:n=1} for these degenerate cases). It is well-known that the Hilbert scheme of $n$ points on $S$ provides a symplectic resolution
 \[
 f\colon X:=\Hilb^{[n]}(S)\longrightarrow Y:=\C^{2n}/\Gamma_n
 \]
  of the corresponding quotient singularity. In particular, $f$ is a projective, crepant resolution of singularities.

 In order to study the birational geometry of $X$ over $Y$, we first recall that $X$ can be constructed by GIT as a quiver variety. Consider the affine Dynkin graph associated to $\Gamma$ by McKay~\cite{McKay80}, where the vertex set is by definition the set of irreducible representations of $\Gamma$. Define the dimension vector $\bv:= n \delta$ and the framing vector $\bw = \rho_0$, where $\delta$ and $\rho_0$ are the regular and trivial representations of $\Gamma$ respectively. If we write $R(\Gamma)$ for the representation ring of $\Gamma$, then our interest lies in studying the quiver varieties $\mathfrak{M}_\theta:= \mathfrak{M}_\theta(\bv,\bw)$, where the GIT stability parameter $\theta$ can be regarded as an element in the rational vector space $\Theta:=\Hom(R(\Gamma),\Q)$; equivalently, $\mathfrak{M}_\theta$ can be regarded as a moduli space of $\theta$-semistable $\Pi$-modules, where $\Pi$ is the framed preprojective algebra of $\Gamma$ (see section~\ref{sec:Nakasec} for details). A result of Kuznetsov~\cite{Kuznetsov07} (also due to Haiman~\cite{Haiman00} and Nakajima~\cite{NakajimaPC}) determines an open GIT chamber $C_-$ in $\Theta$ and a commutative diagram 
  \begin{equation*}
\begin{tikzcd}
 X=\Hilb^{[n]}(S) \ar[r]\ar[d,"{f}"] & \mathfrak{M}_\theta \ar[d,"{f_\theta}"] \\
 Y= \C^{2n}/\Gamma_n\ar[r] &  \mathfrak{M}_0
 \end{tikzcd}
\end{equation*}
 for any $\theta\in C_-$, where the horizontal arrows are isomorphisms and where the right-hand symplectic resolution is obtained by variation of GIT quotient. 
 
 Our first main result calculates explicitly the GIT chamber decomposition of $\Theta$ and describes the geometry of the quiver varieties $\mf{M}_\theta$ whenever $\theta$ lies in a chamber.  To state the result, write $\Phi^+$ for the set of positive roots in the ADE root system of finite type associated to $\Gamma$ by the McKay correspondence, and for $\gamma\in R(\Gamma)$ we write $\gamma^\perp:= \{\theta\in \Theta \mid \theta(\gamma)=0\}$ for the dual hyperplane. 

\begin{thm}
\label{thm:introsmoothness}
For $\theta\in \Theta$, the following are equivalent:
 \begin{enumerate}
     \item[\one] $\theta$ is generic, i.e.\ $\theta$ is contained in a GIT chamber;
     \item[\two] $\theta$ does not lie in any hyperplane from the hyperplane arrangement
\[
\mc{A}:=  \big\{ \delta^\perp, (m\delta\pm \alpha)^\perp \mid \alpha \in \Phi^+, \ 0 \le m < n \big\};
\]
\item[\three] the morphism $f_\theta\colon \mf{M}_\theta\to \mf{M}_0$ obtained by variation of GIT quotient is a projective crepant resolution. 
 \end{enumerate}
  In particular, for any such $\theta\in \Theta$, there is a birational map $\psi_\theta\colon X\dashrightarrow \mf{M}_\theta$ over $Y$ that is an isomorphism in codimension-one.
\end{thm}

 As a result, taking the proper transform along $\psi_\theta$ enables us to identify canonically the N\'{e}ron--Severi spaces of $X$ and $\mf{M}_\theta$ for any generic $\theta$, so we may regard the ample cone of $\mf{M}_\theta$ as lying in $N^1(X/Y)$.
 
   In order to understand the birational geometry of $X$, we exploit the link between wall-crossing for stability parameters and birational transformations of quiver varieties provided by the linearisation map arising from the GIT construction of $\mf{M}_\theta$. To define this map, let $C$ be any GIT chamber in $\Theta$ and fix $\theta\in C$. The quiver variety $\mf{M}_\theta$ comes equipped with a tautological locally-free sheaf $\mc{R}:= \bigoplus_{i\in I} \mc{R}_i$, where the summand $\mc{R}_i$ has rank $v_i$ for each vertex $i\in I$ in the framed affine Dynkin graph (see section~\ref{sec:tautbundles}). 
  The \emph{linearisation map} for the chamber $C$ is the $\Q$-linear map
 \[
 L_C\colon \Theta\longrightarrow N^1(X/Y)
 \]
 defined by sending $\eta = (\eta_i)_{i\in I}$ to the class of the line bundle $L_C(\eta) = \bigotimes_{i\in I} \det(\mc{R}_i)^{\otimes \eta_i}$.   Theorem~\ref{thm:introsmoothness} is a key ingredient in enabling us to prove that $L_C$ is a linear isomorphism that identifies the chamber $C$ with the ample cone $\Amp(\mf{M}_\theta/Y)$ for $\theta\in C$. In order to understand how the linearisation maps $L_C$ are related as we cross a wall between adjacent chambers in $\Theta$, we focus our attention initially on chambers contained in the simplicial cone
\begin{equation}
    \label{eqn:Fintro}
F:=\langle\delta,\rho_1,\dots,\rho_r\rangle^\vee:=\{\theta\in \Theta \mid \theta(\delta)\geq 0, \theta(\rho_i)\geq 0 \text{ for }1\leq i\leq r\}.
\end{equation}
 Note that $F$ is a union of the closures of GIT chambers by Theorem~\ref{thm:introsmoothness}. 
 
  \begin{thm}
 \label{thm:simplifiedmainintro}
  The linearisation maps $L_C$ for chambers $C$ in the cone $F$ glue to define an isomorphism 
  \[
  L_F\colon \Theta\longrightarrow N^1(X/Y)
  \]
  of rational vector spaces that identifies the GIT wall-and-chamber structure of $F$ with the decomposition of $\Mov(X/Y)$ into Mori chambers. In particular, for any generic $\theta\in F$, the moduli space $\mathfrak{M}_\theta$ is the birational model of $X$ determined by the line bundle $L_F(\theta)$.
  \end{thm}
 
 This result provides information only about the quiver varieties $\mf{M}_\theta$ for generic parameters $\theta$ in the cone $F$ (see Theorem~\ref{thm:mainintro} for a stronger statement), but this is all that we require to establish the following result:
 
  \begin{cor}
  \label{cor:introCIconjecture}
 For $n\geq 1$ and a finite subgroup $\Gamma\subset \SL(2,\C)$, suppose that $X^\prime\to \C^{2n}/\Gamma_n$ is a projective, crepant resolution. Then there exists a generic stability parameter $\theta$ such that $X^\prime\cong \mf{M}_\theta$.
 \end{cor}
 
 Analogous statements appear in the literature for certain classes of singularities in dimension three: for Gorenstein affine toric threefolds, see 
 Craw--Ishii~\cite{CrawIshii}, Ishii--Ueda~\cite{IshiiUeda}; and for compound du Val singularities, see Wemyss~\cite{Wemyss18}. In higher dimensions, the analogous result for nilpotent orbit closures is due to Fu~\cite{FuNilpotentResolutions}.
  
 \subsection{Wall-crossing and Ext-graphs}
 \label{sec:proofofsimplifiedmain}
  To prove Theorem~\ref{thm:simplifiedmainintro}, we fix a reference chamber $C$ in $F$ and study how the quiver varieties $\mf{M}_\theta$ and, where necessary, the tautological bundles $\mathcal{R}_i$, change as we cross walls. In fact, the reference chamber need not be the chamber $C_-$ defining $X=\Hilb^{[n]}(S)$, so all of our results are independent of the paper~\cite{Kuznetsov07} cited above. 
  
  For each wall of a chamber $C$ in $F$, we study the morphism $f\colon \mf{M}_\theta\to \mf{M}_{\theta_0}$ obtained by varying a parameter $\theta\in C$ to a parameter $\theta_0$ that is generic in the wall. The idea is to understand $f$ \'{e}tale locally by providing a relative version of the \'{e}tale local description of $\mf{M}_\theta$ by Bellamy--Schedler~\cite{BelSchQuiver}, which in turn builds on work of Crawley--Boevey~\cite{CBnormal}, Nakajima~\cite{Nak1994} and Kronheimer~\cite{Kronheimer}. The key statement is the following result that is valid more generally for the Nakajima quiver variety $\mf{M}_\theta(\bv,\bw)$ associated to any graph, any choice of dimension and framing vectors, and any stability condition; see section~\ref{sec:localnormalform} for the relevant definitions and a more precise statement. 
  
   \begin{thm}
   \label{thm:etalelocal}
 Let $x\in \mf{M}_{\theta_0}$ be a closed point. Then the Ext-graph associated to $x$ admits a dimension vector $\bm$, a framing vector $\bn$ and a stability condition $\varrho$ such that the morphism $f\colon \mf{M}_\theta\to \mf{M}_{\theta_0}$ is equivalent \'{e}tale locally over $x$ to the product of the canonical morphism
 $\mathfrak{M}_\varrho(\bm,\bn)\to \mf{M}_0(\bm,\bn)$ and the identity map on $\C^{2\ell}$ for some $\ell\geq 0$. Moreover, this morphism depends only on the GIT stratum of $\mf{M}_{\theta_0}$ containing $x$.
   \end{thm}

Compare this with the main result of Arbarello--Sacc\`a~\cite[Theorem~1.1]{ArbarelloSacca18} in their study of moduli spaces of sheaves of pure dimension one on a K3 surface.

 Returning to the proof of Theorem~\ref{thm:simplifiedmainintro}, our description of the hyperplane arrangement in Theorem~\ref{thm:introsmoothness}, provides enough information to compute explicitly the Ext-graph associated to any closed point on the quiver variety $\mf{M}_{\theta_0}$, where $\theta_0$ is generic in any GIT wall that lies in $F$. The graphs that arise are quite simple, including for example, the disjoint union of a collection of graphs each comprising one vertex and one edge loop. As such, it is possible to recognise the morphisms $\mathfrak{M}_\varrho(\bm,\bn)\to \mf{M}_0(\bm,\bn)$ from Theorem~\ref{thm:etalelocal} that appear in the \'{e}tale local description of the contractions induced by each wall. In this way, we show for every wall in the interior of $F$ separating chambers $C$ and $C^\prime$ in $F$, that the birational map
 \[
\begin{tikzcd}
\mf{M}_{\theta} \ar[rr,"\varphi",dashed] \ar[dr,"f_{\theta}"'] & & \mf{M}_{\theta'} \ar[dl,"f_{\theta'}"] \\
& \mf{M}_{\theta_0} & 
\end{tikzcd}
\]
induced by crossing the wall is a flop, and moreover, that the line bundles $\det(\mc{R}_i)$ on $\mf{M}_\theta$ are each the proper transform along $\varphi$ of the corresponding line bundle $\det(\mc{R}_i^\prime)$ on $\mf{M}_{\theta^\prime}$. It then follows from the definition that the linearisations maps $L_C$ and $L_{C^\prime}$ of the chambers on either side of the wall agree. Repeating this argument across all chambers in $F$ determines the linear isomorphism $L_F$ from Theorem~\ref{thm:simplifiedmainintro} whose restriction to any chamber $C$ in $F$ identifies $C$ with the cone $\Amp(\mf{M}_\theta/Y)$ for any $\theta\in C$. In addition, we demonstrate that the morphism  $f\colon \mf{M}_\theta\to \mf{M}_{\theta_0}$ induced by moving a GIT parameter from a chamber $C$ of $F$ into a boundary wall of $F$ is necessarily a divisorial contraction; this includes the wall $\delta^\perp\cap F$ of the chamber $C_-$ which induces the Hilbert--Chow morphism $\Hilb^{[n]}(S)\to \Sym^n(S)$. 
 
  \subsection{The movable cone}
   The hyperplanes in the arrangement $\mc{A}$ from Theorem~\ref{thm:introsmoothness} that pass through the interior of the cone $F$ can be computed explicitly, so the decomposition of the movable cone $\Mov(X/Y)$ into Mori chambers can be obtained easily from Theorem~\ref{thm:simplifiedmainintro}: 

\begin{thm}
 \label{thm:movableintro}
  The division of the movable cone $\Mov(X/Y) = L_F(F)$ into Mori chambers is determined by the images under the isomorphism $L_F$ of the hyperplanes $(m\delta-\alpha)^\perp$ for all $0<m<n$ and $\alpha\in \Phi^+$. We have that $\Amp(X/Y) = L_F(C_-)$; more generally,  the chambers in this decomposition are precisely the ample cones of the projective, crepant resolutions of $Y$.
 \end{thm}
  
  This generalises the result of Andreatta--Wi\'{s}niewski~\cite[Theorem~1.1]{AW14} in the case when $n=2$ and $\Phi$ is of type $A_r$ (see Example~\ref{exa:n=2Ar}), and provides an answer to the question of Fu~\cite[Problem~1]{Fu06}. 
  
  It turns out that an affine slice of the movable cone admits a purely combinatorial description. The affine hyperplane $\Lambda:= \{\theta\in \Theta_v \mid \theta(\delta)=1\}$ lies parallel to the supporting hyperplane $\delta^\perp$ of $F$. In particular, the slice $F\cap \Lambda$ determines completely the wall-and-chamber decomposition of $F$, so the image of this slice under $L_F$ determines completely the Mori chamber decomposition of $\Mov(X/Y)$: 
  
 \begin{cor}
 \label{cor:sliceMov}
 The intersection of $\Mov(X/Y)$ with the affine hyperplane $\{L_F(\theta) \mid \theta(\delta)=1\}$ is isomorphic to the decomposition of the fundamental chamber of the $(n-1)$-extended Catalan hyperplane arrangement associated to $\Phi$. 
 \end{cor}
 
 This result provides a geometric realisation of the extended Catalan hyperplane arrangement that was introduced originally by Postnikov--Stanley~\cite{PS00} and studied further by Athanasiadis~\cite{Athanasiadis04}.
 
 Our description of the movable cone also provides new proofs for several results from the literature:
 \begin{itemize}
     \item Corollary~\ref{cor:sliceMov} implies that the number of non-isomorphic projective crepant resolutions of $\C^{2n}/\Gamma_n$ is  
     \[
 \prod_{i=1}^{r} \frac{(n-1)h+d_i}{d_i},
 \]
 where $r$ is the rank, $h$ is the Coxeter number and $d_1,\dots, d_r$ are the degrees of the basic polynomial invariants of the Weyl group $W_\Gamma$. This agrees with the count of Bellamy~\cite[Equation~(1.B)]{BellamyCounting}.
 \item The cone $\Mov(X/Y)$ is simplicial by Theorem~\ref{thm:simplifiedmainintro}, because $F$ is simplical; this is a special case of the result by Andreatta--Wi\'{s}niewski~\cite[Theorem~4.1]{AW14}.
 \item Define an action of $W$ on $N^1(X/Y)$ by setting $s_\delta$ to be reflection in $L_F(\delta^\perp)$ and $s_i$ to be reflection in $L_F(\rho_i^\perp)$ for $1\leq i\leq r$ (compare section~\ref{sec:NamWeylgroup}). Then $\Mov(X/Y)$ is a fundamental domain for $W$. Together with the  observation by Braden--Proudfoot--Webster~\cite[Proposition~2.17]{BLPWAst}, this gives a new description of the action of Namikawa's Weyl group on $N^1(X/Y)$ in our context.
 \item We provide a purely quiver-theoretic 
 proof of the fact that $X$ is a relative Mori Dream Space over $Y$ (see Corollary~\ref{cor:MDS}). This is a special case of \cite[Theorem~3.2]{AW14} when $n=2$, and follows from work of Namikawa~\cite[Lemma~1, Lemma~6]{Namikawa3} when $n>2$.
 \end{itemize}
 
   It is perhaps worth making a philosophical remark about the final point. Any (relative) Mori Dream Space has a finitely generated Cox ring, and in our situation this ring was described by generators and relations by Donten-Bury--Grab~\cite[Section~5]{DontenBuryGrab} when $n=2$ and $\Phi$ is of type $A_1$. While we do not make use of the Cox ring in this paper, the fact that our Corollary~\ref{cor:introCIconjecture} reconstructs all small birational models of $X$ by GIT as quiver varieties for the affine Dynkin graph of $\Gamma$ suggests that the preprojective algebra $\Pi$ should be thought of as a kind of `noncommutative Cox ring' for each of the varieties $\Hilb^{[n]}(S)$ with $n\geq 1$.

 \subsection{Strong version via the Namikawa Weyl group}
  We now explain how to understand the quiver varieties $\mf{M}_\theta$ for parameters $\theta$ that lie beyond the simplicial cone $F$.  For $1\leq i\leq r$, write $s_i\colon \Theta\to \Theta$ for the reflection in the hyperplane $\rho_i^\perp$, and write $s_\delta\colon \Theta\to \Theta$ for reflection in the hyperplane $\delta^\perp$. The \emph{Namikawa Weyl group} is the group  
  \[
  W:=\langle s_\delta, s_1,\dots,s_r\rangle
  \]
  generated by these reflections. We prove (see Proposition~\ref{prop:W}) that the action of $W$ permutes the set of GIT chambers in $\Theta$, and that the simplicial cone $F$ introduced in \eqref{eqn:Fintro} above is a fundamental domain for the action of $W$ on $\Theta$. The next result provides a stronger version of Theorem~\ref{thm:simplifiedmainintro}:
  
  \begin{thm}
 \label{thm:mainintro}
 \begin{enumerate}
 \item[\one] Under the identification of the N\'{e}ron--Severi spaces induced by the birational maps from Theorem~\ref{thm:introsmoothness}, the linearisation maps $L_C$ glue to a piecewise-linear, continuous map 
 \[
 L\colon \Theta\longrightarrow N^1(X/Y).
 \]
 \item[\two] The map $L$ is invariant with respect to the action of $W$ on $\Theta$, i.e.\ $L(\theta)=L(w\theta)$ for all $w\in W$ and $\theta\in \Theta$. In particular, the image of $L$ is the movable cone $\Mov(X/Y)$.
 \item[\three]  The map $L$ is compatible with the chamber decomposition of $\Theta$ and the Mori chamber decomposition of $\Mov(X/Y)$, in the sense that for any chamber $C\subset \Theta_{v}$ and for any $\theta\in C$, the moduli space $\mathfrak{M}_\theta$ is the birational model of $X$ determined by the line bundle $L(\theta)$.
  \item[\four] For each chamber $C\subset \Theta$, the map $L\vert_C \colon C\to \Amp(\mf{M}_\theta/Y)$ is an isomorphism for $\theta\in C$.
 \end{enumerate}
\end{thm}

 The following result was anticipated by Losev~\cite{LosevProcesi}.

\begin{cor}
\label{cor:Worbits}
 Let $C, C^\prime\subset \Theta$ be chambers and let $\theta \in C$,  $\theta^\prime \in C^\prime$. Then $\mathfrak{M}_\theta\cong \mathfrak{M}_{\theta^\prime}$ as schemes over $Y$ if and only if there exists $w \in W$ such that $w(C) = C^\prime$.
 \end{cor}
 
 Once we prove that $L$ is invariant under the action of $W$ as in Theorem~\ref{thm:mainintro}\one, then parts \two-\four\ follow from Theorem~\ref{thm:simplifiedmainintro}. To achieve this, for each reflection $s_1,\dots, s_r$, we study the corresponding Nakajima reflection functor and its effect on the tautological bundles of $\mf{M}_\theta$. The case of the reflection $s_\delta$ has to be treated separately by studying the isomorphism from $\mf{M}_\theta$ to $\mf{M}_{-\iota(\theta)}$, where $\iota$ is either the identity or is induced by an order two symmetry of the McKay graph of $\Gamma$ (see section~\ref{sec:graphinvolution}).
 
  The work of Bezrukavnikov--Kaledin~\cite{SympMcKay} shows that every symplectic resolution $X^\prime\to \C^{2n}/\Gamma_n$ possesses a collection of tilting bundles (called \textit{Procesi bundles}) that induce derived equivalences between the derived category of coherent sheaves on $X^\prime$ and the derived category of $\Gamma_n$-equivariant coherent sheaves on $\C^{2n}$. Remarkably, these Procesi bundles were classified completely by Losev~\cite{LosevProcesi}, at least when $X^\prime$ is a quiver variety $\mf{M}_\theta(\bv,\bw)$. Moreover, he confirmed the first half of a conjecture of Haiman~\cite[Conjecture 7.2.13]{HaimanSurvey}, that there is a unique Procesi bundle on each $\mf{M}_\theta(\bv,\bw)$ whose $\Gamma_{n-1}$-invariant summand is the tautological bundle $\mc{R}_\theta(\bv,\bw)$. Corollaries~\ref{cor:introCIconjecture} and~\ref{cor:Worbits} now imply the following:

  \begin{cor}
  \label{cor:introLosev}
  Let $X^\prime\to \C^{2n}/\Gamma_n$ be a projective, symplectic resolution, and let  $C\subset F$ be the chamber satisfying $L(C)=\Amp(X^\prime/Y)$. For every normalised Procesi bundle $\mathcal{P}$ on $X^\prime$, there exists a unique $w\in W$ such that the $\Gamma_{n-1}$-invariant part of $\mathcal{P}$ is the tautological bundle $\mathcal{R}_{w(\theta)}$ on $\mf{M}_{w(\theta)}\cong X^\prime$ for $\theta\in C$. Moreover, every normalised Procesi bundle on $X^\prime$ arises in this way. In particular, there is a bijection between elements of $W$ and the normalised Procesi bundles on each projective crepant resolution of $\C^{2n}/\Gamma_n$.
 \end{cor}
 
  In addition, confirmation of the second half of Haiman's $\Gamma_n$-constellation conjecture, when combined with Corollary~\ref{cor:introCIconjecture}, would imply every projective crepant resolution $X^\prime$ of $\C^{2n}/\Gamma_n$ is a fine moduli space of stable modules over the skew group algebra $\C[V^{\times n}]\rtimes \Gamma_n$. It would then follow from Bayer--Craw--Zhang~\cite[Section~7]{BCZ17}, together with the derived equivalence of \cite{SympMcKay}, that every such $X^\prime$ can be realised as a moduli space of Bridgeland-stable objects in the derived category of coherent sheaves on $X$.

\subsection{The universal Poisson deformation}

Kronheimer's realisation of the minimal resolution of the Kleinian singularity as a morphism of quiver varieties led to a new construction of the semiuniversal deformation of $\C^2/\Gamma$ and its simultaneous resolution. In higher dimensions, the semi-universal deformation does not behave well; instead the natural object to consider is the universal graded Poisson deformation, as defined by Ginzburg--Kaledin~\cite{GK} and Namikawa~\cite{Namikawa}. 

It was shown by Kaledin--Verbitsky~\cite{KaledinVerbitskyPeriod} and Losev~\cite{Losev} that each symplectic resolution $\mf{M}_{\theta}$ of the quotient singularity $Y$ admits a universal graded Poisson deformation $\mc{X} \rightarrow \mf{h}$, where $\mf{h} = N^1(X/Y) \o_{\Q} \C \cong \Theta \o_{\Q} \C$. Namikawa~\cite{Namikawa} showed that $Y$ also admits a universal graded Poisson deformations $\mc{Y} \rightarrow \mf{h}/ W$. On the other hand, by taking the preimage under the moment map of a point $\lambda \in \mf{h}$, one gets a graded Poisson deformation $\boldsymbol{\mf{M}}_{\theta} \rightarrow \mf{h}$ for any $\theta \in \Theta$. The morphism $f_{\theta} \colon \mf{M}_{\theta} \rightarrow \mf{M}_0$ obtained by variation of GIT quotient extends to a projective morphism $\boldsymbol{f}_{\theta} \colon \boldsymbol{\mf{M}}_{\theta} \rightarrow \boldsymbol{\mf{M}}_{0}$. 

\begin{thm}
Let $C\subseteq \Theta$ be a chamber, and let $\theta \in C$.
\begin{enumerate}
    \item[\one] 
There exists a unique $w \in W$ and graded Poisson isomorphism $\boldsymbol{\mf{M}}_{\theta} \iso \mc{X}$ such that the diagram
$$
\begin{tikzcd}
\boldsymbol{\mf{M}}_{\theta} \ar[r] \ar[d,"\wr"] & \h \ar[d,"w \cdot"] \\
\mc{X} \ar[r] & \h
\end{tikzcd}
$$
commutes. Thus, the flat family $\boldsymbol{\mf{M}}_{\theta} \rightarrow \h$ is the universal graded Poisson deformation of $\mf{M}_{\theta}$. 
\item[\two] The morphism $\boldsymbol{f}_{\theta} \colon \boldsymbol{\mf{M}}_{\theta} \rightarrow \boldsymbol{\mf{M}}_{0}$ is a crepant resolution, and an isomorphism in codimension one. 
\end{enumerate}
\end{thm}

The situation is summarised in the following commutative diagram
$$
 \begin{tikzcd}
   & \mf{M}_\theta\ar[r,"f_\theta"]\ar[ld,hook']\ar[ddr] & \C^{2n}/\Gamma_n\ar[ld,hook',crossing over]\ar[dd]\ar[dr,hook'] & & & \\
\boldsymbol{\mf{M}}_{\theta} \ar[r,"\boldsymbol{f}_{\theta}"] \ar[ddr] & \boldsymbol{\mf{M}}_0 \ar[dd] \ar[rr,crossing over] & & \mc{Y}\ar[dd] & & \\
 & & 0\ar[ld,hook']\ar[rd,hook'] & & & \\
& \h \ar[rr] & & \h/W & 
\end{tikzcd}
$$
where the lower right rectangle is shown to be Cartesian. 
 
  \subsection{Relation to other work}
  We have chosen to state Theorem~\ref{thm:mainintro} in a manner parallel to the main result of Bayer-Macr\`{i}~\cite[Theorem~1.3]{BayerMacriMMP} on moduli spaces of Bridgeland-stable objects on a projective K3 surface. On the face of it, such a link is perhaps unsurprising because $X$ has trivial canonical class and is projective over an affine. However, it is important to note that $X$ is not derived equivalent to the preprojective algebra $\Pi$; the stability space $\Theta$ studied here is isomorphic to the rational Picard group of $X$ rather than the rational Grothendieck group. Put another way, the framed McKay quiver has too few vertices to be the quiver encoding a tilting bundle on $X$ when $n>1$. In particular, the quiver varieties $\mf{M}_\theta(\bv,\bw)$ that we study here cannot be realised directly as moduli spaces of Bridgeland-stable objects in the derived category of coherent sheaves on $X$ in the manner described in \cite[Section~7]{BCZ17}.

Recently, McGerty--Nevins~\cite{QuiverKirwan} have shown that Kirwan surjectivity holds for quiver varieties. That is, for each chamber $C$, with $\theta \in C$, the natural map $K_C \colon \C[\mf{g}] = H_G^*(\mathrm{pt},\C) \rightarrow H^*(\mf{M}_\theta(\bv,\bw),\C)$ is surjective. Our map $L_C$ fits naturally into a commutative diagram
$$
\begin{tikzcd}
 \Theta_v \o_{\Q} \C \ar[r,"L_C"] \ar[d] & N^1(X/Y) \o_{\Q} \C \ar[d] \\
 \mathbb{X}(\mf{g}) \ar[r,"K_C"] & H^2(\mf{M}_\theta(\bv,\bw),\C),
\end{tikzcd}
$$
where $ \mathbb{X}(\mf{g}) \subset \C[\mf{g}]$ is the space of characters of $\mf{g}$. As noted above, it follows easily from Theorem \ref{thm:introsmoothness} that $L_C$ is an isomorphism. The two vertical maps are also isomorphisms, so the map $K_C \colon  \mathbb{X}(\mf{g}) \rightarrow H^2(\mf{M}_\theta(\bv,\bw),\C)$ is an isomorphism too in our case. Therefore, our main results demonstrate precisely the extent to which the linear map $K_C$ extends across chambers to give a linear map on a union of chambers.
 
Since this paper was written, Theorem~\ref{thm:movableintro} enabled Gammelgaard, Gyenge, Szendr\H{o}i and the second author~\cite{CGGS19} to construct the Hilbert scheme of $n$ points on $\mathbb{C}^2/\Gamma$ as a quiver variety $\mathfrak{M}_\theta$ for a particular non-generic parameter $\theta$. Nakajima~\cite{NakajimaEulerChar20} subsequently used this description in his proof of the conjecture describing the generating series of Euler characteristics of $\Hilb^{[n]}(\C^2/\Gamma)$ by Gyenge, N\'{e}methi and Szendr\H{o}i~\cite{GNS18}.
 
\subsection{Acknowledgements}
The first author thanks Travis Schedler for many interesting discussions regarding quiver varieties. The second author thanks Michael Wemyss for a useful discussion about Procesi bundles. We also thank the anonymous referee for many helpful comments and corrections.
 
\section{A wall-and-chamber structure}\label{sec:combinatorics}
We begin by providing an elementary description of a wall-and-chamber decomposition of a rational vector space $\Theta$, together with an action of the Namikawa Weyl group on $\Theta$. This section is purely combinatorics and uses no machinery from Geometric Invariant Theory (GIT).

\subsection{The chamber decomposition}\label{sec:chamberdecomp}
Let $\Gamma\subset \SL(2,\C)$ be a finite subgroup and let $n\geq 1$ be an integer. Let $V$ denote the given 2-dimensional representation of $\Gamma$ and list the irreducible representations as $\rho_0, \dots, \rho_r$, where $\rho_0$ is the trivial representation. 

The McKay graph is the affine Dynkin diagram associated to $\Gamma$ by the McKay correspondence; explicitly, the vertex set is $\Irr(\Gamma)$, and there are $\dim\Hom_{\Gamma}(\rho_i,\rho_j\otimes V)$ edges between vertices $\rho_i$ and $\rho_j$. Since $V$ is self-dual, this is symmetric in $\rho_i$ and $\rho_j$. Let $A_\Gamma$ denote the adjacency matrix of this graph. McKay~\cite{McKay80} observed that if we define $C_\Gamma:= 2\mathrm{Id}-A_\Gamma$ and equip the integral representation ring 
\[
R(\Gamma):=\bigoplus_{0\leq i\leq r} \Z \rho_i
\]
with the symmetric bilinear form given by $(\alpha,\beta)_\Gamma := \alpha^tC_\Gamma\beta$, then we obtain the root lattice of an affine root system $\Phi_{\aff}$ of type ADE in which the McKay graph is the Dynkin diagram, the irreducible representations $\{\rho_0, \dots, \rho_r\}$ provide a system of simple roots, and the regular representation 
 \[
 \delta:=\sum_{0\leq i\leq r} (\dim_\C\rho_i) \rho_i
 \]
 is the minimal imaginary root. The corresponding root system of finite type $\Phi\subset \Phi_\aff$ is the intersection of $\Phi_\aff$ with the integer span of the nontrivial irreducible representations. Let $\Phi^+$ denote the set of positive roots.
 
 Let $\Theta:=\Hom(R(\Gamma),\Q)$ denote the rational vector space whose underlying lattice is dual to $R(\Gamma)$. We adopt the following notation: write elements of $\Theta$ as $\theta=(\theta_0,\dots,\theta_r)\in \Theta$ where $\theta_i:=\theta(\rho_i)$ for $0\leq i\leq r$. Given $v\in R(\Gamma)$, we let $v^\perp:= \{\theta\in \Theta \mid \theta(v)=0\}$ denote the dual hyperplane; and given $v_1,\dots, v_m\in R(\Gamma)$, we let $\langle v_1,\dots,v_m\rangle^\vee:= \{\theta\in \Theta \mid \theta(v_i)\geq 0 \text{ for }1\leq i\leq m\}$ denote the corresponding polyhedral cone. 
 
 Consider the hyperplane arrangement in $\Theta$ given by
 \begin{eqnarray}
\mc{A} & = & \left\{ \delta^\perp, (m\delta+\alpha)^\perp \mid \alpha \in \Phi, -n < m < n \right\} \nonumber \\
 & = & \{ \delta^\perp, (m\delta\pm \alpha)^\perp \mid \alpha \in \Phi^+, \ 0 \le m < n \}.\label{eqn:hyperplanes}
\end{eqnarray}
  A \emph{chamber} in $\Theta$ is the intersection with $\Theta$ of a connected component of the locus 
$$
\big(\Theta\otimes_{\mathbb{Q}} \mathbb{R}\big) \smallsetminus \bigcup_{\gamma^\perp \in \mc{A}} \gamma^\perp. 
$$
and we let $\Theta^{\reg}$ denote the union of all chambers in $\Theta$. The closure of each chamber defines a top-dimensional cone in a chamber complex in $\Theta$ determined by $\mc{A}$, and the codimension-one faces of these top-dimensional cones are called \emph{walls} in $\Theta$. Each wall is contained in a unique hyperplane from $\mc{A}$. 

\begin{example}
\label{exa:Cpmn}
 \begin{enumerate}
 \item The interior $C_-$ of the closed cone
\begin{equation}
\label{eqn:C-}
 \big\langle \delta, \pm m \delta + \alpha \mid 0 \le m < n, \;\alpha \in \Phi^+\big\rangle^\vee 
\end{equation}
 is a chamber of $\Theta$. Indeed, since $C_-$ is cut out by specifying a strict inequality for each hyperplane in $\mathcal{A}$, the claim follows provided that we show that it's non-empty. Let $h= \sum_{0\leq i\leq r} \delta_i$ be the Coxeter number of $\Phi$. If we set $\theta_i = 1$ for $i \ge 1$ and $\theta_0 = \frac{1}{2n} - h + 1$, then $\theta(m \delta) = \frac{m}{2n}$ and $\theta(\alpha) \ge 1$ for all $\alpha \in \Phi^+$. This shows that $\theta \in C_-$ as required.

 \item The interior $C_+$ of the closed cone
 \begin{equation}
 \label{eqn:C+}
 \big\langle \delta, \alpha, m \delta \pm \alpha \mid 1 \le m < n, \alpha \in \Phi^+ \big\rangle^\vee
 \end{equation}
 is also a chamber in $\Theta$, because if we set $\theta_i = 1$ for $i \ge 0$, then $\theta\in C_+$ and the statement follows by the same logic as in part (1). In fact, $C_+$ can be described more simply as the interior of the cone
 \begin{equation}
 \label{eqn:poschamber}
 \big\langle \rho_0,\rho_1,\dots,\rho_r\big\rangle^\vee.
 \end{equation}
 Indeed, since $\theta_0 = \theta(\delta - \beta)$ where $\beta$ is the highest positive root, we have that the closure of $C_+$ is contained in \eqref{eqn:poschamber}. For the opposite inclusion, $m\delta > \alpha$ for $1\leq m < n$, so each of $\delta, \alpha, m\delta\pm \alpha$ can be expressed as a positive sum of $\rho_0, \dots, \rho_r$. It follows that the inequalities defining $C_+$ in \eqref{eqn:C+} can be deduced from the inequalities $\theta_i\geq 0$ for $0\leq i\leq r$ which characterise the cone \eqref{eqn:poschamber}.
 \end{enumerate}
\end{example}

 \subsection{The Namikawa Weyl group}\label{sec:NamWeylgroup}
 We now introduce an action of a finite group on $\Theta$. For $1\leq i\leq r$, let  $s_i\colon R(\Gamma)\to R(\Gamma)$ denote reflection in the hyperplane $\{v\in R(\Gamma) \mid (v,\rho_i)_\Gamma=0\}$ orthogonal to $\rho_i$. Explicitly, for any $0\leq j\leq r$, if we write $c_{i,j}$ for the $(i,j)$-th entry of the Cartan matrix $C_{\Gamma}$, then
 \begin{equation}
 \label{eqn:sieta}
 s_i(\rho_j) = \rho_j-c_{i,j}\rho_i.
 \end{equation}
 In addition, consider the involution $s_\delta\colon R(\Gamma) \to R(\Gamma)$ defined by sending $\delta$ to $-\delta$ and fixing $\rho_1, \dots, \rho_r$; explicitly, 
 \begin{equation}
 \label{eqn:sdelta}
 s_{\delta}(\rho_j) = \left\{\begin{array}{cr}  \rho_0 - 2 \delta & \text{for }j=0, \\  
 \rho_j & \text{ for }1\leq j\leq r.
 \end{array}\right.
 \end{equation}
We caution the reader that the subspace in $R(\Gamma)$ orthogonal (with respect to $( - , - )_{\Gamma}$) to $\rho$ differs from $\rho^\perp = \{\theta\in \Theta \mid \theta(\rho)=0\}$, which is a hyperplane in the dual space $\Theta$.  
  
  We are primarily interested in the dual action on $\Theta$. For $1\leq i\leq r$, we use the same notation $s_i\colon \Theta\to\Theta$ for the linear map defined for $\theta\in \Theta$ and $0\leq j\leq r$ by setting $s_i(\theta)(\rho_j) = \theta(s_i^{-1}(\rho_j)) = \theta(s_i(\rho_j))$, where we use the fact that $s_i$ is an involution. It follows for $\theta\in \Theta$ and $1\leq i\leq r$ that $s_i(\theta)\in \Theta$ has components
	\[
	s_i(\theta)_j = \theta_j - c_{i,j}\theta_i
	\]
 for all $0\leq j\leq r$. Note that $s_i$ is reflection in the hyperplane $\rho_i^\perp$. Similarly, the dual map $s_\delta\colon \Theta\to \Theta$ sends $\theta$ to the vector $s_\delta(\theta)$ with components
 \[
 s_{\delta}(\theta)_j = \left\{\begin{array}{cr}  \theta_0 - 2 \theta(\delta) & \text{for }j=0, \\  
 \theta_j & \text{ for }1\leq j\leq r.
 \end{array}\right.
 \]
 Now define the \emph{Namikawa Weyl group} to be the group 
 \[
 W:=\langle s_\delta, s_1, \dots, s_r\rangle
 \]
 generated by these reflections of $\Theta$.

 \begin{prop}
 \label{prop:W}
 The Namikawa Weyl group $W$ satisfies the following properties:
 \begin{enumerate}
 \item[\one] it is isomorphic to $\s_2 \times W_{\Gamma}$, where $W_\Gamma$ is the Weyl group of the root system $\Phi$;
  \item[\two] the simplicial cone in $\Theta$ defined by 
 \[
 F:=  \langle \delta,\rho_1,\dots,\rho_r\rangle^\vee
 \]
 is a fundamental domain for the action of $W$ on $\Theta$; and
 \item[\three] the action of $W$ on $\Theta$ permutes the hyperplanes in $\mathcal{A}$ and hence permutes the set of chambers.
\end{enumerate}
In particular, for every chamber $C$ in $\Theta$, there exists a unique $w\in W$ such that $w(C)\subset F$.
 \end{prop}
 \begin{proof}
  For \one, change basis on $R(\Gamma)$ from $\{\rho_0,\rho_1,\dots,\rho_r\}$ to $\{\delta,\rho_1,\dots,\rho_r\}$. The vector $\delta$ is orthogonal to the hyperplane spanned by $\rho_1,\dots,\rho_r$ with respect to the symmetric bilinear form on $R(\Gamma)$, so the basis $\{\delta,\rho_1,\dots,\rho_r\}$ provides a set of simple roots for the decomposable root system of type $A_1\oplus \Phi$. Since $s_1,\dots,s_r$ generate $W_\Gamma$, it follows $s_\delta,s_1,\dots,s_r$ generate the Weyl group $\s_2 \times W_{\Gamma}$ of this root system and part \one\ follows by changing basis back to $\{\rho_0,\rho_1,\dots,\rho_r\}$. For \two, the positive orthant in the basis $\{\delta,\rho_1,\dots,\rho_r\}$ provides a fundamental domain for the action of the Weyl group $\s_2 \times W_{\Gamma}$. After changing basis back to $\{\rho_0,\rho_1,\dots,\rho_r\}$, we see that the cone $F$ is a fundamental domain for the action of $W$. For \three, the action of the generator $s_\delta$ fixes $\delta^\perp$ and exchanges $(m\delta+\alpha)^\perp$ with $(m\delta-\alpha)^\perp$ for all $0\leq m<n$ and $\alpha\in \Phi^+$, while the simple reflections $s_1, \dots, s_r$ permute the hyperplanes in $\mathcal{A}$. It follows that $W$ permutes the chambers. 
 \end{proof}

 \subsection{Counting chambers}
 The supporting hyperplanes of $F$ lie in $\mathcal{A}$, so $F\cap\Theta^\reg$ is a union of chambers. Our next goal is to count the number of chambers in $F$. We begin with a useful lemma.
 
\begin{lem}
\label{lem:Fwalls}
 The hyperplanes of $\mc{A}$ passing through the interior of $F$ are those of the form $(m\delta-\alpha)^\perp$ for $0 < m< n$ and $\alpha\in \Phi^+$.  
 \end{lem}
 \begin{proof}
 We give two proofs. For the first, we claim that $F$ is equal to the cone in $\Theta$ generated by the closures of the cones $C_\pm$ from Example~\ref{exa:Cpmn}. Indeed, the cone $F$ is generated by the vectors $f_0,\dots, f_r\in \Theta$ where $f_0:=(1,0,\dots,0)$ and, for $1\leq i\leq r$, the vector $f_i$ satisfies
\[
f_i(\rho_j):=\left\{\begin{array}{cr} 
-\dim(\rho_i) & \text{for }j=0, \\
1 & \text{ for }j=i, \\
0 & \text{otherwise}.
\end{array}\right.
\]
 We have that $f_0\in \overline{C_+}$ and $f_1,\dots, f_r\in \overline{C_-}$, so $F$ lies in the cone generated by $\overline{C_-}\cup \overline{C_+}$. For the opposite inclusion, we have $C_+\subset F$ by \eqref{eqn:poschamber}, while the inequalities $\theta(\alpha)>0$ for $\alpha\in \Phi^+$ defining $C_-$ include $\theta(\rho_i)>0$ for $1\leq i\leq r$ and hence $C_-\subset F$. This proves the claim. It follows that the walls passing through the interior of $F$ are those for which the corresponding defining inequality changes from $>$ to $<$ or vice-versa when we compare $C_\pm$. The result follows by comparing the lists from \eqref{eqn:C-} and \eqref{eqn:C+}.
 
The second approach is more explicit. The hyperplanes $\delta^\perp$ and $\alpha^\perp$ for $\alpha \in \Phi^+$ support the facets of $F$ and can be discarded. Notice that if $\theta$ is in the interior of $F$ then $\theta(\delta), \theta(\alpha) > 0$ implies that $\theta(m \delta + \alpha) > 0$, so the hyperplanes $(m \delta + \alpha)^\perp$ for $\alpha \in \Phi^+$ and $0<m<n$ do not intersect the interior of $F$. Hence it suffices to show that $(m \delta - \alpha)^\perp$ intersects the interior of $F$ for all $0 < m < n$ and $\alpha \in \Phi^+$. Let $\theta_i= 1$ for $1\leq i\leq r$. Then for any choice of $\theta_0$, the parameter $\theta=(\theta_0,\dots, \theta_r)\in \Theta$ satisfies $\theta(\gamma) = \mathrm{ht}(\gamma)>0$ for all $\gamma \in \Phi^+$. Now 
$$
\theta(m \delta - \alpha) = m \theta_0 + m \mathrm{ht}(\beta) - \mathrm{ht}(\alpha),
$$
where $\beta \in \Phi^+$ is the highest root (notice that $\mathrm{ht}(\beta) = h - 1$). Set $\theta_0 = \frac{1}{m} \mathrm{ht}(\alpha) - \mathrm{ht}(\beta)$. Then the parameter $\theta=(\theta_0,\dots, \theta_r)\in \Theta$ satisfies 
$$
\theta(\delta) = \frac{1}{m} \mathrm{ht}(\alpha) - \mathrm{ht}(\beta) + \mathrm{ht}(\beta) > 0,
$$
so it lies in the interior of $F$. By construction, we have $\theta(m \delta - \alpha) = 0$, so $\theta$ lies on the required hyperplane. 
\end{proof}

\begin{thm}
\label{thm:counting}
The cone $F$ contains precisely 
 \begin{equation}
 \label{eqn:count}
 \prod_{i=1}^{r} \frac{(n-1)h+d_i}{d_i}
 \end{equation}
 chambers, where $r$ is the rank, $h$ is the Coxeter number and $d_1,\dots, d_r$ are the degrees of the basic polynomial invariants of $W_\Gamma$.
\end{thm}  
 \begin{proof}
 Every chamber in $F$ intersects the affine hyperplane $\Lambda:=\{\theta\in \Theta \mid \theta(\delta)=1\}$ in an open region of dimension $r$, so it suffices to count the number of these regions. The hyperplanes $\rho_i^\perp$ for $1\leq i\leq r$ intersect $\Lambda$ to give a system of coordinate hyperplanes in $\Lambda\cong \Q^r$ with origin at $f_0 = (1,0,\dots,0)\in \Theta$, and $\Lambda\cap F$ is the positive orthant $\Q^r_{\geq 0}$. More generally, the intersection of $\Lambda$ with the hyperplanes from Lemma~\ref{lem:Fwalls} and the hyperplanes $\rho_i^\perp$ for $1\leq i\leq r$ defines the following collection of affine hyperplanes in $\Lambda$:
 \[
 \big\{\theta\in \Lambda \mid \theta(\beta) = m\big\} \quad \text{for }0 \leq m< n, \; \beta\in \Phi^+;
 \]
 this is the \emph{$(n-1)$-extended Catalan hyperplane arrangement of $\Phi$} from \cite{Athanasiadis04}, or one of the \emph{truncated $\Phi_\aff$-affine arrangements} from \cite{PS00}. It follows that the connected components of $\Lambda\cap F\cap\Theta^{\reg}$ are precisely the regions in the fundamental chamber of this hyperplane arrangement. To see that the number of such regions is given by formula \eqref{eqn:count}, substitute $d_i=e_i+1$ for $1\leq i\leq r$ where $e_1,\dots,e_r$ are the exponents of the Weyl group $W_\Gamma$ (see, for example, Carter~\cite[Corollary~10.2.4]{SimpleGroupsOfLieType}), and apply Athanasiadis~\cite[Corollary~1.3]{Athanasiadis04}.
 \end{proof}
 
 \begin{remark}
 \begin{enumerate}
 \item The proof of Theorem~\ref{thm:counting} shows that the unbounded regions in the $(n-1)$-extended Catalan hyperplane arrangement of $\Phi$ are precisely the intersection with the affine hyperplane $\Lambda$ of those chambers in $F$ whose closure touches the facet $\delta^\perp$ of $F$.
 \item Proposition~\ref{prop:W} and
Theorem~\ref{thm:counting} together imply that there are 
 \[
 2\vert W_\Gamma\vert \prod_{i=1}^{r} \frac{(n-1)h+d_i}{d_i}
 \]
 chambers in $\Theta$.
 \end{enumerate}
 \end{remark}
 
\begin{example}
 \label{exa:A2n=4}
 For $n=4$ and $\Phi$ of type $A_2$, Figure~\ref{fig:A2n=4} illustrates in two ways the decomposition of the cone $F$ into chambers:  Figure~\ref{fig:A2n=4a} shows all 22 regions in the fundamental chamber of the 3-extended Catalan hyperplane arrangement of $\Phi$ in the affine plane $\Lambda$ parallel to $\delta^\perp$ that was introduced in the proof of Theorem~\ref{thm:counting}; 
  \begin{figure}[!ht]
   \centering
      \subfigure[]{
      \label{fig:A2n=4a}
       \begin{tikzpicture}[xscale=1.25,yscale=1.25]
			\tikzset{>=latex}
			\draw [->] (0,0) -- (4,0);
			\draw [->] (0,0) -- (0,4);
			\draw (1,0) -- (1,3.8);
			\draw (2,0) -- (2,3.8);
			\draw (3,0) -- (3,3.8);
            \draw (0,1) -- (3.8,1);
            \draw (0,2) -- (3.8,2);
            \draw (0,3) -- (3.8,3);
            \draw (1,0) -- (0,1);
            \draw (2,0) -- (0,2);
            \draw (3,0) -- (0,3);
            \node at (0.35,0.3) {$C_+$};
            \node at (3.5,3.4) {$C_-$};
            \node at (-0.3,2) {$\rho_1^\perp$};
			\node at (2,-0.3) {$\rho_2^\perp$};
            \end{tikzpicture} 
       }
      \qquad  \qquad
      \subfigure[]{
       \label{fig:A2n=4b}
            \begin{tikzpicture}[baseline={(0,-0.3)},xscale=1.5,yscale=1.25]
			\tikzset{>=latex}
            \draw (2,0) -- (4,4);
			\draw (0,4) -- (4,4);
			\draw (0,4) -- (2,0);
            \draw (1.25,1.5) -- (4,4);
            \draw (2.75,1.5) -- (0,4);
            \draw (1.25,1.5) -- (2.75,1.5);
            \draw (10/11,24/11) -- (34/11,24/11);
            \draw (10/11,24/11) -- (4,4);
            \draw (0,4) -- (34/11,24/11);
            \draw (5/7,72/28) -- (23/7,72/28);
            \draw (5/7,72/28) -- (4,4);
            \draw (0,4) -- (23/7,72/28);
              \node at (2,1) {$C_+$};
  \node at (2,3.5) {$C_-$};
   \node at (2,4.2) {$\delta^\perp$};
  \node at (0.7,2) {$\rho_1^\perp$};
 \node at (3.3,2) {$\rho_2^\perp$};
			\end{tikzpicture}
            }
           \caption{Chamber decomposition: (a) fundamental chamber; (b) transverse slice  of $F$.}
           \label{fig:A2n=4}
  \end{figure}
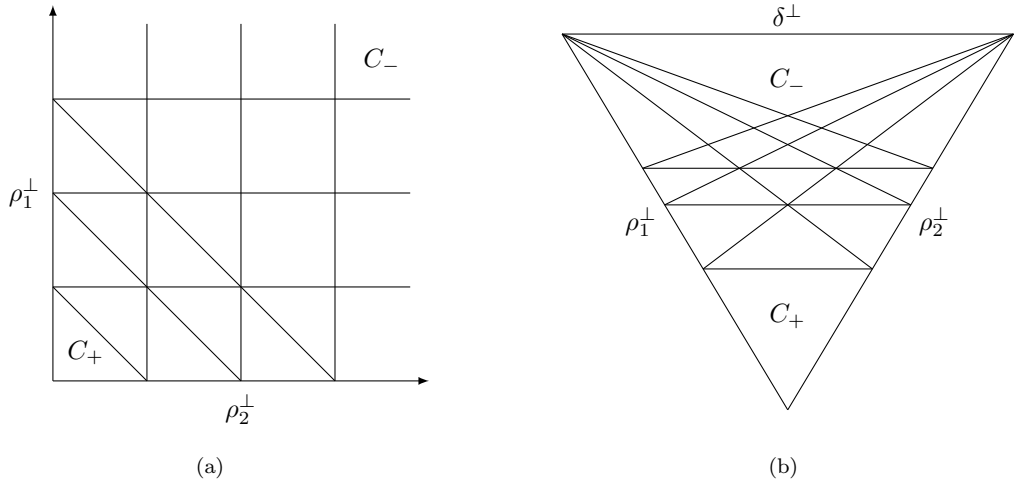
  Figure~\ref{fig:A2n=4b} shows the height-one slice of $F$ and its division into 22 chambers. In each case, we indicate where the chambers $C_\pm$ lie. The seven unbounded regions in Figure~\ref{fig:A2n=4a} correspond to the seven chambers in Figure~\ref{fig:A2n=4b} whose closure touches the facet of $F$ in $\delta^\perp$. Notice that three chambers in $F$ are not the interior of a simplicial cone.
 \end{example}

\section{\'{E}tale local normal form for quiver varieties}
Modelled on Crawley-Boevey's \'etale local description of affine quiver varieties, we give an \'etale local normal form for the morphism between quiver varieties defined by variation of GIT quotient. Pullback allows us to identify the tautological bundles on the quiver variety with the tautological bundles on the normal form. The results of this section hold for arbitrary quiver varieties. 

\subsection{Nakajima quiver varieties}
\label{sec:Nakasec}
Choose an arbitrary finite graph and let $H$ be the set of pairs consisting of an edge, together with an orientation on it. Let $\tail(a)$ and $\head(a)$ denote the tail and head respectively of the oriented edge $a \in H$. Let $a^*$ denote the same edge, but with opposite orientation. We fix an orientation of the graph, that is, a subset $\Omega \subset H$ such that $\Omega \cup \Omega^* = H$ and $\Omega \cap \Omega^* = \emptyset$. Then $\epsilon : H \rightarrow \{ \pm 1 \}$ is defined to take value $1$ on $\Omega$ and $-1$ on $\Omega^*$. Identify the vertex set of the graph with $\{ 0, 1, \ds, r \}$ for some $r\geq 1$.   

Fix collections $V_0, \ds, V_r$ and $W_0, \ds, W_r$ of finite-dimensional complex vector spaces and set 
$$
\bv = (\dim V_0, \ds, \dim V_r), \quad \bw = (\dim W_0, \ds, \dim W_r).
$$
The group $G(\bv) := \prod_{k = 0}^r \GL(V_k)$ acts naturally on the space 
$$
\mathbf{M}(\bv,\bw) := \left(\bigoplus_{a \in H} \Hom_{\C}(V_{\tail(a)}, V_{\head(a)})\right) \oplus \left(\bigoplus_{k = 0}^r \big(\Hom_{\C}(V_k,W_k) \oplus \Hom_{\C}(W_k,V_k)\big)\right).
$$
This action of $G(\bv)$ is Hamiltonian for the natural symplectic structure on $\mathbf{M}(\bv,\bw)$ and, after identifying the dual of $\mf{g}(\bv) := \mathrm{Lie} \ G(\bv)$ with $\mf{g}(\bv)$ via the trace pairing, the corresponding moment map $\gbf{\mu} \colon \mathbf{M}(\bv,\bw) \rightarrow \mf{g}(\bv)$ satisfies
$$
\gbf{\mu}(B,i,j) = \left( \sum_{\head(a) = k} \epsilon(a) B_a B_{a^*} + i_k j_k \right)_{k = 0}^r
$$
where $i_k \in \Hom_{\C}(W_k,V_k), j_k \in \Hom_{\C}(V_k,W_k)$ and $B_a \in \Hom_{\C}(V_{\tail(a)},V_{\head(a)})$. Though one can talk about arbitrary stability conditions in this context, as was done in \cite{NakajimaBranching}, it is easier in our case to apply the trick of Crawley-Boevey~\cite{CBmomap} and reduce to the case where each $W_k = 0$ by introducing a framing vertex. 

The set $H$ associated to the graph can be thought of as the arrow set of a quiver. We frame this quiver by adding an additional vertex $\infty$, as well as $\bw_i$ arrows from vertex $\infty$ to vertex $i$ and another $\bw_i$ arrows from vertex $i$ to vertex $\infty$. This framed (doubled) quiver is denoted $Q = (I,Q_1)$, where $I = \{ \infty, 0, \ds, r \}$. Each dimension vector $\bv = (\dim V_0, \ds, \dim V_r)$ for the original graph determines a dimension vector for $Q$ that we write (without bold font) as $v = (1,\dim V_0, \ds, \dim V_r)$. We may identify $\mathbf{M}(\bv,\bw)$ with the space 
\[
\Rep(Q,v):= \bigoplus_{a\in Q_1} \Hom_{\C}(\C^{\tail(a)},\C^{\head(a)})
\]
of representations of $Q$ of dimension vector $v$ in such a way that the $G(\bv)$-action on $\mathbf{M}(\bv,\bw)$ corresponds to the action of the group $G(v):=\big(\prod_{i\in I} \GL(v_i)\big)/\C^\times$ on $\Rep(Q,v)$ by conjugation and, moreover, that the above map $\gbf{\mu}$ corresponds to the moment map $\mu$ induced by this $G(v)$-action on $\Rep(Q,v)$. If we write 
\[
\Theta_{v}:= \big\{\theta \in \Hom(\Z^I,\Q) \mid  \theta(v)=0\big\},
\]
 then every character of $G(v)$ is of the form $\chi_\theta\colon G(v)\to \C^\times$ for some integer-valued $\theta\in \Theta_{v}$, where $\chi_\theta(g) = \prod_{i \in I} \det(g_i)^{\theta_i}$ for $g\in G(v)$. For $\theta\in \Theta_v$, after replacing $\theta$ by a positive multiple if necessary, the (Nakajima) \emph{quiver variety} associated to $\theta$ is the categorical quotient 
 \[
 \mathfrak{M}_\theta(\bv,\bw) :=  \mu^{-1}(0)\git_{\chi_\theta} G(v)= \mu^{-1}(0)^{\theta}/\!\!/G(v) = \Proj\bigoplus_{k\geq 0} \C[\mu^{-1}(0)]^{\chi_{k\theta}},
 \]
where $\mu^{-1}(0)^{\theta}$ denotes the locus of $\chi_\theta$-semistable points in $\mu^{-1}(0)$ and $\C[\mu^{-1}(0)]^{\chi_{k\theta}}$ is the $\chi_{k\theta}$-semi-invariant slice of the coordinate ring of the affine variety $\mu^{-1}(0)$. Note that $\C^\times$ acts on $\mathbf{M}(\bv,\bw)$ by scaling, and this action descends to an action on $\mathfrak{M}_\theta(\bv,\bw)$.
  
 For a more algebraic description of $\mathfrak{M}_\theta(\bv,\bw)$,  extend $\epsilon$ to $Q$ by setting $\epsilon(a) = 1$ if $a \colon \infty \rightarrow i$ and $\epsilon(a) = -1$ if $a \colon i \rightarrow \infty$. The \textit{preprojective algebra} $\Pi$ is the quotient of the path algebra $\C Q$ by the relation
 \begin{equation}
 \label{eqn:preprojective}
\sum_{a \in Q_1} \epsilon(a) aa^* = 0. 
\end{equation}
Given $\theta\in \Theta_v$, we say that a $\Pi$-module $M$ of dimension vector $v$ is $\theta$-semistable if $\theta(N)\geq 0$ for all submodules $N\subseteq M$, and it is $\theta$-stable if $\theta(N)>0$ for all proper nonzero submodules. A finite dimensional $\Pi$-module is said to be $\theta$-polystable if it is a direct sum of $\theta$-stable $\Pi$-modules. King~\cite{KingStable} proved that a $\Pi$-module $M$ of dimension vector $v$ is $\theta$-semistable (resp.\ $\theta$-stable) if and only if the corresponding point of $\mu^{-1}(0)$ is $\chi_\theta$-semistable (resp.\ $\chi_\theta$-stable) in the sense of GIT. In fact \cite[Propositions~3.2,5.2]{KingStable} establishes that the quiver variety $\mathfrak{M}_\theta(\bv,\bw)$ is the coarse moduli space of $\Seshadri$-equivalence classes of $\theta$-semistable $\Pi$-modules of dimension vector $v$, where the closed points of $\mathfrak{M}_\theta(\bv,\bw)$ are in bijection with the $\theta$-polystable representations of $\Pi$ of dimension $v$. The (possibly empty) open subset of $\mathfrak{M}_\theta(\bv,\bw)$ parameterizing $\theta$-stable representations will be denoted $\mathfrak{M}_\theta(\bv,\bw)^s$. 

The geometry of the quiver varieties $\mathfrak{M}_\theta(\bv,\bw)$ may change as we vary the stability parameter $\theta\in \Theta_v$. We say that $\theta\in \Theta_v$ is \emph{effective} (with respect to $v$) if there exists a $\theta$-semistable $\Pi$-module of dimension vector $v$, and it is \emph{generic} (with respect to $v$) if every $\theta$-semistable $\Pi$-module of dimension vector $v$ is $\theta$-stable. The work of Dolgachev--Hu~\cite{DolgachevHu} and Thaddeus~\cite{Thaddeus} implies that there is a wall-and-chamber structure on the cone of effective stability parameters in $\Theta_v$, where two generic parameters $\theta, \theta^\prime\in \Theta_v$ lie in the same (open) GIT chamber if and only if the notions of $\theta$-stability and $\theta^\prime$-stability coincide, in which case $\mathfrak{M}_\theta(\bv,\bw)\cong \mathfrak{M}_{\theta^\prime}(\bv,\bw)$. The GIT walls of a GIT chamber are the codimension-one faces of the closure of the chamber. 

\begin{remark}
Note that a priori, the locus of generic stability parameters could be empty. 
\end{remark}

 \subsection{A local normal form}\label{sec:localnormalform}
We begin by describing an \'etale local form for $\mf{M}_{\theta}(\bv,\bw)$ based on Luna's slice theorem, generalising \cite[Section~4]{CBnormal}. Let $A$ be the adjacency matrix of the framed (doubled) quiver $Q$, i.e.
$$
A = (a_{i,j}), \quad a_{i,j} := | \{ a \in Q_1 \ | \ \tail(a) = i, \head(a) = j \}|.
$$
Then $A$ is a symmetric matrix and we define a symmetric bilinear form on $\Z^I$ by setting 
\[
(\alpha,\beta) := \alpha^t C \beta
\]
 where $C = 2 \mathrm{Id} - A$ is the Cartan matrix of $Q$. Let $p$ be the quadratic form on $\Z^{I}$ defined by setting 
 \[
 p(\alpha) = 1 - \frac{1}{2}(\alpha,\alpha).
 \]
Let $\theta \in \Theta_v$, and choose $\theta_0 \in \Theta_v$ that lies in the boundary of the closure of the GIT chamber containing $\theta$ (the stability condition $\theta_0$ should not be confused with the component of $\theta$ corresponding to the vertex $0$; in all that follows, we never refer to the latter). Then there is a projective morphism 
\[
f \colon \mf{M}_{\theta}(\bv,\bw) \rightarrow  \mf{M}_{\theta_0}(\bv,\bw)
\]
 induced by variation of GIT quotient. In this generality, $f$ need not be birational. 
 
 Choose a closed point $x \in  \mf{M}_{\theta_0}(\bv,\bw)$ and write $M_{\infty} \oplus M_0^{\bm_0} \oplus \cdots \oplus M_k^{\bm_k}$ for the corresponding $\theta_0$-polystable representation, where $M_{\infty}$ is the unique $\theta_0$-stable summand with $(\dim M_{i})_{\infty} \neq 0$.  For $0\leq i\leq k$, let $\beta^{(i)}\in \Z^I$ denote the dimension vector of $M_i$. Following Nakajima~\cite[Section~6]{Nak1994} and Crawley-Boevey~\cite[Section~4]{CBnormal}, define the \emph{$\Ext$-graph} associated to $x$ to be the graph with vertex set $\{ 0, \ds, k \}$, and edge set comprising $p(\beta^{(i)})$ loops at vertex $i$ and $-(\beta^{(i)}, \beta^{(j)})$ edges between $i$ and $j$; the terminology is motivated by \cite[Lemma 1]{CBKleinian}. Note that the $\Ext$-graph is empty if and only if $x$ is a $\theta_0$-stable point, see Remark~\ref{rem:emptyExtgraph}. We form new dimension vectors 
\begin{equation}
\label{eqn:mandn}
\bm = (\bm_0, \ds, \bm_k) \quad\text{ and }\quad\bn = (\bn_0, \ds, \bn_k),
\end{equation}
where $\bn_i = - (\beta^{(\infty)},\beta^{(i)})$. Finally, define the exponent $\varrho\in \Hom(\Z^{k+1},\Q)$ for a rational character of $G(\bm)$ by 
\begin{equation}
\label{eqn:varrho}
\varrho(\gamma) = \theta\left( \sum_{i = 0}^k \gamma_i \beta^{(i)}\right) 
\end{equation}
for $\gamma = (\gamma_i)\in \Z^{k+1}$. The corresponding character of $G(\bm)$ is obtained from the character $\chi_\theta$ of $G(\bv)$ by restriction, i.e., $\chi_\varrho = \res^{G(\bv)}_{G(\bm)}(\chi_\theta)$. We write $f_{\varrho} \colon \mf{M}_{\varrho}(\bm,\bn) \rightarrow \mf{M}_{0}(\bm,\bn)$ for the projective morphism. 

\begin{thm}\label{thm:commdiagram22}
Let $\ell = p(\beta^{(\infty)}) \ge 0$. There exist affine open neighbourhoods $V \subset \mf{M}_{\theta_0}(\bv,\bw)$ and $V^{\prime} \subset \mf{M}_0(\bm,\bn) \times \C^{2\ell}$ of $x$ and $0$ respectively, together with a projective morphism $\xi\colon \pZ \rightarrow Z$ and a closed point $z\in Z$, forming a diagram 
\begin{equation}\label{eq:cartdiagramquiver0}
\begin{tikzcd}
f^{-1}(V) \ar[d,"f"] & \ar[l] \pZ \ar[r] \ar[d,"{\xi}"] & (f_{\varrho} \times \mathrm{id})^{-1}(V^{\prime})  \ar[d,"{f_{\varrho} \times \mathrm{id}}"] \\
V & \ar[l,"{p}"'] Z \ar[r,"q"] & V^{\prime} 
\end{tikzcd}
\end{equation}
where both squares are Cartesian, all horizontal maps are \'etale, and where $p(z)=x$, $q(z)=0$.
\end{thm}

\begin{rem}
 \label{rem:emptyExtgraph}
 If the point $x$ is $\theta_0$-stable, then the $\Ext$-graph is empty, the quiver variety $\mf{M}_0(\bm,\bn)$ is a closed point, and $\ell=p(\bv) = \frac{1}{2}\dim \mf{M}_{\theta_0}(\bv,\bw)$.
\end{rem}

Theorem \ref{thm:commdiagram22} implies:

\begin{cor}
There is an isomorphism $f^{-1}(x) \cong (f_{\varrho}\times \mathrm{id})^{-1}(0)$. In particular, 
\begin{enumerate}
\item[\one] $f^{-1}(x) \neq \emptyset$ if and only if $(f_{\varrho}\times \mathrm{id})^{-1}(0)  \neq \emptyset$; and
\item[\two] $\dim f^{-1}(x) = \dim\; (f_{\varrho}\times \mathrm{id})^{-1}(0)$.
\end{enumerate}
\end{cor}

Passing to the formal neighbourhood of $x$ in $V$, and the formal neighbourhood of $f^{-1}(x)$ in $f^{-1}(V)$, we deduce, as was explained in \cite[Section 2.1.6]{BezLosevConj}, that there is a commutative diagram of formal schemes:
\begin{equation}\label{eq:cartdiagram2}
\begin{tikzcd}
\mf{M}_{\theta}(\bv,\bw)_{f^{-1}(x)} \ar[d,"f"] \ar[r,"\sim"] & (\mf{M}_{\varrho}(\bm,\bn) \times \C^{2\ell})_{f_{\varrho}^{-1}(0) \times \{ 0 \} } \ar[d,"{f_{\varrho}\times \mathrm{id}}"] \\
\mf{M}_{\theta_0}(\bv,\bw)_x \ar[r,"q"',"\sim"] & (\mf{M}_{0}(\bm,\bn)\times \C^{2\ell})_{0}. 
\end{tikzcd}
\end{equation}

\subsection{The proof of Theorem \ref{thm:commdiagram22}}

Choose a lift $y \in \mu^{-1}(0)^{\theta_0}$ of $x$ whose $G(\bv)$-orbit is closed in $\mu^{-1}(0)^{\theta_0}$. As shown in \cite[Lemma~3.2]{BelSchQuiver}, the stabiliser subgroup $G(\bv)_y\cong G(\bm)$ is a reductive subgroup of $G(\bv)$. Since $G(\bm)$ is reductive, it is explained in \cite[\S 4]{CBnormal} that one can find a coisotropic $G(\bm)$-module complement $C$ to $\mf{g}(\bv)\cdot y$ in $\mathbf{M}(\bv,\bw)$. As in \textit{loc.\ cit.}, we let $\mu_y$ denote the composition of $\mu$ with the quotient map $\mf{g}(\bv)^* \rightarrow \mf{g}(\bm)^*$ and let $L$ be a $G(\bm)$-stable complement to $\mf{g}(\bm)$ in $\mf{g}(\bv)$. As in \textit{loc.\ cit.}, we define $\nu\colon C \rightarrow L^*$ by 
$$
\nu(c)(l) = \omega(c,ly) + \omega(c,lc) + \omega(y,lc),
$$
 where $\omega$ is the $G$-invariant symplectic form on $\mathbf{M}(\bv,\bw)$. For $c\in C$, $g\in \mf{g}(\bm)$ and $l\in L$ we calculate 
 \begin{equation}
 \label{eqn:moMapFormula}
 \mu(y+c)(g+l) = \omega\big(y+c,(g+l)(y+c)\big) = \nu(c)(l) + \mu_y(c)(g).
 \end{equation}
 The following two results are each a relative version of \cite[Theorem~3.3]{BelSchQuiver}. 

\begin{lem}\label{lem:elate1}
There exists a $G(\bm)$-saturated affine open subset $U_0$ of $0 \in C$ such that:
\begin{enumerate}
\item[\one] the map from $U := U_0 \cap \mu_y^{-1}(0) \cap \nu^{-1}(0)$ to $\mu^{-1}(0)^{\theta_0}$ sending $c$ to $y+c$ induces an \'etale morphism 
\[
p \colon G(\bv) \times_{G(\bm)} U \rightarrow \mu^{-1}(0)^{\theta_0}
\]
 whose image $V$ is a $G(\bv)$-saturated affine open subset of $\mu^{-1}(0)^{\theta_0}$; 
\item[\two] the map $p$ restricts to an \'etale map $G(\bv) \times_{G(\bm)} U^{\varrho} \rightarrow V^{\theta}$ sending the point $(1,0)$ to $y$; and
\item[\three] these maps induce a Cartesian diagram 
\begin{equation}\label{eq:Cartdi1}
\begin{tikzcd}
V^{\theta} \git \ G(\bv) \ar[d,"f"] &  U^{\varrho} \git \ G(\bm) \ar[l] \ar[d,"{\xi}"]\\
V \git \ G(\bv) &  U \git \ G(\bm) \ar[l,"p"']. 
\end{tikzcd}
\end{equation}
with both horizontal maps \'etale.
\end{enumerate}
\end{lem}

\begin{proof}
 Equation~\eqref{eqn:moMapFormula} shows that if $c\in \mu_y^{-1}(0)\cap \nu^{-1}(0)$ then $y+c\in \mu^{-1}(0)$. We now apply \cite[Equation~(4), Lemma~3.9]{BelSchQuiver} to obtain a $G(\bm)$-saturated affine open neighbourhood $U_0$ of $0$ in $C$ such that \one\ holds. 
 
 For \two, it suffices to show that $p^{-1}(V^{\theta}) = G(\bv) \times_{G(\bm)} U^{\varrho}$. To show that the left hand side is contained in the right hand side, we need only show that if $c \in U$ such that $y + c \in V^{\theta}$, then $c \in U^{\varrho}$. If $y + c \in V^{\theta}$ then there exists an $m \theta$-semi-invariant function $\gamma$ on $V$ such that $\gamma(y + c) \neq 0$. Define $h\colon C \rightarrow \C$ by setting $h(c) = \gamma(y + c)$. Then $h$ is $\varrho$-semi-invariant, and hence $c \in U^{\varrho}$. For the opposite inclusion, the tensor product of the counit and the identity map gives a $G(\bm)$-module homomorphism $\C[G(\bv) \times_{G(\bm)} U]\to\C[U]$, and Frobenius reciprocity~\cite[Proposition~3.4]{Jantzen87} gives
 \[
 \C[G(\bv) \times_{G(\bm)} U]^{\chi_\theta}\cong 
 \Hom_{G(\bv)}\big(\chi_\theta,\C[G(\bv) \times_{G(\bm)} U]\big)\cong 
 \Hom_{G(\bm)}\big(\res^{G(\bv)}_{G(\bm)} (\chi_\theta),\C[U]\big)\cong
 \C[U]^{\chi_{\varrho}}.
 \]
 If $c \in U^{\varrho}$ then choose $h \in \C[U]^{k\varrho}$ for some $k>0$ such that $h(c) \neq 0$. We define $h' \in \C[G(\bv) \times_{G(\bm)} U]^{\chi_{k\theta}}$ by $h'(g,c) = \chi_\theta(g) h(c)$. Now, 
$$
G(\bv) \times_{G(\bm)} U \cong V \times_{V/\!\!/ G(\bv)} U/\!\!/ G(\bm) 
$$
 and hence $\C[G(\bv) \times_{G(\bm)} U] \cong \C[V] \o_{\C[V]^{G(\bv)}} \C[U]^{G(\bm)}$. Thus, there exist $h_i \in \C[V]^{\chi_{k\theta}}$ and $\gamma_i \in \C[U]^{G(\bm)}$ such that $h' = \sum_i h_i \o \gamma_i$. In particular, 
$$
h(c) = h'(1,c) = \sum_i h_i(y + c) \gamma_i(\overline{c}) \neq 0
$$
implies that there is some $h_i \in \C[V]^{\chi_{k\theta}}$ such that $h_i(y + c) \neq 0$. In particular, $y + c \in V^{\theta}$, so \two\ holds. 

This also implies that that diagram 
$$
\begin{tikzcd}
V^{\theta} \ar[d] & G(\bv) \times_{G(\bm)} U^{\varrho} \ar[l] \ar[d] \\
V & G(\bv) \times_{G(\bm)} U  \ar[l]. 
\end{tikzcd}
$$
is Cartesian, with the horizontal maps being \'etale and $G(\bv)$-equivariant. Taking the GIT quotient gives the Cartesian diagram \eqref{eq:Cartdi1}, so \three\ holds as required. 
\end{proof}

 Let $\mu_{\bm}$ denote the moment map for the action of $G(\bm)$ on $\mathbf{M}(\bm,\bn)$. It is explained in \cite[\S 4]{CBnormal} that $C\cap (\mf{g}(\bv) \cdot y)^\perp$ can be identified, as a $G(\bm)$-module, with representations of a certain doubled quiver. This doubled quiver is precisely the framed doubled quiver associated to the $\Ext$-graph described in section \ref{sec:localnormalform}, except that we have neglected to include the $p(\beta^{(\infty)}) = \ell$ loops at vertex $\infty$ in our $\Ext$-graph. Since the dimension vector $m$ of the framed doubled quiver satisfies $m_{\infty} = 1$, there is a factor of $\C^{2 \ell}$ in the representation space, corresponding to the value of the endomorphisms at the loops, on which $G(\bm)$ acts trivially. That is, we can identify $C\cap (\mf{g}(\bv) \cdot y)^\perp = \mathbf{M}(\bm,\bn) \times \C^{2 \ell}$ as $G(\bm)$-modules, where $G(\bm)$ acts trivially on $\C^{2 \ell}$, in such a way that $\mu_{\bm}$ is identified with the restriction of $\mu_y$ to $C \cap (\mf{g}(\bv) \cdot y)^\perp$. 
 
\begin{lem}\label{lem:elate2}
There exists a $G(\bm)$-saturated affine open subset $U$ of $0 \in \mu_y^{-1}(0) \cap \nu^{-1}(0)$ such that:
\begin{enumerate}
\item[\one] the projection $q\colon U \rightarrow \mu_{\bm}^{-1}(0) \times \C^{2 \ell}$ is \'etale with image $W$ a $G(\bm)$-saturated affine open subset;
\item[\two] $q$ restricts to an \'etale map $U^{\varrho} \rightarrow (V^{\prime})^{\varrho}$ sending $0$ to $0$; and 
\item[\three] these maps induce a Cartesian diagram 
$$
\begin{tikzcd}
U^{\varrho} /\!\!/ G(\bm) \ar[r] \ar[d,"{\xi}"] & (V^{\prime})^{\varrho} /\!\!/ G(\bm) \ar[d,"{f_{\varrho}}"] \\
U /\!\!/ G(\bm) \ar[r,"{q}"] & V^{\prime} /\!\!/ G(\bm). 
\end{tikzcd}
$$
with both horizontal maps \'etale.
\end{enumerate}
\end{lem}

\begin{proof}
The proof of the lemma is a straight-forward (but easier) adaptation of the proof of \cite[Lemma 4.8]{CBnormal}, as in the proof of Lemma \ref{lem:elate1}. 
\end{proof}

Theorem \ref{thm:commdiagram22} follows directly from the following more precise statement. 

\begin{thm}
\label{thm:precise}
There exists an affine $G(\bm)$-variety $U$, and an affine open $G(\bv)$-stable subset $V \subset \mu^{-1}(0)^{\theta_0}$ containing $y$ and an affine open $G(\bm)$-stable subset $V^{\prime} \subset \mu_{\bm}^{-1}(0) \times \C^{2 \ell}$ containing $0$, forming a diagram 
\begin{equation}\label{eq:cartdiagram1}
\begin{tikzcd}
V^{\theta}/\!\!/G(\bv) \ar[d,"f"] & \ar[l] U^{\varrho} /\!\!/ G(\bm) \ar[r] \ar[d,"{\xi}"] & (V^{\prime})^{\varrho} /\!\!/ G(\bm) \ar[d,"{f_{\varrho}}"] \\
V/\!\!/G(\bv) & \ar[l,"{p}"] U /\!\!/ G(\bm) \ar[r,"q"'] & V^{\prime} /\!\!/ G(\bm). 
\end{tikzcd}
\end{equation}
where both squares are Cartesian, all horizontal maps are \'etale, and where the class of a closed point $u\in U$ has image under $p$ and $q$ equal to $x$ and $0$ respectively.
\end{thm}
 \begin{proof}
 The required properties of diagram \eqref{eq:cartdiagram1} follow from Lemma~\ref{lem:elate1} and Lemma~\ref{lem:elate2} since, taking their intersection if necessary, we may assume that the two affine sets $U$ of the lemmata are equal. Note that the image of the class of $u:=0\in U$ under $p$ and $q$ is the class of $y$ and the class of $0$ respectively.
\end{proof}

\subsection{Tautological bundles}
\label{sec:tautbundles}
 Tautological bundles play a key role in understanding the birational transformations that occur as one crosses the walls in the space $\Theta_v$ of GIT stability conditions. We now describe what happens to the tautological bundles under the morphisms of Theorem \ref{thm:commdiagram22}.

  Recall that $Q$ is a doubled quiver with framing vertex $\infty$, and $v$ denotes the dimension vector for $Q$ associated to a dimension vector $\bv$ for $Q \smallsetminus \{ \infty \}$. Since $v$ is primitive, King~\cite[Proposition~5.3]{KingStable} proves that for generic $\theta\in \Theta_v$, the quiver variety $\mathfrak{M}_\theta(\bv,\bw)=\mu^{-1}(0)^{\theta} \git \  {G(v)}$ is the fine moduli space of isomorphism classes of $\theta$-stable $\Pi$-modules of dimension vector $v$. In this case, the universal family on $\mf{M}_\theta(\bv,\bw)$ is a tautological locally-free sheaf
 \[
 \mathcal{R}:=\mathcal{R}_\theta(\bv,\bw) \cong \bigoplus_{i\in I} \mathcal{R}_i
 \]
 where $\rank(\mathcal{R}_i)=v_i$ for $i\in I$, together with a $\C$-algebra homomorphism $\Pi\to \End(\mathcal{R})$. Explicitly, 
for $i \in I$ we write $F_i$ for the representation of $G(v)$ obtained by pulling back the vectorial representation from the $i$th factor of $G(v)$. Then $\mathcal{R}_i$ is the locally-free sheaf associated to the vector bundle 
$$
\mu^{-1}(0)^{\theta} \times_{G(v)} F_i \rightarrow \mu^{-1}(0)^{\theta} \git \  {G(v)} = \mf{M}_{\theta}(\bv,\bw).
$$
 When we wish to emphasise the dependence of $\mc{R}_i$ on the dimension vectors, we write $\mc{R}_i(\bv,\bw)$. Since $\mathcal{R}$ is defined only up to tensor product by an invertible sheaf, we normalise by fixing $\mathcal{R}_\infty$ to be the trivial bundle.

As in section~\ref{sec:localnormalform}, choosing a closed point $x\in\mf{M}_{\theta_0}(\bv,\bw)$, where $\theta_0$ lies in the closure of the GIT chamber containing $\theta$, determines $k\geq 0$ and dimension vectors $\beta^{(0)}, \dots, \beta^{(k)}\in \Z^{I}$, dimension vectors $\bm, \bn$ for the Ext-graph of $x$, and a stability condition $\varrho\in \Hom(\Z^{k+1},\Q)$, which determine the quiver variety $\mf{M}_\varrho(\bm,\bn)$. Recall now the statement of Theorem~\ref{thm:commdiagram22}, and specifically, diagram \eqref{eq:cartdiagramquiver0}: 
\begin{equation}
\begin{tikzcd}
f^{-1}(V) \ar[d,"f"] & \ar[l,"{p_{\theta}}"'] Z' \ar[r,"{q_{\varrho}}"] \ar[d,"{\xi}"] & (f_{\varrho} \times \mathrm{id})^{-1}(V^{\prime}) \ar[d,"{f_{\varrho} \times \mathrm{id}}"] \\
V & \ar[l,"{p}"'] Z \ar[r,"q"] & V^{\prime}
\end{tikzcd}
\end{equation}

\begin{prop}
\label{prop:linebundles}
 The quiver variety $\mf{M}_\varrho(\bm,\bn)$ carries a tautological locally-free sheaf $\bigoplus_{j} \mathcal{R}_j(\bm,\bn)$, where $j$ ranges over the set $\{\infty,0,\dots,k\}$. Moreover, for each $i\in I\setminus \{\infty\}$, there is an isomorphism of bundles on $Z'$ given by
$$
p_{\theta}^* \ \mc{R}_i(\bv,\bw) \cong \bigoplus_{j = 0}^k q_{\varrho}^*\ \mc{R}_j(\bm,\bn)^{\oplus \beta_i^{(j)}}.
$$
\end{prop}
\begin{proof}
The vector $\bm$ determines a dimension vector $m$ for the framed (doubled) quiver satisfying $m_\infty=1$, so $m$ is primitive. In light of \cite[Proposition~5.3]{KingStable}, it remains to show that $\varrho$ is generic in order to prove the first statement. 
 
 We claim that if $\theta$ is generic with respect to $v$ then $\varrho$ is generic with respect to $m$. Our argument is based on the proof of \cite[Proposition 1.1]{LeBSimple}. Assume that $\varrho$ is not generic for $m$. Then the locus of strictly polystable representations in $\mf{M}_\varrho(\bm,\bn)$ is non-empty. The scaling action of $\Cs$ on $\mf{M}_0(\bm,\bn)$ defined in section \ref{sec:Nakasec} lifts to $\mf{M}_\varrho(\bm,\bn)$, and the locus of strictly polystable representations is both closed and $\Cs$-stable. Hence, there exists a strictly polystable point $r$ in $f_{\varrho}^{-1}(0)$. Choose a lift $(r',0)$ of $(r,0)$ in $W^{\varrho} \subseteq \mu^{-1}_{\mathbf{m}}(0) \times \C^{2 \ell}$ whose $G(\mathbf{m})$-orbit is closed. The fact that the diagram (\ref{eq:cartdiagram1}) is Cartesian means that there is a point $z' \in \xi^{-1}(z)$ mapping to $(r,0)$. Let $x'$ be the image of this point in $V^{\theta}/\!\!/G(\bv)$. If $y'$ is a lift of $x'$ in $V^{\theta}$, whose $G(\bv)$-orbit is closed then the final statement of \cite[Theoreme du slice \'etale]{Luna} implies that $G(\bv)_{y'} \cong G(\mathbf{m})_{(r',0)} = G(\mathbf{m})_{r'}$. The fact that $r$ is strictly polystable means that $G(\mathbf{m})_{r'}$, and hence $G(\bv)_{y'}$, is non-trivial. This in turn implies that $x'$ is strictly polystable, contradicting the assumption that $\theta$ is generic for $v$. 

For the second statement, the locally-free sheaf $p_{\theta}^* \mc{R}_i(\bv,\bw)$ is the sheaf of sections of the map 
$$
(V^{\theta} \times_{G(\bv)} F_i) \times_{V^{\theta} \git \ G(\bv)} U^{\varrho} \git \ G(\bm) \cong (G(\bv) \times_{G(\bm)} U^{\varrho}) \times_{G(\bv)} F_i \cong U^{\varrho} \times_{G(\bm)} F_i \longrightarrow U^{\varrho} \git\ G(\bm);
$$
 here the first isomorphism follows from the proof of Lemma \ref{lem:elate1}, and the second is a consequence of Luna's slice theorem \cite{Luna}. Similarly, $q_{\varrho}^* \mc{R}_j(\bm,\bn)$ corresponds to sections of $U^{\varrho} \times_{G(\bm)} F_j(\bm,\bn) \rightarrow U^{\varrho} \git\ G(\bm,\bn)$. Thus, the result follows from the fact that 
$$
F_i |_{G(\bm)} \cong \bigoplus_{j = 0}^k F_j(\bm,\bn)^{\oplus \beta_i^{(j)}}.
$$
This latter decomposition is simply the fact that $x$ corresponds to the $\theta_0$-polystable representation
$$
M_{\infty} \oplus M_0^{\oplus \bm_0} \oplus \cdots \oplus M_k^{\oplus \bm_{k}} =  M_{\infty} \oplus \big(M_0 \o F_1(\bm,\bn)\big) \oplus \cdots \oplus \big(M_k \o F_k(\bm,\bn)\big). 
$$
This completes the proof.
\end{proof}

\subsection{A stratification}\label{sec:quiverstratification}

There are two natural ways of defining a finite stratification of the quiver variety. The first is Luna's stratification, coming from the fact that it is a GIT quotient by a reductive group. The second is the stratification into symplectic leaves, which is finite since the quiver variety has symplectic singularities \cite[Theorem 1.2]{BelSchQuiver}. It is known that these two stratifications agree \cite[Proposition 3.6]{BelSchQuiver}. In this section, we recall the definition of these strata. In order to do so, we first recall Crawley-Boevey's canonical decomposition of a dimension vector. 

 As explained in \cite{KacBook}, associated to the Cartan matrix $C$ of the framed (doubled) quiver $Q=(I,Q_1)$ is a root system $R \subset \Z^I$, with positive roots $R^+ = R \cap \Z^I_{\ge 0}$. For $\theta\in \Hom(\Z^{I},\Q)$,  set $R_{\theta}^+  = \{ \alpha \in R^+ \mid \theta(\alpha) = 0 \}$ and define
\begin{equation}\label{eq:Sigmathetadefn}
\Sigma_{\theta} = \left\{ \alpha \in R^+_{\theta} \ \Big| \ p(\alpha) > \sum_{i = 0}^k p\left(\beta^{(i)}\right), \ \forall \ \alpha = \beta^{(0)} + \cdots + \beta^{(k)}, k > 0, \beta^{(i)} \in R^+_{\theta} \right\}.
\end{equation}
In general, it is very difficult to compute $\Sigma_{\theta}$, but notice that if $\theta(\beta) \neq 0$ for all roots $\beta < \alpha$, then $\alpha \in \Sigma_{\theta}$.

\begin{thm}
\label{thm:canondecomp}
\begin{enumerate}
\item[\one] For $v\in \N^{I}$, there exists a $\theta$-stable $\Pi$-module of dimension vector $v$ iff $v \in \Sigma_{\theta}$.
\item[\two] Let $\N R_{\theta}^+$ denote the subsemigroup of $\N^{I}$ generated by $R_{\theta}^+$. Each $v \in \N R_{\theta}^+$ admits a decomposition
 \begin{equation}
 \label{eqn:canonical}
v = \gamma^{(0)} + \cdots + \gamma^{(\ell)},
\end{equation}
where $\gamma^{(i)} \in \Sigma_{\theta}$ for $0\leq i\leq \ell$, such that any other decomposition of $v$ as a sum of elements from $\Sigma_\theta$ is a refinement of the sum \eqref{eqn:canonical}; this is the \emph{canonical decomposition of $v$ with respect to $\theta$}.
\item[\three] The canonical decomposition (\ref{eqn:canonical}) is characterised by the fact that
$$
\sum_{i = 0}^{\ell} p\left(\gamma^{(i)}\right) > \sum_{j = 0}^{k} p\left(\beta^{(j)}\right)
$$
for any other decomposition $v = \beta^{(0)} + \cdots + \beta^{(k)}$ of $v$ with each $\beta^{(i)} \in \Sigma_{\theta}$. 
 \item[\four] For $v\in \N R_{\theta}^+$ as in \eqref{eqn:canonical}, let $M$ be a generic $\theta$-polystable $\Pi$-module of dimension $v$ with decomposition $M=M_0^{\oplus m_0}\oplus \dots \oplus M_{\ell}^{\oplus m_{\ell}}$ into pairwise non-isomorphic $\theta$-stable representations. Then, grouping like terms $v = m_1 \xi^{(1)} + \cdots + m_{\ell} \xi^{(\ell)}$ in the canonical decomposition \eqref{eqn:canonical}, we have $\dim M_i = \xi^{(i)}$. 
\end{enumerate}
\end{thm}
\begin{proof}
This is due to Crawley-Boevey~\cite{CBmomap,CBdecomp}, but see Bellamy--Schedler~\cite{BelSchQuiver} in this generality. In particular, property \three\ can be deduced from \cite[Corollary 1.4]{CBdecomp}.  
\end{proof}

The strata of $\mf{M}_{\theta}(\bv,\bw)$ are labelled by the ``representation types'' of $v$, which we now recall. A \textit{representation type} $\tau$ of $v$ is a tuple
$$
\tau = (n_0,\beta^{(0)}; \ds ; n_{k}, \beta^{(k)})
$$
where $\beta^{(i)} \in \Sigma_{\theta}$, $n_i \in \Z_{> 0}$ and $\sum_{i = 0}^k n_i \beta^{(i)} = v$. The stratum labelled by the representation type $\tau$ is:
$$
\mf{M}_{\theta}(\bv,\bw)_{\tau} := \left\{ M_0^{\oplus n_0} \oplus \cdots \oplus M_{k}^{\oplus n_{k}} \in \mf{M}_{\theta_0}(\bv,\bw) \ | \ M_i \textrm{ is $\theta$-stable and $\dim M_i = \beta^{(i)}$} \right\}.
$$

\begin{remark}\label{rem:stratapropertydecomp}
By \cite[Corollary 3.25]{BelSchQuiver}, each stratum is a connected locally-closed, smooth subvariety of $\mf{M}_{\theta}(\bv,\bw)$. Since the dimension of the locus of $\theta$-stable points in $\mf{M}_{\theta}(\bv,\bw)$ equals $2 p(v)$, Theorem~\ref{thm:canondecomp}\one\ implies that the stratum $\mf{M}_{\theta}(\bv,\bw)_{\tau}$ is non-empty if and only if $\beta^{(i)} \neq \beta^{(j)}$ when $\beta^{(i)}$ and $\beta^{(j)}$ are real roots. The stratification is finite, and 
$$
\dim \mf{M}_{\theta}(\bv,\bw)_{\tau} = 2 \sum_{i = 0}^{k} p\left(\beta^{(i)}\right).
$$
See \cite[\S 3.5]{BelSchQuiver} and references therein for more information.
\end{remark}

\section{Quiver varieties for the framed McKay quiver}
We now specialise to the case where the quiver varieties $\mf{M}_\theta(\bv,\bw)$ are constructed from the affine Dynkin diagram associated to $\Gamma$ by the McKay correspondence.

\subsection{Canonical decomposition of roots}
 Once and for all, we fix the graph to be the McKay graph of $\Gamma\subset \SL(2,\C)$ and fix dimension vectors $\bv:= n \delta$ and $\bw = \rho_0$, so that the corresponding framed, doubled quiver $Q$ is the framed McKay quiver, with vertex set $I = \{\rho_\infty,\rho_0,\rho_1,\dots,\rho_r\}$. The dimension vector for representations of $Q$ is
 \[
 v:= \rho_{\infty} + n \delta\in \Z^{I}.
 \]
 Define
 \[
\Theta_v:= \{\theta\in \Hom(\Z^{I},\Q) \mid \theta(v)=0\},
\]
 and write elements of $\Theta_{v}$ as $\theta=(\theta_\infty,\theta_0,\dots,\theta_r)$ where $\theta_i:=\theta(\rho_i)$ for $i\in \{\infty,0,1,\dots,r\}$. For $\theta \in \Theta_v$, set
 \begin{equation}
 \label{eqn:v}
 \mathfrak{M}_\theta:= \mathfrak{M}_\theta(\bv,\bw) = \mf{M}_{\theta}(Q,v).
 \end{equation}
As explained in section \ref{sec:quiverstratification}, associated to the quiver $Q$ is the root system $R \subset \Z^I$, and $R^+=\N^{I}\cap R$ the set of positive roots. We can recover the affine root system $\Phi_\aff$ associated to $\Gamma$ in section \ref{sec:chamberdecomp} as follows.

 \begin{lem}
 \label{lem:recoverPhi}
We have $\Phi_\aff = \{ \alpha = (\alpha_i)_{i\in I} \in R \mid \alpha_{\infty} = 0 \}$
\end{lem}
\begin{proof}
 The fact that $\Phi_\aff$ is contained in the right-hand side follows from the fact that the adjacency matrix $A_\Gamma$ of the McKay graph is obtained by deleting the row and column indexed by $\rho_\infty$ from the adjacency matrix $A$ for $Q$. For the opposite inclusion, suppose $\alpha=(\alpha_\infty,\alpha_0,\dots,\alpha_r)\in R^+$ satisfies $\alpha_\infty=0$. Since $\alpha$ is a positive root, we have $(\alpha,\alpha)\leq 2$. If we set $\alpha^\prime:=(\alpha_0,\dots,\alpha_r)\in R(\Gamma)$, then $\alpha_\infty=0$ gives $(\alpha^\prime,\alpha^\prime)_\Gamma=(\alpha,\alpha)\leq 2$ and hence the fact that $A_{\Gamma}$ is positive semi-definite implies, by \cite[Proposition 5.10]{KacBook}, that $\alpha\in \Phi_\aff^+$. It follows that the set $\{\alpha\in R \mid \alpha_\infty=0\}$ is contained in $\Phi_\aff$.
\end{proof}

 Lemma~\ref{lem:recoverPhi} enables us to identify the root lattice $\Z^{I}$ with the lattice $\Z\oplus R(\Gamma)$, with the standard basis corresponding to $\{\rho_\infty, \rho_0,\rho_1,\dots,\rho_r\}$.

 Our next goal is to prove the following result about the canonical decomposition of $v$ (see Theorem~\ref{thm:canondecomp} for the definition).
 
\begin{prop}
\label{prop:canonicaldecomp}
Let $\theta\in \Theta_{v}$. Then $v\in R_\theta^+$. Moreover:
\begin{enumerate}
\item[\one] if $\theta_\infty \neq 0$ then the canonical decomposition of $v$ with respect to $\theta$ is $v$; and
\item[\two] if $\theta_\infty = 0$ then the canonical decomposition is $v = \rho_{\infty} + \delta + \cdots + \delta$, where $\delta$ appears $n$ times here. 
\end{enumerate}
\end{prop}
 
 The proof requires two preliminary results, the first of which is an application of the Frenkel--Kac theorem.

\begin{lem}
\label{lem:FrenkelKac}
If $\gamma=(\gamma_i)_{i\in I} \in R^+$ satisfies $\gamma_{\infty} = 1$, then there exists $\nu \in \bigoplus_{i\geq 1} \Z\rho_i$ such that
\begin{equation}\label{eq:wFrenkelKac}
\gamma = \rho_{\infty} + m \delta + \frac{1}{2} (\nu,\nu) \delta - \nu,
\end{equation}
 where $m=p(\gamma)$. Conversely, any vector of this form lies in $R^+$. 
\end{lem}
\begin{proof}
Let $V(\omega_0)$ be the vacuum module (of level one) for the Kac--Moody algebra with root system $R$. The Frenkel-Kac Theorem \cite[Lemma 12.6]{KacBook} says that $w$ is a weight of $V(\omega_0)$ if and only if 
$$
w = \omega_0 - m \delta - \frac{1}{2} (\nu,\nu) \delta + \nu,
$$
 for some $\nu \in \bigoplus_{i\geq 1} \Z\rho_i$ and $m \ge 0$. Then, the statement follows from \cite[Lemma 2.14]{NakajimaBranching}.
\end{proof}

\begin{lem}
\label{lem:vimroot}
 For $\theta\in \Theta_{v}$, we have $v\in R_\theta^+$ and $p(v) = n$. In particular, if $n > 1$ then $v$ is an anisotropic root. 
\end{lem}
\begin{proof}
Lemma~\ref{lem:FrenkelKac} implies that $v\in R_\theta^+$. For $i > 0$, we have $(\rho_i,\rho_{\infty} + n \delta) = 0$. We compute $(\rho_0, \rho_{\infty} + n \delta) = (\rho_0, \rho_{\infty}) = -1$ and $(\rho_{\infty}, \rho_{\infty} + n \delta) = 2 - n$, giving $p(v) = n$. Therefore, if $n > 1$, then $v$ belongs to the fundamental domain of $Q$ and $p(v) > 1$. This means that $v$ is an anisotropic root. 
\end{proof}

\begin{proof}[Proof of Proposition~\ref{prop:canonicaldecomp}]
 Lemma~\ref{lem:vimroot} gives $v \in R_{\theta}^+$, so we obtain a canonical decomposition 
 \[
v= \gamma^{(0)} + \cdots + \gamma^{(\ell)}
 \]
 by Theorem~\ref{thm:canondecomp}. We may assume that $\gamma^{(0)}_{\infty} = 1$.   
 
  Suppose first that $\theta_\infty=0$, so that $\theta(\delta) = 0$. Then the fact that $p(s \delta) = 1$ for all $s \ge 1$ implies that the only multiple of $\delta$ in $\Sigma_{\theta}$ is $\delta$ itself. Similarly, the decomposition 
  $$
  (\rho_{\infty}  + m \delta) = \rho_{\infty} + \delta + \cdots + \delta
  $$
  implies that $\rho_{\infty}  + m \delta \notin \Sigma_{\theta}$ for any $m > 0$ because $m = p(\rho_{\infty}  + m \delta) = p(\rho_{\infty}) + p(\delta) + \cdots + p(\delta)$ would contradict the inequality in (\ref{eq:Sigmathetadefn}). In particular, $v \notin \Sigma_{\theta}$. Therefore the canonical decomposition of $v$ has $\ell \ge 1$. Let $\gamma^{(1)}, \ds, \gamma^{(k)}$ be all roots in the expression that equal $\delta$ (for some $1 \le k \le n$). By Lemma \ref{lem:FrenkelKac}, the root $\gamma^{(0)}$ is of the form \eqref{eq:wFrenkelKac} for some $m \le n - \ell$ and $\nu \in \oplus_{i \ge 1} \Z \rho_i$. Each $\gamma^{(i)}$ for $i > k$ is a real root in $(\Phi_\aff)^+_{\theta} := \{ v \in \Phi_\aff^+ \mid \theta(v) = 0 \}$. In particular, $\sum_i p\left(\gamma^{(i)}\right) = m + k$. By definition of $\Sigma_{\theta}$ and the characterisation of canonical decomposition given in Theorem~\ref{thm:canondecomp}\three, we must have 
$$
n = p(v) \le \sum_i p\left(\gamma^{(i)}\right) = m + k;
$$
otherwise $v \in \Sigma_{\theta}$, which we have shown is not true. Therefore, we deduce that $n = m + k$. This implies that $(\nu,\nu) = 0$. Since $( - , - )$ is positive definite on $\bigoplus_{i \ge 1} \Z \rho_i$, this means that $\nu = 0$, forcing $\gamma^{(0)} = \rho_{\infty} + m \delta$. But we have shown above that this in turn forces $m = 0$ because $\gamma^{(0)} \in \Sigma_{\theta}$. Hence $k = n$, implying that the canonical decomposition of $v$ is $\rho_{\infty} + \delta + \cdots + \delta$. 
  
If $\theta_{\infty} \neq 0$ then none of the roots $\gamma^{(i)}$ is a multiple of $\delta$. Hence every $\gamma^{(i)}$, for $i > 0$, is a real root in $(\Phi_\aff)^+_{\theta}$. Again assuming that $\gamma^{(0)}$ has the form given in \eqref{eq:wFrenkelKac} for some $m \le n$, this implies by Theorem~\ref{thm:canondecomp}\three \ that $\sum_i p\left(\gamma^{(i)}\right) = p\left(\gamma^{(0)}\right) = m \ge n$. Therefore, we must have $m = n$, and $(\nu,\nu) = 0$. Once again, since $( - , - )$ is positive definite on $\bigoplus_{i \ge 1} \Z \rho_i$, this means that $\nu = 0$, and hence $v \in \Sigma_{\theta}$ has trivial canonical decomposition. 
\end{proof}

\subsection{Characterising smooth quiver varieties}
The following description of the affine quotient corresponding to the stability parameter $\theta=0$ is well-known.

\begin{lem}
\label{lem:singularity}
There is an isomorphism of algebraic varieties $\mathfrak{M}_0\cong \C^{2n} / \Gamma_n$ that is also an isomorphism of Poisson varieties (up to rescaling the bracket by some $t \in \C^{\times}$). 
\end{lem}
\begin{proof}
  This follows from results of Crawley-Boevey~\cite{CBdecomp}; see also Kuznetsov~\cite[Proposition~33]{Kuznetsov07}. Indeed, Proposition~\ref{prop:canonicaldecomp} shows that the canonical decomposition of $v$ in $\Sigma_0$ is $\rho_{\infty} + \delta + \cdots + \delta$, where $\delta$ appears $n$ times in the sum. If $\mf{M}^{\mathrm{pt}}_0(\bv,\bw)$ denotes a quiver variety associated to the graph with a single vertex and no edges, and $\mf{M}^{\mathrm{Mc}}_0(\bv,\bw)$ a quiver variety associated to the unframed McKay quiver, then Crawley-Boevey's canonical decomposition (see also \cite[Lemma 9.2]{CBmomap}) implies that 
$$
\mathfrak{M}_0 \cong \mf{M}^{\mathrm{pt}}_0(1,0) \times  \Sym^n \left( \mf{M}^{\mathrm{Mc}}_0(\delta,0)\right) 
$$
because $\rho_{\infty}$ is real and hence $p(\rho_{\infty}) = 0$. The quiver variety $\mf{M}^{\mathrm{pt}}_0(1,0)$ is of course just a point, and the isomorphism $\mf{M}^{\mathrm{Mc}}_0(\delta,0) \cong \C^2 / \Gamma$ is due to Kronheimer~\cite{Kronheimer}. The description of $\mathfrak{M}_0$ as an algebraic variety follows from the isomorphism $\Sym^n(\C^2 / \Gamma) \cong \C^{2n} / \Gamma_n$. For the final statement, both varieties have a natural Poisson structure making them symplectic varieties, i.e.\ they have a Poisson bracket that is generically non-degenerate. Moreover, in both cases this bracket is homogeneous of degree $-2$. Any such structure on $\C^{2n} / \Gamma_n$ is unique up to rescaling \cite[Lemma 2.23]{EG}.  
\end{proof}

To state the main result of this section, identify the space $\Theta_{v}$ of GIT stability parameters with $\Theta$ via the projection away from the $\theta_\infty$ coordinate. We obtain from \eqref{eqn:hyperplanes} a hyperplane arrangement in $\Theta_{v}$ given by
\begin{equation}
    \label{eqn:Av}
\mc{A}_{v} =  \{ \delta^\perp, (m\delta\pm \alpha)^\perp \mid \alpha \in \Phi^+, \ 0 \le m < n \},
\end{equation}
where now each $\gamma\in \Z^I = \Z\oplus R(\Gamma)$ defines the dual hyperplane $\gamma^\perp:= \{\theta\in \Theta_{v} \mid \theta(\gamma)=0\}$. As before, a \emph{chamber} in $\Theta_{v}$ is the intersection with $\Theta_v$ of a connected component of the locus 
 \begin{equation}
 \label{eqn:Thetav}
\big(\Theta_{v}\otimes_{\mathbb{Q}} \mathbb{R}\big) \smallsetminus \bigcup_{\gamma^\perp \in \mc{A}_{v}} \gamma^{\perp}, 
 \end{equation}
and we let $\Theta_v^{\reg}$ denote the union of all chambers in $\Theta_v$. A \emph{wall} is a codimension-one face of the closure of a chamber. A vector $w \in \Sigma_{\theta}$ is said to be \textit{minimal} (with respect to $\theta$) if it does not admit a proper decomposition $w = \beta^{(0)} + \cdots + \beta^{(k)}$ with $k > 0$ and $\beta^{(i)} \in \Sigma_{\theta}$.
 
\begin{thm}
\label{thm:smoothmoduli}
 For $n>1$, the following conditions on a stability parameter $\theta\in \Theta_{v}$ are equivalent:
\begin{enumerate}
\item[\one] the variety $\mathfrak{M}_\theta$ is smooth;
\item[\two] the canonical decomposition of $v$ with respect to $\theta$ is of the form $\sigma^{(0)} + \cdots + \sigma^{(\ell)}$ where each $\sigma^{(i)}\in \Sigma_\theta$ is minimal and a given imaginary root can appear at most once as a summand;
 \item[\three] the parameter $\theta$ lies in $\Theta_{v}^{\reg}$;
\item[\four] the parameter $\theta$ is generic, i.e.\ every $\theta$-semistable $\Pi$-module of dimension vector $v$ is $\theta$-stable.
\end{enumerate}
\end{thm}
\begin{proof}
 Statements \one\ and \two\ are equivalent by \cite[Corollary~1.17]{BelSchQuiver}. 
  To show that \two\ and \three\ are equivalent, assume first that $\theta \in \Theta^{\reg}_{v}$. Since $\delta^\perp = \rho_{\infty}^\perp = \{\theta \mid \theta_{\infty}=0\}$ is a wall, the canonical decomposition of $v$ is $v$ by Proposition~\ref{prop:canonicaldecomp}. Assume that $v$ is not minimal, i.e.\ it admits a proper decomposition $v = \beta^{(0)} + \cdots + \beta^{(k)}$ with $k > 0$ and $\beta^{(i)} \in \Sigma_{\theta}$. Then, without loss of generality $\beta^{(0)}_{\infty} = 1$.  Hence $\beta^{(i)} \in \Phi_\aff^+$ for $i > 0$. In particular, there exists some $\beta \in E := \{ \beta \in \Phi_\aff^+ \mid 0 < \beta \le n \delta \}$ such that $\theta(\beta) = 0$; explicitly,  
$$
E = \{ m \delta + \alpha \ | \ 0 \le m < n, \ \alpha \in \Phi^{+} \} \cup \{ m \delta - \alpha \ | \ 0 < m \le n , \ \alpha \in \Phi^+ \} \cup \{ m \delta \ | \ 1 \le m \le n \}.
$$
The equations $\theta(\beta) = 0$ define the hyperplanes in $\mc{A}_{v}$, contradicting the assumption $\theta \in \Theta^{\reg}_{v}$, except in the case $n \delta - \alpha\in E$ for $\alpha \in \Phi^+$. Suppose $v$ admits a decomposition with some $\beta^{(i)} = n \delta - \alpha$, for $\alpha \in \Phi^+$. Then $\rho_{\infty} + \alpha \in \N \Sigma_{\theta}$, and hence $\rho_{\infty} \in \Sigma_{\theta}$ because the support of any root is connected; see for instance \cite[\S 5.3]{KacBook}. This implies that $\theta_{\infty} = 0$, so $\theta\in \delta^\perp$ which contradicts $\theta \in \Theta^{\reg}_{v}$. Therefore $v$ is minimal, so statement \two\ holds. Conversely, suppose that $\theta \not\in \Theta^{\reg}_{v}$, i.e.\ $\theta$ lies in a hyperplane from $\mc{A}_{v}$. First, if $\theta\in \delta^\perp = \rho_\infty^\perp$, then $\theta_{\infty} = 0$, so the canonical decomposition of $v$ is $\rho_{\infty}+\delta + \cdots + \delta$ by Proposition~\ref{prop:canonicaldecomp}. Since $\delta$ is imaginary and $n > 1$, condition \two\ fails to hold. Next, if $\theta \in (m\delta+\alpha)^\perp$ satisfies $\theta_{\infty} \neq 0$ for some $-n< m < n$ and $\alpha \in \Phi$, then the canonical decomposition of $v$ is $v$ by Proposition~\ref{prop:canonicaldecomp}. We claim that $v \in \Sigma_{\theta}$ is not minimal. If $m > 0$ then set $\gamma := m \delta + \alpha \in \Phi_\aff^+$, and if $m < 0$ then set $\gamma := - (m \delta + \alpha) \in \Phi_\aff^+$. Either way, $v - \gamma$ can be expressed as a vector of the form \eqref{eq:wFrenkelKac} and hence $v-\gamma \in R^+$. In fact $v-\gamma \in R_\theta^+$ because $\theta \in (m\delta+\alpha)^\perp=\gamma^\perp$ gives $\theta(v-\gamma)= \pm\theta(m\delta+\alpha) = 0$. Then both $\gamma$ and $v-\gamma$ admit a canonical decomposition by Theorem~\ref{thm:canondecomp}\two, so $v$ admits a proper decomposition. This shows that \two\ fails to hold, so \two\ and \three\ are equivalent.

 For the equivalence of \three\ and \four, the previous paragraph shows that if $\theta$ lies on any of the hyperplanes in $\mc{A}_{v}$, then either $v$ admits a canonical decomposition with more than one term (i.e., the generic polystable $\Pi$-module of dimension vector $v$ is not $\theta$-stable), or $v$ admits a proper decomposition. Grouping like terms, such a decomposition defines a representation type $\tau$ consisting of more than one term. This implies that there are $\theta$-semistable representations of dimension $v$ that are not $\theta$-stable. Therefore \four\ implies \three. Conversely, if $\theta$ does not lie on a hyperplane in $\mc{A}_{v}$, then the previous paragraph shows that the only decomposition of $v$ as a sum of elements from $\Sigma_{\theta}$ is $v$ itself. Therefore every $\theta$-semistable representation is $\theta$-stable.
\end{proof}

\begin{rem}
\label{rem:transfer}
\begin{enumerate}
\item Theorem~\ref{thm:smoothmoduli} implies that the wall-and-chamber structure on $\Theta_{v}$ determined by $\mc{A}_{v}$ coincides with the wall-and-chamber structure arising from the GIT construction of the quiver varieties $\mathfrak{M}_\theta$ for $\theta\in \Theta_{v}$. In particular, all results of Section~\ref{sec:combinatorics} hold for the GIT decomposition of $\Theta_{v}$, so the walls and chambers introduced there are GIT walls and GIT chambers.
\item It follows from Proposition~\ref{prop:canonicaldecomp} that whenever the equivalent conditions of Theorem~\ref{thm:smoothmoduli} hold, the canonical decomposition of $v$ is $v$, and moreover $v$ is minimal.
\end{enumerate}
\end{rem}

 \subsection{Crepant resolutions}
 We now investigate how the quiver varieties $\mathfrak{M}_\theta$ are related to each other. 
 
 \begin{prop}
\label{prop:crepres}
For any $\theta \in \Theta_{v}$, the morphism to the affine quotient $f_\theta\colon \mathfrak{M}_\theta \rightarrow \mathfrak{M}_0$ obtained by variation of GIT quotient is projective and birational. If $\theta \in \Theta^{\reg}_{v}$ then $f_\theta$ is a crepant resolution. 
\end{prop}

\begin{proof}
The morphism $f_{\theta}$ is projective by construction. Moreover, as noted in Theorem \ref{thm:birationalimage}, $f_{\theta}$ is a Poisson morphism, birational onto its image. By Proposition \ref{prop:canonicaldecomp}, $\dim \mathfrak{M}_\theta = \dim \mathfrak{M}_0 = 2n$ and hence the image of $f_{\theta}$ is in fact $\mathfrak{M}_0$, Thus, $f_{\theta}$ is birational. Finally, if we assume that $\theta \in \Theta^{\reg}_{v}$, then Theorem \ref{thm:smoothmoduli} says that $\mathfrak{M}_{\theta}$ is a smooth symplectic variety. This implies that $f_{\theta}$ is a symplectic, and hence crepant, resolution.  
\end{proof}

 Two of the resolutions from Proposition~\ref{prop:crepres} are well-known. For the first, let $f_S\colon S\to \C^2/\Gamma$ denote the minimal resolution. Since $S$ is smooth, Fogarty~\cite{Fogarty68} shows the Hilbert scheme of $n$-points on $S$, denoted
 \[
 X:= \Hilb^{[n]}(S),
 \]
 is smooth, and the Hilbert--Chow morphism $\Hilb^{[n]}(S)\to \Sym^n(S)$ is a crepant resolution. By composing this with the functorial morphism $\Sym^n(S)\to \Sym^n(\C^2/\Gamma)$, we obtain a crepant resolution of singularities
 \begin{equation}
 \label{eqn:tau}
 f\colon X\longrightarrow Y:= \Sym^n(\C^2/\Gamma)\cong \C^{2n}/\Gamma_n,
 \end{equation}
 where the wreath product $\Gamma_n:=\s_n \wr \Gamma = \s_n \ltimes \Gamma^n$ acts on $(\C^2)^n\cong \C^{2n}$ in the natural way. That $f$ is an example of one of the morphisms from Proposition~\ref{prop:crepres} is due to Kuznetsov~\cite{Kuznetsov07}; the same observation was made independently by Haiman~\cite{Haiman00} and Nakajima~\cite{NakajimaPC}. To state the result, Remark~\ref{rem:transfer} enables us to regard the chamber $C_-$ from Example~\ref{exa:Cpmn}(1) as a GIT chamber in $\Theta_{\bv}$.
  
 \begin{thm}[Kuznetsov]
 \label{thm:Kuznetsov}
 For any $\theta\in C_-$, there is a commutative diagram
  \begin{equation*}
\begin{tikzcd}
 X=\Hilb^{[n]}(S) \ar[r]\ar[d,"{f}"] & \mathfrak{M}_\theta \ar[d,"{f_\theta}"] \\
 Y= \C^{2n}/\Gamma_n\ar[r] &  \mathfrak{M}_0
 \end{tikzcd}
\end{equation*}
 in which the horizontal arrows are isomorphisms and the vertical arrows are symplectic resolutions. 
  \end{thm}
  \begin{proof}
 Kuznetsov~\cite[Theorem~43]{Kuznetsov07} gives the horizontal isomorphism $X\to \mathfrak{M}_\theta$ for stability parameters $\theta$ in an open subset of $C_-$. Since $C_-$ is a GIT chamber, the isomorphism exists for all $\theta\in C_-$. The lower horizontal isomorphism from the diagram is from Lemma~\ref{lem:singularity}, and the vertical arrows are the symplectic resolutions from \eqref{eqn:tau} and Proposition~\ref{prop:crepres}. The diagram commutes by \cite[Proof of Proposition~44]{Kuznetsov07}.
  \end{proof}

 For the second well-known example of the morphism from Proposition~\ref{prop:crepres}, set $N:= \vert \Gamma\vert$. The action of $\Gamma$ on $\C^2$ induces an action of $\Gamma$ on the Hilbert scheme $\Hilb^{[nN]}(\C^2)$. Let $n\Gamma\text{-}\Hilb(\C^2)$ denote the subscheme parametrising $\Gamma$-invariant ideals $I$ in $\Hilb^{[nN]}(\C^2)$ such that the quotient $\C[x,y]/I$ is isomorphic to the direct sum of $n$ copies of the regular representation of $\Gamma$. Recall the chamber $C_+$ from Example~\ref{exa:Cpmn}(2). 

 \begin{thm}[Varagnolo--Vasserot, Wang]
 \label{thm:VVW}
 For any $\theta\in C_+$, there is a commutative diagram
  \begin{equation*}
\begin{tikzcd}
 n\Gamma\text{-}\Hilb(\C^2) \ar[r]\ar[d,"{}"] & \mathfrak{M}_\theta \ar[d,"{f_\theta}"] \\
 Y= \C^{2n}/\Gamma_n\ar[r] &  \mathfrak{M}_0
 \end{tikzcd}
\end{equation*}
 in which the horizontal arrows are isomorphisms and the vertical arrows are symplectic resolutions. 
 \end{thm}
 \begin{proof}
 The isomorphism given by the top horizontal arrow is due to Varagnolo-Vasserot~\cite{VVcyclic} and Wang~\cite[Theorem 2]{wang}, while commutativity of the diagram is noted in \cite[Remark~41]{Kuznetsov07}.
 \end{proof}
 
 \begin{remark}
 The moduli space $\mathfrak{M}_\theta$ defined by any parameter $\theta\in C_+$ is characterised among the moduli $\mathfrak{M}_\theta$ for arbitrary $\theta\in \Theta_{\bv}$ by the property that the tautological vector bundles $\mc{R}_i$ for $i\in I$ on $\mathfrak{M}_\theta$ are all globally generated, see \cite[Corollary~2.4]{CIK18}. 
 \end{remark}
 
\subsection{The case $n = 1$}
\label{sec:n=1}
This case is somewhat different since the quiver variety $\mathfrak{M}_\theta$ can be smooth even for non-generic parameters.

\begin{lem}\label{lem:nequal1chamberdecomp}
If $n = 1$, then $\mathfrak{M}_\theta$ is smooth if and only if $\theta$ does not lie on any $\alpha^{\perp}$, for $\alpha \in \Phi$. 
\end{lem}

\begin{proof}
The proof is identical to the proof of Theorem \ref{thm:smoothmoduli}, except in this case when $\theta_{\infty} = 0$, the minimal imaginary root $\delta$ only appears with multiplicity one in the canonical decomposition. Therefore, a parameter $\theta$ on the hyperplane $\delta^{\perp}$, which is not contained in any $\alpha^{\perp}$ for $\alpha\in \Phi$, is such that $\mf{M}_\theta$ is smooth. Every point of $\mf{M}_{\theta}$ corresponds to a strictly polystable representation even though the moduli space is smooth.
\end{proof}

In particular, when $n=1$ we recover the chamber structure for the unframed McKay quiver moduli construction due to Kronheimer~\cite{Kronheimer} (see also Cassens--Slodowy~\cite{CassensSlodowy}) in the hyperplane $\delta^\perp$. Indeed, it is well-known that the chamber decomposition in this case is the Weyl chamber decomposition for $\Phi$. 

\subsection{Universal Poisson deformations}

The description of the Namikawa Weyl group in Section \ref{sec:NamWeylgroup} allows us to identify the space $\Theta_v \o_{\Q} \C$ with the (complex) reflection representation $\h$ of $W$. For $\theta \in \Theta_v$, consider the family
\[
\sigma_{\theta}\colon \boldsymbol{\mf{M}}_{\theta} := \mu^{-1}(\h)^{\theta} \git \, G(v)\longrightarrow \h
\]
deforming $\mf{M}_{\theta}$. This becomes a graded Poisson morphism if we regard $\h$ as a Poisson variety with trivial bracket. We now explain how to identify this family, when $\theta$ is generic, with the universal graded Poisson deformation of $\mf{M}_{\theta}$ whose existence was shown by Kaledin-Verbitsky~\cite{KaledinVerbitskyPeriod} in the ungraded case, and adapted by Losev~\cite{Losev} to the graded case. 

There is a commutative diagram
\begin{equation}
\label{eq:universalfamilywreath}
 \begin{tikzcd}
   & \mf{M}_\theta\ar[r,"f_\theta"]\ar[ld,hook']\ar[rdd] & \C^{2n}/\Gamma_n\ar[ld,hook',crossing over]\ar[dd] \\
\boldsymbol{\mf{M}}_{\theta} \ar[r,"\boldsymbol{f}_{\theta}"] \ar[ddr,"\sigma_{\theta}"] & \boldsymbol{\mf{M}}_0 \ar[dd,"\sigma_{0}"] & & \\
 & & 0 \ar[ld,hook']  \\
& \h & 
\end{tikzcd}
\end{equation}
where $\boldsymbol{f}_{\theta} \colon \boldsymbol{\mf{M}}_{\theta} \rightarrow \boldsymbol{\mf{M}}_0$ is the morphism given by variation of GIT. 

\begin{lem}\label{lem:boldftheta}
The morphism $\boldsymbol{f}_{\theta}$ is projective, birational and Poisson. 
\end{lem}

\begin{proof}
The fact that the morphism is projective and Poisson follows directly from the definition of variation of GIT. To see that it is birational it suffices to note that there is a finite collection of hyperplanes in $\h$ such that $\boldsymbol{f}_{\theta}$ is an isomorphism away from those hyperplanes. Indeed, let $S$ denote the complement to all hyperplanes $\alpha^{\perp}$ for $\alpha < v$ a positive root; these define proper hyperplanes in $\h$ because $v$ is indivisible. If $\lambda \in S$ then the only positive root $\alpha \le v$ in the set $\Sigma_{\lambda}$ is $v$ itself. By \cite[Theorem 1.2]{CBmomap}, this implies that every representation of the deformed preprojective algebra $\Pi^{\lambda}$ of dimension $v$ is simple, and hence $\theta$-stable. Thus, the restriction of $\boldsymbol{f}_{\theta}$ to $\sigma_{\theta}^{-1}(S)$ is the identity map $\sigma_{\theta}^{-1}(S) \rightarrow \sigma_0^{-1}(S)$ because $\mu^{-1}(S)^{\theta} = \mu^{-1}(S)$. In particular, $\boldsymbol{f}_{\theta}$ is birational.
\end{proof}

\begin{lem}
\label{lem:defspaceflatCMquiver}
\begin{enumerate}
    \item[\one] The morphism $\sigma_0$ is flat, and $\boldsymbol{\mf{M}}_0$ is irreducible, normal, and Cohen-Macaulay.
    \item[\two] If $\theta \in \Theta_v^{\reg}$, then $\sigma_\theta$ is flat and $\boldsymbol{\mf{M}}_{\theta}$ is smooth and connected.
\end{enumerate}
\end{lem}
\begin{proof}
 For \one, the locus $\mu^{-1}(\h) $ is cut out by $n^2 |\Gamma| - r$ equations in the affine space $\mathbf{M}(\bv,\bw)$ of dimension $2n (n |\Gamma| + 1)$, so it is a complete intersection of dimension $n^2 |\Gamma| + 2n + r$ by \cite[Theorem 3.7]{AlmostCommutingVariety}. In particular, $\mu^{-1}(\h) $ is Cohen-Macaulay, and hence the map $\mu^{-1}(\h) \rightarrow \h$ is flat \cite[Corollary 23.1]{MatCom}.  Since $\C[\mu^{-1}(\h)]^{G(v)}$ is a direct summand of $\C[\mu^{-1}(\h)]$ as a $\C[\h]$-module, it follows that $\sigma_0 \colon \boldsymbol{\mf{M}}_0 \rightarrow \h$ is also flat. As for $\boldsymbol{\mf{M}}_0$,  each fiber of $\sigma_0$ has symplectic singularities by \cite[Theorem 1.2]{BelSchQuiver}, so the fibres have rational singularities and hence are Cohen-Macaulay. It follows from \cite[Corollary 23.3]{MatCom} that $\boldsymbol{\mf{M}}_0$ is Cohen-Macaulay. To show normality, it suffices to show that $\boldsymbol{\mf{M}}_0$ satisfies $(R_1)$. Let $U$ be the set of points $x$ in $\boldsymbol{\mf{M}}_0$ such that $x$ is a regular point in the fiber $\sigma^{-1}_0(\sigma_0(x))$. Then \cite[Theorem 23.7]{MatCom} says that $U$ is contained in the smooth locus of $\boldsymbol{\mf{M}}_0$. Thus, we need to show that the complement $C$ to $U$ in $\boldsymbol{\mf{M}}_0$ has codimension at least two. The image of $C$ in $\h$ is contained in the finitely many hyperplanes described in the proof of Lemma \ref{lem:boldftheta} since the set $S$ described there satisfies $\sigma_0^{-1}(S) \subset U$. Since $C \cap \sigma_0^{-1}(\lambda)$ has codimension at least one in $\sigma_0^{-1}(\lambda)$ (in fact it has codimension at least two since each fiber $\sigma_0^{-1}(\lambda)$ is normal), $\mathrm{codim}_{\boldsymbol{\mf{M}}_0} \, C \ge \mathrm{codim}_{\h} \, \sigma_0(C) + 1 \ge 2$. 
 
 For \two, an argument similar to that for $\sigma_0$ above shows that $\sigma_{\theta}$ is also flat.  As for the variety $\boldsymbol{\mf{M}}_\theta$, all maps appearing in diagram \eqref{eq:universalfamilywreath} are equivariant for the natural scaling action of $\Cs$ described in section \ref{sec:Nakasec}, if we make $\Cs$ act on $\h$ with weight one. This implies that the singular locus of $\boldsymbol{\mf{M}}_{\theta}$ is a $\Cs$-stable closed subvariety. In particular, if it is non-empty then it would intersect the zero fiber $\sigma_{\theta}^{-1}(0) = \mf{M}_{\theta}$ non-trivially. However, since $\sigma_{\theta}$ is flat and $\mf{M}_{\theta}$, $\h$ are smooth, \cite[Theorem 23.7]{MatCom} says that this intersection is trivial. Thus, $\boldsymbol{\mf{M}}_{\theta}$ is smooth. Similarly, each connected component of $\boldsymbol{\mf{M}}_{\theta}$ is a closed $\Cs$-stable subvariety, and hence must intersect $\mf{M}_{\theta}$ non-trivially. But the latter is irreducible, so $\boldsymbol{\mf{M}}_{\theta}$ is connected.
\end{proof}

\begin{prop}\label{prop:deformedresiscrepant}
The morphism $\boldsymbol{f}_{\theta}$ is crepant, and an isomorphism in codimension one. 
\end{prop}

\begin{proof}
Let $2n = \dim \mf{M}_{\theta}$. As in the proof of Lemma \ref{lem:defspaceflatCMquiver}\two, we note that all maps appearing in diagram \eqref{eq:universalfamilywreath} are equivariant for the natural scaling action of $\Cs$. The Poisson bracket on both $\boldsymbol{\mf{M}}_{\theta}$ and $\boldsymbol{\mf{M}}_{0}$ is $\mc{O}_{\h}$-linear. Since the bracket is homogeneous of degree $-2$, it defines a $\Cs$-equivariant map $\phi_{\theta} \colon \Omega^1_{\boldsymbol{\mf{M}}_{\theta}/ \h} \rightarrow \Tang_{\boldsymbol{\mf{M}}_{\theta} / \h} \langle 2\rangle$, where $\Tang_{\boldsymbol{\mf{M}}_{\theta} / \h}$ is the sheaf of $\mc{O}_{\h}$-linear derivations on $\boldsymbol{\mf{M}}_{\theta}$, and $\langle 2\rangle$ denotes a shift in the grading by two. The degeneracy locus of $\phi_{\theta}$ is closed and $\Cs$-stable. In particular, if it were non-empty it would intersect the fiber $\sigma_{\theta}^{-1}(0) = \mf{M}_{\theta}$ non-trivially. However, if $i : \mf{M}_{\theta} \hookrightarrow \boldsymbol{\mf{M}}_{\theta}$ is the embedding of the zero fibre, then $i^* \phi_{\theta} \colon \Omega^1_{\mf{M}_{\theta}} \rightarrow \Tang_{\mf{M}_{\theta}}$ is an isomorphism since the Poisson bracket restricted to $\mf{M}_{\theta}$ is non-degenerate. Therefore, we deduce that $\phi_{\theta}$ is an isomorphism. Thus, there is a relative symplectic form $\omega_{\boldsymbol{\mf{M}}_{\theta}/ \h}$ on $\boldsymbol{\mf{M}}_{\theta}$, dual to the Poisson bracket. Since $\h$ is obviously smooth, this implies that $\sigma_{\theta}$ is smooth and, if $\omega_{\h}$ is a nowhere vanishing top form on $\h$, then $\omega_{\boldsymbol{\mf{M}}_{\theta}/ \h}^n \wedge \sigma_{\theta}^* \omega_{\h}$ is a nowhere vanishing top form on $\boldsymbol{\mf{M}}_{\theta}$. In particular, the canonical divisor $K_{\boldsymbol{\mf{M}}_{\theta}}$ is trivial. 

Similarly, on $\boldsymbol{\mf{M}}_{0}$, we have $\phi_0 \colon \Omega^1_{\boldsymbol{\mf{M}}_{0}/ \h} \rightarrow \Tang_{\boldsymbol{\mf{M}}_{0}/ \h}\langle 2\rangle$. Let $D \subset \boldsymbol{\mf{M}}_{0}$ be its degeneracy locus. This is a closed, $\Cs$-stable subset of $\boldsymbol{\mf{M}}_{0}$. Let $D_{\le i} = \{ x \in D \ | \ \dim (D \cap \sigma_0^{-1}(\sigma_0(x))) \le i \}$. Then, by Chevalley's Theorem, each $D_{\le i}$ is closed and $\Cs$-stable. We know that $\dim D \cap \sigma^{-1}_0(0) \le 2n -2$ since the map $\Omega^1_{\mf{M}_0} \rightarrow \Tang_{\mf{M}_0}$ is an isomorphism on the smooth locus of $\mf{M}_0$. Therefore, we deduce $D = D_{\le 2n - 2}$ since the $\Cs$-action contracts every $\lambda \in \h$ to $0$. This implies that $\dim D \le (2n - 2) + \dim \h = \dim \boldsymbol{\mf{M}}_{0} - 2$, i.e. there exists an open subset $U \subset \boldsymbol{\mf{M}}_{0}$ whose complement $D$ has codimension at least two, such that $\phi_0$ is an isomorphism over $U$. We deduce that there exists a relative symplectic form $\omega_{\boldsymbol{\mf{M}}_{0} /\h}$ on the open set $U$. Moreover, the fact that $\boldsymbol{f}_{\theta}$ is Poisson implies that $\boldsymbol{f}_{\theta}^*\omega_{\boldsymbol{\mf{M}}_{0} /\h} = \omega_{\boldsymbol{\mf{M}}_{\theta} /\h}$. Again, the top form $\omega_{\boldsymbol{\mf{M}}_{0} /\h}^n \wedge \sigma_0^* \omega_{\h}$ trivialises the canonical divisor $K_{\boldsymbol{\mf{M}}_{0}}$ over $U$, and we deduce by normality that $K_{\boldsymbol{\mf{M}}_{0}} = 0$, i.e. $\boldsymbol{\mf{M}}_{0}$ is Gorenstein. Thus, $\boldsymbol{f}_{\theta}^* K_{\boldsymbol{\mf{M}}_{0}} = K_{\boldsymbol{\mf{M}}_{\theta}}$, i.e. $\boldsymbol{f}_{\theta}$ is crepant. 

Finally, we check that $\boldsymbol{f}_{\theta}$ is an isomorphism in codimension one. Suppose otherwise, so $\boldsymbol{f}_{\theta}$ contracts a divisor $E$ in $\boldsymbol{\mf{M}}_{\theta}$. Note that $E$ is necessarily stable under the action of $\Cs$. The map $\sigma_{\theta}$ is proper, therefore $\sigma_{\theta}(E)$ is closed in $\h$. Since it lies in the complement to the set $S$ defined in the proof of Lemma \ref{lem:boldftheta}, it is a proper subset of $\h$. Therefore, there must exist a fiber of $\sigma_{\theta} |_E \colon E \rightarrow \h$ of dimension at least $2n$. Arguing as above, we must in fact have $\dim (\sigma_{\theta} |_E)^{-1}(0) \ge 2n$. But $(\sigma_{\theta} |_E)^{-1}(0)  \subset \sigma_{\theta}^{-1}(0) = \mf{M}_{\theta}$, which is irreducible of dimension $2n$. Thus, $(\sigma_{\theta} |_E)^{-1}(0) = \mf{M}_{\theta}$. But $\boldsymbol{f}_{\theta} |_{\mf{M}_{\theta}} = f_{\theta}$ is generically an isomorphism. In particular, a generic fiber is zero dimensional. This is a contradiction. 
\end{proof}

In the proof of Proposition \ref{prop:deformedresiscrepant}, we also established the following result:

\begin{cor}
The  variety $\boldsymbol{\mf{M}}_{0}$ is Gorenstein.
\end{cor}

Given any conic symplectic singularity $Y$ that admits a symplectic resolution $X \rightarrow Y$, Namikawa~\cite{Namikawa} showed that there exist universal graded Poisson deformations $\mc{Y} \rightarrow \h/ W$ and $\mc{X} \rightarrow \h$ of $Y$ and $X$ respectively. Write $q\colon \mf{h}\to \mf{h}/W$ for the quotient map.  

\begin{thm}
There exists a unique $w \in W$ and graded Poisson isomorphism $\boldsymbol{\mf{M}}_{\theta} \iso \mc{X}$ such that the diagram
$$
\begin{tikzcd}
\boldsymbol{\mf{M}}_{\theta} \ar[r] \ar[d,"\wr"] & \h \ar[d,"w \cdot "] \\
\mc{X} \ar[r] & \h
\end{tikzcd}
$$
 commutes. In particular, the flat family $\boldsymbol{\mf{M}}_{\theta} \rightarrow \h$ is the universal graded Poisson deformation of the quiver variety $\mf{M}_{\theta}$. 
\end{thm}

\begin{proof}
This follows from the main results of \cite{BellamyCounting} and \cite{Namikawa3}. To see this, first identify $\mf{h}$ with the subset $\{ (\lambda_i \mathrm{Id}_{V_i}) \mid \lambda_i \in \C, \ \lambda(v) = 0 \} \subset \mf{g}(v)$ 
of diagonal matrices pairing to zero with $v$. It is shown in the proof of Theorem 1.4 of \cite{MarsdenWeinsteinStratification} that there is an explicit isomorphism between $\boldsymbol{\mf{M}}_0 \rightarrow \h$ and the Calogero-Moser space $\mathsf{CM}(\Gamma_n) \rightarrow \mf{c}$ considered in \cite{BellamyCounting}. In particular, this defines an isomorphism $\mf{h} \iso \mf{c}$, which one can check is $W$-equivariant if the action of $W$ on $\mf{c}$ is defined via the isomorphism (3.B) of \cite{BellamyCounting}. Therefore \cite[Theorem 1.4]{BellamyCounting} implies that $\boldsymbol{\mf{M}}_0 \rightarrow \h$ is isomorphic to $\mc{Y} \times_{\h/W} \h$, where $\h \rightarrow \h/W$ is the quotient map $q$. Since $\boldsymbol{f}_{\theta} \colon \boldsymbol{\mf{M}}_{\theta} \rightarrow \boldsymbol{\mf{M}}_0$ is a crepant resolution by Proposition \ref{prop:deformedresiscrepant}, the theorem follows from \cite[Corollary 7]{Namikawa3}. 
\end{proof} 

\begin{remark}
The proof of \cite[Theorem 1.4]{BellamyCounting} can be applied directly, word for word, to $\boldsymbol{\mf{M}}_0 \rightarrow \h$ by using the factorisation from \cite[Corollary 3.4]{BelSchQuiver} to show that $\boldsymbol{\mf{M}}_0 \rightarrow \h$ is isomorphic to $\mc{Y} \times_{\h/W} \h$. This avoids using the identification of the former with the Calogero-Moser space, and extends \eqref{eq:universalfamilywreath} to a commutative diagram
\begin{equation}
 \begin{tikzcd}
   & \mf{M}_\theta\ar[r,"f_\theta"]\ar[ld,hook']\ar[ddr] & \C^{2n}/\Gamma_n\ar[ld,hook',crossing over]\ar[dd]\ar[dr,hook'] & & & \\
\boldsymbol{\mf{M}}_{\theta} \ar[r,"\boldsymbol{f}_{\theta}"] \ar[ddr,"\sigma_{\theta}"] & \boldsymbol{\mf{M}}_0 \ar[dd,"\sigma_{0}"] \ar[rr,crossing over] & & \mc{Y}\ar[dd] & & \\
 & & 0\ar[ld,hook']\ar[rd,hook'] & & & \\
& \h \ar[rr,"q"] & & \h/W & 
\end{tikzcd}
\end{equation}
 where the large rectangle at the front of the diagram is Cartesian. 
\end{remark}

\section{Variation of GIT quotient in the cone \protect$F\protect$}
 We now describe the geometry of the quiver varieties $\mf{M}_\theta$ under variation of GIT quotient as the stability parameter $\theta$ passes from a chamber of $F$ into a wall of the chamber. We also study the geometry of $\mf{M}_\theta$ and track what happens to the tautological bundles as $\theta$ crosses any wall that passes through the interior of $F$.

\subsection{Classification of walls}
 Consider the  wall-and-chamber structure on the space $\Theta_v$ of GIT stability conditions, as in equation \eqref{eqn:Thetav}. We now classify the walls into three families. We say that a wall is: 
\begin{enumerate}
\item[\one] an \emph{imaginary boundary wall} if it is contained in $\delta^{\perp}$;
\item[\two] a \emph{real boundary wall} if it is contained in $\alpha^{\perp}$ for some $\alpha \in \Phi^+$; and
\item[\three] a \emph{real internal wall} otherwise. 
\end{enumerate}
This terminology is motivated by the walls that lie in the cone $F$ introduced in Proposition~\ref{prop:W}\two: each wall in the boundary of $F$ is contained in either $\delta^{\perp}$ or $\alpha^{\perp}$ for some $\alpha \in \Phi^+$; while every other wall in $F$ passes through the interior of $F$.

\subsection{The imaginary boundary wall}

The hyperplane $\delta^{\perp}$ contains a single boundary wall of $F$. Let $\theta_0$ be a generic point of this wall. There exists a unique chamber $C$ in $F$, whose closure contains $\theta_0$; this is the chamber $C_-$ introduced in Example~\ref{exa:Cpmn}(1). Theorem~\ref{thm:Kuznetsov} shows that $\mf{M}_\theta \cong \Hilb^{[n]}(S)$ for $\theta\in C_-$, and Kuznetsov~\cite[Remark~48]{Kuznetsov07} proves further that the morphism $f\colon \mf{M}_\theta\to\mf{M}_{\theta_0}$ is isomorphic to the Hilbert--Chow morphism
\[
X=\Hilb^{[n]}(S)\longrightarrow \Sym^n(S).
\]
In particular, $f$ is a divisorial contraction. The statement of Corollary~\ref{cor:imaginarywallconstraction} below is therefore not new. Nevertheless, we present the next result to provide an independent proof of this fact and, furthermore, to illustrate our approach via the Ext-graph as described in section~\ref{sec:localnormalform}.

\begin{prop}\label{prop:etalelocalHilChowdelta}
The strata of $\mf{M}_{\theta_0}$ are $(\mf{M}_{\theta_0})_{\lambda}$, where the partition $\lambda = (n_0,\dots,n_{k}) \vdash n$ corresponds to the representation type
$\lambda = (1,\rho_{\infty};n_0, \delta; \ds ; n_{k}, \delta)$. In an \'etale neighbourhood of $x \in (\mf{M}_{\theta_0})_{\lambda}$, the morphism $f \colon X=\mf{M}_\theta \rightarrow \mf{M}_{\theta_0}$ is isomorphic to the product of Hilbert-Chow morphisms
$$
\Hilb^{[n_0]}(\C^2) \times \cdots \times \Hilb^{[n_{k}]}(\C^2) \longrightarrow \Sym^{n_0} (\C^2) \times \cdots \times \Sym^{n_{k}}(\C^2).
$$
\end{prop}

\begin{proof}
First, we describe the strata of $\mf{M}_{\theta_0}$ by listing the representation types of $v$. Since $\theta_0$ is a generic point on $\delta^{\perp}$, the only positive roots $\gamma< v$ satisfying $\theta_0(\gamma)=0$ are $\rho_{\infty} + m \delta$ and $m \delta$, for $0 \le m < n$. We have $\rho_{\infty}, \delta \in \Sigma_{\theta_0}$; we have seen in the proof of Proposition \ref{prop:canonicaldecomp} that $\rho_{\infty} + m \delta \notin \Sigma_{\theta_0}$ when $m \ge 1$, even though it is a root. Thus, the representation types of $v$ are
$$
\lambda = (1,\rho_{\infty};n_0, \delta; \ds ; n_{k}, \delta) \ \left( =:(1,\beta^{(\infty)};n_0, \beta^{(0)}; \ds ; n_{k}, \beta^{(k)} \right);
$$
here the open stratum corresponds to $\lambda = (1^n)$, and the unique closed stratum is $\lambda = (n)$. 

Fix a partition $\lambda = (n_0,\dots,n_{k})$ of $n$ and choose $x \in (\mf{M}_{\theta_0})_{\lambda}$. We apply Theorem \ref{thm:commdiagram22} to describe the \'etale-local picture of $f \colon \mf{M}_{\theta} \rightarrow \mf{M}_{\theta_0}$ at $x$. Recall from section \ref{sec:localnormalform} that the associated Ext-graph has vertices $\{ 0, \ds, k \}$, with $p\left(\beta^{(i)}\right) = p(\delta) = 1$ loops at vertex $i$ and $- \left(\beta^{(i)}, \beta^{(j)}\right) = - (\delta,\delta) = 0$ edges between vertex $i$ and $j$. That is, the graph is simply the disjoint union of $k+1$ loops. Since $p\left(\beta^{(\infty)}\right) = p(\rho_{\infty}) = 0$, there are no loops at the framing vertex. Because $-\left(\beta^{(\infty)},\beta^{(i)}\right) = - (\rho_{\infty},\delta) = 1$, the corresponding dimension vectors are 
$$
\bm = (n_0,\dots,n_{k}), \quad \bn = (1, 1, \ds, 1). 
$$
Choose $\theta\in C_-$ such that $\theta(\rho_\infty)=n$ and $\theta(\delta) = -1$. The stability condition $\varrho$ for the Ext-graph satisfies $\varrho_i = \theta(\delta) = -1$ for $0\leq i\leq k$ by \eqref{eqn:varrho}. Since the graph has $k+1$ connected components, it is a product 
$$
\mf{M}_{\varrho}(\bm,\bn) \cong \prod_{0\leq i\leq k} \mf{M}_{(-1,n_i)}((n_i),(1)).
$$ 
The result follows since $\mf{M}_{(-1,m)}((m),(1)) \cong \Hilb^{[m]}(\C^2)$ and $\mf{M}_{0}((m),(1)) \cong \Sym^m(\C^2)$.
\end{proof}

\begin{cor}[Kuznetsov]\label{cor:imaginarywallconstraction}
 For $\theta\in C_-$, the morphism $f \colon X=\mf{M}_\theta \rightarrow \mf{M}_{\theta_0}$ is a divisorial contraction and the exceptional locus is an irreducible effective divisor. 
\end{cor}
\begin{proof}
Let $Z$ denote the singular locus of $\mf{M}_{\theta_0}$. Since each of the leaves $(\mf{M}_{\theta_0})_{\lambda}$ is contained in the closure of $(\mf{M}_{\theta_0})_{(2,1^{n-2})}$, and the latter is smooth and connected, we deduce that $Z$ is irreducible. Therefore, to check that the exceptional locus $E = f^{-1}(Z)$ is an irreducible divisor, it suffices to do so \'etale locally on $Z$. If $x \in  (\mf{M}_{\theta_0})_{\lambda} \subset Z$, then by Proposition~\ref{prop:etalelocalHilChowdelta}, it suffices to show that the exceptional locus of the Hilbert-Chow morphism $\Hilb^{[m]}(\C^2) \rightarrow \Sym^m(\C^2)$ is an irreducible divisor for $m>1$. This is shown in \cite{Fogarty68}. 
\end{proof}

\subsection{The stratification induced by real walls}\label{sec:realwallsstrata}

By Lemma \ref{lem:Fwalls}, the real interior walls are precisely those contained in $\gamma^\perp:= (m \delta - \alpha)^{\perp}$ for some $\alpha \in \Phi^+$ and $0 < m < n$. The real boundary walls are contained in  $\alpha^{\perp}$ for some $\alpha \in \Phi^+$.

 Fix a real wall and take $\theta_0$ a generic point of this wall. The hyperplane containing the wall is of the form $\gamma^\perp = (m\delta-\alpha)^\perp$ for some $\alpha\in \Phi^+$ and $0\leq m < n$, and we define $N$ to be the largest positive integer such that $N(N+m) \le n$. 

\begin{lem}
\label{lem:leavesinternalwall}
 The variety $\mf{M}_{\theta_0}$ has strata $\{ \mc{L}_k \ |  \ 0 \le k \le N \}$, where $\mc{L}_k$ is isomorphic to the fine moduli space of $\theta_0$-stable $\Pi$-modules of dimension vector $v - k \gamma$ with 
$$
\gamma := \left\{ \begin{array}{ll} 
m\delta-\alpha & \textrm{if $m > 0$,} \\
\alpha & \textrm{if $m = 0$}.
\end{array} \right.
$$
Moreover, $\mc{L}_k$ is (non-empty) of codimension $2k(k+m)$.
\end{lem}

\begin{proof}
We claim that $v - k \gamma \in \Sigma_{\theta_0}$ for all $0 \le k \le N$. First, we note from Lemma \ref{lem:FrenkelKac} that 
$$
v - k \gamma = \rho_{\infty} + (n - k(k+m)) \delta + \frac{1}{2} (k \alpha,k\alpha) \delta + k\alpha
$$
is a root in $R_{\theta_0}^+$. Next, let $\tau = (1,\beta^{(\infty)};n_0,\beta^{(0)}; \ds ;n_{\ell}, \beta^{(\ell)})$ be a representation type of $v - k \gamma$ with $\beta^{(i)} \in \Sigma_{\theta_0}$ and $\beta^{(\infty)}_{\infty} = 1$. The assumption that $\theta_0$ is a generic element in $\gamma^{\perp}$ implies that the only positive root $\xi < n \delta$ that pairs to zero with $\theta_0$ is $\gamma$. This implies that $\beta^{(i)} = \gamma$ for all $i \ge 0$. That is, $v - k \gamma = (v - (k + l) \gamma) + n_0 \gamma+ \cdots + n_{\ell} \gamma$. In fact, as noted in Remark \ref{rem:stratapropertydecomp}, the fact that $\gamma$ is real forces $\ell = 0$ and hence $n_0 = l$ i.e. $\tau = (1,v - (k + l) \gamma; l, \gamma)$. Then, the fact that this is a representation type forces $v - (k + l) \gamma \in \Sigma_{\theta_0}$. In particular, $v - (k + l) \gamma$ must be a positive root. Lemma~\ref{lem:FrenkelKac} implies that 
$$
p(v - (k+ l) \gamma) = n - (k+l) (m + k + l) \ge 0,
$$
i.e. $0 \le k + l \le N$. Moreover, $\gamma$ is a real root, so $p(\gamma) = 0$ and 
$$
\sum_{i = 1}^{\ell} p(\beta^{(i)}) = p(v - (k+l)\gamma) = n - (k+l)(k+l +m) < n - k(k+m)
$$
if $l \ge 1$. This implies that $v - k \gamma \in \Sigma_{\theta_0}$ for all $0 \le k \le N$.

In particular, this applies to $v$ (in the case $k = 0$) and shows that every representation type of $v$ is of the form $(1,v - k \gamma;k,\gamma)$. The symplectic leaves of $\mf{M}_{\theta_0}$ are in bijection with these representation types, with the type $(1,v - k \gamma;k,\gamma)$ corresponding to 
$$
\mc{L}_k = \big\{ M = M_{\infty} \oplus M_0^{\oplus k} \ | \ \dim M_{\infty} = v - k \gamma, \ \dim M_0 = \gamma, \textrm{ and $M_{\infty},M_0$ are $\theta_0$-stable} \big\}.
$$
But $\gamma$ is real, so up to isomorphism there is only one $\theta_0$-stable representation of dimension $\gamma$. Therefore $\mc{L}_k$ is isomorphic to the fine moduli space of $\theta_0$-stable $\Pi$-modules of dimension vector $v - k \gamma$. The fact that $v - k \gamma \in \Sigma_{\theta_0}$ implies that $\mc{L}_k$ is non-empty of dimension $2 p(v - k \gamma) = 2(n - k(k+m))$.
\end{proof}

\subsection{The real boundary walls}

Fix a positive real root $\alpha \in \Phi^+$. The walls of $F$ contained in $\alpha^{\perp}$ are the real boundary walls. Assume that $\theta$ is chosen in a chamber $C \subset F$ with $\theta_0 \in \overline{C} \cap \alpha^{\perp}$ a generic point on the wall $\alpha^{\perp}$.

\begin{prop}
\label{prop:realboundary}
The morphism $f \colon \mf{M}_{\theta} \rightarrow \mf{M}_{\theta_0}$ is a divisorial contraction and the exceptional locus is an irreducible effective divisor.
\end{prop}

\begin{proof}
By Lemma \ref{lem:leavesinternalwall} with $m=0$, the leaves of $\mf{M}_{\theta_0}$ are $\mc{L}_k$ where $k^2 \le n$. The leaf $\mc{L}_0$ is open and its complement $Z$ is the singular locus of $\mf{M}_{\theta_0}$, which equals $\overline{\mc{L}}_1$. As in the proof of Corollary \ref{cor:imaginarywallconstraction}, the fact that $\mc{L}_1$ is irreducible implies that it suffices to check \'etale locally on $Z$ that the exceptional locus is an irreducible divisor. 

Assume that $x \in \mc{L}_k$ for some $k \ge 1$. The representation type of $v$ labelling this leaf is $(1,v - k \alpha; k, \alpha)$. Note first that $\ell:=p(v-k\alpha) = n-k^2$. Also, the $\Ext$-graph has one vertex $0$, with $p(\alpha) = 0$ loops at this vertex. 
We have $\bm = (k)$ and the fact that $-(v-k\alpha,\alpha) = 2k$ implies that $\bn = (2k)$. Finally, we have $\varrho(\infty) = - \theta(k\alpha) < 0$ and $\varrho(0) = \theta(\alpha) > 0$. If we fix a vector space $\GVect$ of dimension $2k$, then an explicit calculation of quiver varieties shows that there is a commutative diagram
  \begin{equation*}
\begin{tikzcd}
 T^* G(k,\GVect) \ar[r]\ar[d,"{}"] & \mathfrak{M}_\varrho(\bm,\bn) \ar[d,"{f_\varrho}"] \\
 \overline{\mc{O}}\ar[r] & \mf{M}_0(\bm,\bn)
 \end{tikzcd}
\end{equation*}
 in which the horizontal arrows are isomorphisms and the vertical arrows are symplectic resolutions; here, $\mc{O} \subset \mf{gl}(\GVect)$ is the nilpotent orbit of rank $k$ matrices whose square is zero, and $G(k,\GVect)$ is the Grassmanian of $k$-planes in $\GVect$. It now follows from Theorem~\ref{thm:commdiagram22} that, \'etale locally at $x$, the morphism $f \colon \mf{M}_{\theta} \rightarrow \mf{M}_{\theta_0}$ can be identified with 
 \begin{equation}\label{eq:fetalebetawall}
\C^{2 (n-k^2)}\times T^* G(k,\GVect)  \longrightarrow \C^{2(n-k^2)} \times \overline{\mc{O}},
\end{equation}
 with $x$ mapped to $0 \in \C^{2(n-k^2)} \times \overline{\mc{O}}$. We claim that the exceptional locus of this morphism is an irreducible divisor. It suffices to assume that $n = k^2$. Let $\mathrm{Exp}$ be the exceptional locus. If $\pi : T^* G(k,\GVect) \rightarrow \mathrm{G}(k,\GVect)$ is projection to the zero section, then we claim that $\pi |_{\mathrm{Exp}}$ is a Zariski locally trivial fibre bundle, with each fibre $\pi^{-1}(V) \cap \mathrm{Exp}$ an irreducible hypersurface in $\pi^{-1}(V)$. Fixing a basis of $V$ and extending this to a basis of $\GVect$, we can identify the space $\pi^{-1}(V) \cap \mathrm{Exp}$ with the space of all $k \times k$ matrices $B$ of rank $< k$ (note that $\dim V = k$ and $\dim \GVect = 2k$). This is the zero set of the determinant, and hence a hypersurface. Being a special case of a determinantal variety, it is well-known to be irreducible. The claim follows and we deduce that the exceptional locus of $f$ is irreducible. We note that this argument shows that the exceptional divisor is singular. Finally, if $x$ belongs to the open subset $\mc{L}_1$ of $Z$, then \eqref{eq:fetalebetawall} becomes $\C^{2 n -2}\times T^* \mathbb{P}^1 \rightarrow \C^{2n-2} \times \C^2/\Z_2$ which is a divisorial contraction.
\end{proof}

\subsection{Internal walls}
In order to state the main result for real internal walls, we recall the definition of a Mukai flop of type $A$ due to Namikawa~\cite{NamikawaNilpotentIII}.

Let $\GVect$ be a finite dimensional vector space and choose an integer $1 \le k < \dim \GVect$. Let $\mc{O} \subset \mf{gl}(\GVect)$ be the nilpotent orbit of rank $k$ matrices whose square is zero, and let $G(k,\GVect)$ be the Grassmanian of $k$-planes in $\GVect$. Provided $2k \neq \dim \GVect$, the symplectic singularity $\overline{\mc{O}}$ admits two different symplectic resolutions such that the diagram 
\begin{equation}\label{eq:MukaiflopA}
\begin{tikzcd}
T^* G(k,\GVect) \ar[rr,"\psi",dashed] \ar[dr,"g_+"'] & &  T^* G(k,\GVect^*) \ar[dl,"g_-"] \\
& \overline{\mc{O}}. & 
\end{tikzcd}
\end{equation}
 is a flop \cite[Lemma 5.4]{NamikawaNilpotentIII}; that is, $g_\pm$ are small projective resolutions, and if $L_+$ is a $g_+$-ample line bundle on $T^* G(k,\GVect)$ then the proper transform $L_-$ is such that  $L_-^{-1}$ is $g_-$-ample. In particular, $\psi$ is not regular. 
 
 One can obtain the above diagram by variation of GIT quotient for a certain quiver variety: begin with the graph with one vertex $0$ and no arrows; take vectors $\bv = (k)$ and $\bw = (\dim \GVect)$; and calculate that 
$$
T^* G(k,\GVect)\cong \mf{M}_1(\bv,\bw), \quad \overline{\mc{O}} = \mf{M}_0(\bv,\bw), \quad  T^* G(k,\GVect^*)\cong \mf{M}_{-1}(\bv,\bw), 
$$
in such a way that the morphisms $g_+$ and $g_-$ are identified with $f_1$ and $f_{-1}$. Moreover, if $\mc{O}_{T^* G(k,\GVect)}(1)$ is the pull-back to $T^* G(k,\GVect)$ of $\mc{O}_{G(k,\GVect)}(1)$ then under the identification $T^* G(k,\GVect)\cong \mf{M}_1(\bv,\bw)$, the line bundle $\mc{O}_{T^* G(k,\GVect)}(1)$ is identified with $\det \mc{R}_0$. Similarly, $\mc{O}_{T^* G(k,\GVect^*)}(1)$ is identified with $\det \mc{R}_0^\prime$, where $\mc{R}_0^\prime$ is the summand of the tautological bundle on $\mf{M}_{-1}(\bv,\bw)$ indexed by vertex 0.

\bigskip
 
  The (real) interior walls are precisely those that are contained in $\gamma^{\perp}$ for $\gamma = m \delta - \alpha$, where $\alpha \in \Phi^+$ and $0 < m < n$. There are uniquely defined chambers $C, C^\prime$ in $F$ such that the wall is $\overline{C} \cap \overline{C^\prime}$. Choose $\theta\in C$, $\theta^\prime\in C^\prime$ and let $\theta_0$ be a generic point of this wall; fix conventions so that $\theta(\gamma) > 0$ and $\theta^\prime(\gamma) < 0$. Let $N$ be the largest positive integer such that $N(N+m) \le n$.

\begin{thm}
\label{thm:MukaiflopAquiver}
 The morphisms $f_{\theta} \colon\mf{M}_{\theta} \rightarrow \mf{M}_{\theta_0}$ and $f_{\theta'} \colon\mf{M}_{\theta'} \rightarrow \mf{M}_{\theta_0}$ induced by variation of GIT quotient are small contractions. 
\end{thm}

\begin{proof}
First we note by Lemma \ref{lem:leavesinternalwall} that $\mf{M}_{\theta_0} = \bigsqcup_{k = 0}^N \mc{L}_k$ admits a finite stratification by smooth locally closed subvarieties, where the codimension of $\mc{L}_k$ is $2k(k+m)$. Choose $x \in  \mc{L}_k \subset \mf{M}_{\theta_0}$. The representation type of $v$ labelled by $x$ is $(1,\beta^{(\infty)};k, \beta^{(0)})$, where $\beta^{(\infty)} = v - k\gamma$ and $\beta^{(0)} = \gamma$, so the $\Ext$-graph has only one vertex in this case. Since $p(\gamma) = 0$, there are no loops at this vertex. The dimension vector $\bm$ equals $(k)$ and $\bn = (m + 2k)$ because $-\left(\beta^{(\infty)},\beta^{(0)}\right) = -(v - k\gamma, \gamma) = m+2k$. We have $\ell = p\left(\beta^{(\infty)}\right)  = p(v - k\gamma) = n - k(m+k)$. Finally, the stability condition $\varrho$ is given by $\varrho_0 = \theta(\gamma)$ and $\varrho_{\infty} = \theta (v - k\gamma)$. Similarly, $\varrho_0' = \theta'(\gamma)$ and $\varrho_{\infty}' = \theta' (v - k\gamma)$.  In particular, $\varrho_{\infty} > 0$ and $\varrho_{\infty}' < 0$. Thus, if we choose a vector space $\GVect$ of dimension $m + 2k$, it follows that \'etale locally, the morphism $f_{\theta} \colon \mf{M}_{\theta} \rightarrow \mf{M}_{\theta_0}$ is isomorphic, in a neighbourhood of $x$, to the morphism $\C^{2 \ell} \times T^* G(k,\GVect) \rightarrow \C^{2 \ell} \times \overline{\mc{O}}$ in a neighbourhood of $0$. Similarly,  the morphism $f_{\theta'} \colon \mf{M}_{\theta'} \rightarrow \mf{M}_{\theta_0}$ is isomorphic, in a neighbourhood of $x$, to the morphism $\C^{2 \ell} \times T^* G(k,\GVect^*) \rightarrow \C^{2 \ell} \times \overline{\mc{O}}$ in a neighbourhood of $0$. Since the maps $g_\pm$ in diagram \eqref{eq:MukaiflopA} are small contractions, the result follows. 
\end{proof}

\begin{rem}
 We prove in Corollary~\ref{cor:internalwallflop} below that the rational map $\varphi$ from Theorem~\ref{thm:MukaiflopAquiver} is a flop. In fact, using quite different techniques it is possible to show that these maps are `stratified Mukai flops of type $A$' in the sense of Fu~\cite{Fu06}. We will come back to this in future work.
\end{rem}

The following proposition  will be important later. 

\begin{prop}\label{prop:tautrestrict} 
Let $Z \subset \mf{M}_{\theta}$ denote the unstable locus for the wall containing $\theta_0$, i.e.\ $Z$ is the set of points consisting of $\theta$-stable $\Pi$-modules that are $\theta'$-unstable. Similarly for $Z' \subset \mf{M}_{\theta'}$. Then:
\begin{enumerate}
\item[\one] there is an isomorphism $\mf{M}_{\theta} \smallsetminus Z \cong \mf{M}_{\theta'} \smallsetminus Z';$
\item[\two] under the identification from part \one, the restriction of the tautological bundle $\mc{R}_i$ to $\mf{M}_{\theta} \smallsetminus Z$ is identified with the restriction of the tautological bundle $\mc{R}_i'$ to $\mf{M}_{\theta'} \smallsetminus Z';$ and  
\item[\three] the closed subschemes $Z$ and $Z'$ have codimension at least $m + 1$ in $\mf{M}_{\theta}$ and $\mf{M}_{\theta'}$ respectively. 
\end{enumerate}
\end{prop}
\begin{proof}

First, we claim that $Z = f_{\theta}^{-1}( \mf{M}_{\theta_0} \smallsetminus \mf{M}_{\theta_0}^s)$. It is clear that $f_{\theta}^{-1}(\mf{M}_{\theta_0}^s) \subseteq \mf{M}_{\theta} \smallsetminus Z$. For the opposite inclusion, let $x\in \mf{M}_{\theta_0}\smallsetminus \mf{M}_{\theta_0}^s$ and apply Lemma \ref{lem:leavesinternalwall} when $k=0$ to see that $x$ corresponds to the polystable representation $M_{\infty} \oplus M_0^{\oplus k}$, where $\dim M_0 = \gamma$ and $\dim M_{\infty} = v - k \gamma$, for some $k > 0$. Since $\theta(\gamma) > 0$, and $\theta'(\gamma) < 0$, any point $y \in f_{\theta}^{-1}(x)$ must correspond to a $\theta$-stable representation $N$ fitting into a sequence 
$$
0 \rightarrow N_0 \rightarrow N \rightarrow M_{\infty} \rightarrow 0. 
$$
 Any such point is clearly $\theta'$-unstable, giving $f_{\theta}^{-1}( \mf{M}_{\theta_0} \smallsetminus \mf{M}_{\theta_0}^s)\subseteq Z$ which proves the claim. The isomorphism $\mf{M}_{\theta} \smallsetminus Z \cong \mf{M}_{\theta'} \smallsetminus Z'$ from part \one\ is given by the restriction of $f_{\theta'}^{-1} \circ f_{\theta}$ to $\mf{M}_{\theta} \smallsetminus Z$. Part \two \ follows from part \one\ because both $\mf{M}_{\theta} \smallsetminus Z$ and $\mf{M}_{\theta'} \smallsetminus Z'$ parametrise precisely those $\Pi$-modules of dimension vector $v$ that are simultaneously $\theta$-stable and $\theta'$-stable. For part \three, the codimension of $Z$ is the maximum codimension of the locally closed subsets $f_{\theta}^{-1}(\mc{L}_k)$ for $k > 0$. Since $f_{\theta}$ is semi-small by Theorem~\ref{thm:birationalimage}, this number is at least $m+1$ by Lemma~\ref{lem:leavesinternalwall}. 
\end{proof}

\section{Linearisation map to the movable cone}
We now construct an isomorphism $L_C \colon \Theta_v \rightarrow N^1(X/Y)$ of rational vector spaces for each chamber $C$ in $\Theta_v$. For any two chambers $C,C'$ in the simplicial cone $F$, it is shown that the isomorphisms $L_C$ and $L_{C'}$ are equal. This allows us to describe explicitly the chamber structure of the movable cone of $X$ over $Y$.  
 
\subsection{Birational geometry}
 \label{sec:birationalgeometry}
 Recall the crepant resolution of singularities from Theorem~\ref{thm:Kuznetsov}:
 \[
 f\colon X:= \Hilb^{[n]}(S)\longrightarrow Y:= \Sym^n(\C^2/\Gamma)\cong \C^{2n}/\Gamma_n.
 \]
  
 Let $N^1(X/Y)$ denote the rational vector space of $\Q$-Cartier divisor classes on $X$ up to numerical equivalence, where divisors $D, D^\prime$ are numerically equivalent, denoted $D\equiv D^\prime$, if $D\cdot \ell = D^\prime\cdot \ell$ for every proper curve $\ell\subset X$. Since $Y$ is affine, proper curves in $X$ are precisely curves in $X$ that are contracted by $f$. Equivalently,  $N^1(X/Y):=(\Pic(X/Y)\otimes\Q/\!\!\equiv)$, where line bundles $L, L^\prime$ on $X$ are numerically equivalent if $\deg L\vert_\ell = \deg L^\prime\vert_\ell$ for every proper curve $\ell\subset X$. Given a line bundle $L$ on $X$, we use the same notation $L\in N^1(X/Y)$ for the corresponding numerical class. 
 
 We now introduce several cones in $N^1(X/Y)$. The (relative) movable cone $\Mov(X/Y)\subset N^1(X/Y)$ is the closure of the convex cone generated by the divisor classes $D$ for which the linear system $\vert mD\vert$ has no fixed component for $m\gg 0$. The (relative) nef cone $\Nef(X/Y)\subset N^1(X/Y)$, is the closed cone of divisor classes $D$ satisfying $D\cdot \ell\geq 0$ for every curve $\ell$ contracted by $f$, and the (relative) ample cone $\Amp(X/Y)$ is the interior of $\Nef(X/Y)$. Note that
  \[
 \Amp(X/Y)\subset \Nef(X/Y)\subseteq \Mov(X/Y)\subset N^1(X/Y).
 \]
 
 Suppose now that $f^\prime\colon X^\prime \to Y$ is another projective, crepant resolution. Since $X$ and $X^\prime$ are birational minimal models over $Y$ \cite[Corollary 3.54]{KollarMori}, there is a commutative diagram
   \begin{equation}
\label{eqn:fi}
\xymatrix{
 X \ar[rd]_{f}\ar@{-->}[rr]^{\psi} & & X^\prime \ar[ld]^{f^\prime} \\
  & Y & 
 }
\end{equation}
 where the birational map $\psi\colon X\dashrightarrow X^\prime$ is an isomorphism in codimension-one. Taking the proper transform along $\psi$ enables us to identify canonically the vector space $N^1(X^\prime/Y)$ with $N^1(X/Y)$, and the movable cone  $\Mov(X^\prime/Y)$ with $\Mov(X/Y)$. The ample and nef cones of $X^\prime$ over $Y$, however, depend on curves in $X^\prime$, but by taking the proper transform along $\psi$ we may nevertheless identify them with the cones $\psi^*\Amp(X^\prime/Y)$ and $\psi^*\Nef(X^\prime/Y)$ respectively in $\Mov(X/Y)$. 
 
 We now turn our attention to the quiver varieties for the framed McKay quiver. For any chamber $C\subset \Theta_v$ and any $\theta\in C$, diagram \eqref{eqn:fi} specialises to the commutative diagram 
    \begin{equation}
\label{eqn:ftheta}
\xymatrix{
 X \ar[rd]_{f}\ar@{-->}[rr]^{\psi_\theta} & & \mathfrak{M}_{\theta} \ar[ld]^{f_\theta} \\
  & Y & 
 }
\end{equation}
 where $f_\theta\colon \mathfrak{M}_\theta\rightarrow Y$ is the symplectic resolution obtained from Proposition~\ref{prop:crepres} and Lemma~\ref{lem:singularity}, and where the birational map $\psi_\theta\colon X\dashrightarrow \mathfrak{M}_\theta$ is an isomorphism in codimension-one. We may therefore identify $N^1(\mathfrak{M}_\theta/Y)$ with $N^1(X/Y)$ and $\Mov(\mathfrak{M}_\theta/Y)$ with $\Mov(X/Y)$ by taking the proper transform along $\psi_\theta$. We also identify the ample and nef cones of $\mathfrak{M}_\theta$ over $Y$ with the cones $\psi_\theta^*(\Amp(\mathfrak{M}_\theta/Y))$ and $\psi_\theta^*(\Nef(\mathfrak{M}_\theta/Y))$ respectively in $\Mov(X/Y)$. All of these cones are of top dimension in $N^1(X/Y)$ because $\mathfrak{M}_\theta$ is projective over $Y$.

 \subsection{The linearisation map}
 Let $C \subset \Theta_v$ be any chamber. Recall that for $\theta\in C$, the quiver variety $\mathfrak{M}_\theta$ carries a tautological locally free sheaf $\mc{R}:=\bigoplus_{i\in I} \mc{R}_i$ that depends on the choice of chamber $C$, where $\mc{R}_{\infty}\cong \mc{O}_{\mathfrak{M}_\theta}$ and where $\mc{R}_i$ has rank $n\dim(\rho_i)$ for $i\in I\setminus \{\infty\}$. Define a $\Q$-linear map 
 \[
 L_C \colon \Theta_v \longrightarrow N^1(X/Y)
 \]
 as follows: for integer-valued maps $\eta\colon \Z^{I}\to\Z$, set 
 \[
 L_C(\eta) = \bigotimes_{i \in I} \det(\mc{R}_i)^{\otimes \eta_i},
 \]
  and define $L_C$ in general by extending linearly over $\Q$. The arguments that follow depend only on the choice of $\eta$ up to a positive multiple, so we  may assume without loss of generality that $\eta$ takes only integer values. The line bundle $L_C(\theta)$ descends from the trivial bundle $\mc{O}_{\mu^{-1}(0)}$ linearised by the character $\chi_\theta$, so $L_C(\theta)$ is the ample line bundle $\mathcal{O}_{\mathfrak{M}_\theta}(1)$ obtained from the GIT construction of $\mathfrak{M}_\theta$.  More generally, for any stability parameter $\eta\in \overline{C}$, we have that
 \begin{equation}
 \label{eqn:pullbackO(1)}
 L_C(\eta) = g^*\big(\mathcal{O}_{\mathfrak{M}_{\eta}}(1)\big),
 \end{equation}
 where $g\colon \mathfrak{M}_\theta\to \mathfrak{M}_{\eta}$ is the morphism obtained by variation of GIT quotient. 
 
 \begin{prop}
 \label{prop:LCC}
 Let $n>1$ and let $C\subset \Theta_{v}$ be any chamber. The map $L_C$ is an isomorphism of rational vector spaces that satisfies $L_C(C)=\Amp(\mathfrak{M}_\theta/Y)$ and $L_C(\overline{C}) = \Nef(\mathfrak{M}_\theta/Y)$ for any $\theta\in C$. 
 \end{prop}
 \begin{proof}
 It is well known (see, e.g.\ Etingof--Ginzburg~\cite[equation~(11.1)]{EG}) that $\Gamma_n$ contains $r+1$ symplectic reflections, and hence $\Gamma_n$ has $r+1$ junior conjugacy classes by Kaledin~\cite[Lemma~1.1]{KaledinMcKay}. The main result of Ito--Reid~\cite[Corollary~1.5]{ItoReid96} implies that $H^2(X,\Q)$ has dimension $r+1$. Since $Y$ has rational singularities, we have $H^i(X,\mc{O}_X)=0$ for $i>0$, so the first Chern class $c_1\colon \Pic(X)\to H^2(X,\Z)$ is an isomorphism and hence $N^1(X)\cong H^2(X,\Q)$. Since $Y$ is affine, the dimension of $N^1(X/Y)$ is also equal to $r+1$. In particular, $\Theta_v$ and $N^1(X/Y)$ have the same dimension. 
 
 Suppose for a contradiction that $L_C$ does not have full rank. We claim that the image of $\overline{C}$ under $L_C$ equals that of the boundary $\partial\overline{C}$. Indeed, for the non-obvious inclusion, suppose there exists $\ell\in L_C(\overline{C})\setminus L_C(\partial \overline{C})$. The intersection of the affine subspace $L_C^{-1}(\ell)$ with $\overline{C}$ is a polyhedron $P$ that contains some point of $C$, so $P$ has positive dimension. The boundary of $P$ is $L_C^{-1}(\ell)\cap \partial \overline{C}$, but this is empty by assumption, so $P$ is an affine subspace of positive dimension. In particular, $\overline{C}$ contains an affine subspace of positive dimension, and hence it contains a vector subspace of positive dimension. But this is a contradiction, because the explicit hyperplane arrangement $\mc{A}_v$ from \eqref{eqn:Av} allows no room for the closure $\overline{C}$ of any chamber to contain a nonzero subspace. Therefore,  if $L_C$ does not have full rank, then  $L_C(\overline{C}) = L_C(\partial\overline{C})$ as claimed. If we can deduce from this that $L_C(\partial\overline{C})$ is contained in the boundary of the nef cone, then we obtain a contradiction because $L_C(C)$ is contained in the interior of the nef cone. Thus, to prove that $L_C$ is an isomorphism, we need only prove that $L_C(\zeta)$ is nef but not ample for $\zeta\in \partial\overline{C}$. Suppose otherwise. After replacing $\zeta$ by a positive multiple if necessary, $L_C(\zeta)$ is very ample and \eqref{eqn:pullbackO(1)} implies that $g\colon \mf{M}_{\theta} \to \mf{M}_{\zeta}$ is a closed immersion. Since $\mf{M}_\theta$ and $\mf{M}_\zeta$ have the same dimension, $g$ is an isomorphism. This is absurd because $\mathfrak{M}_\theta$ is smooth whereas $\mathfrak{M}_\zeta$ is singular by Theorem~\ref{thm:smoothmoduli}, so $L_C(\zeta)$ lies in the boundary of $\Nef(\mathfrak{M}_\theta/Y)$ and hence $L_C$ is an isomorphism.

 Since $L_C$ is a linear isomorphism, it identifies $C$ with the interior of a polyhedral cone of full dimension in $\Amp(\mathfrak{M}_\theta/Y)$. Moreover, we proved above that $L_C$ sends the boundary of $\overline{C}$ into the boundary of the nef cone. In particular, the supporting hyperplanes of the closure of the cone $L_C(C)$ must lie in the boundary of the nef cone. This implies $L_C(C)=\Amp(\mathfrak{M}_\theta/Y)$ and hence $L_C(\overline{C}) = \Nef(\mathfrak{M}_\theta/Y)$.
 \end{proof}

 \begin{rem}
The proof of Proposition~\ref{prop:LCC} shows that the rank of the Picard group of $X = \Hilb^{[n]}(S)$ is equal to $1+\rk(\Pic(S))$. If $S$ were projective, this would follow from the main result of Fogarty~\cite{FogartyPic}.
 \end{rem}
 
\begin{cor}
\label{cor:internalwallflop}
 Let $C, C^\prime$ be adjacent chambers in $F$, and let $\theta_0$ be generic in the separating wall $\overline{C}\cap \overline{C'}$. Then for $\theta\in C$ and $\theta'\in C'$, the diagram 
\begin{equation}
\begin{tikzcd}
\label{thm:MukaiflopAModuli}
\mf{M}_{\theta} \ar[rr,"\varphi",dashed] \ar[dr,"f_{\theta}"'] & & \mf{M}_{\theta'} \ar[dl,"f_{\theta'}"] \\
& \mf{M}_{\theta_0} & 
\end{tikzcd}
\end{equation}
involving the maps from Theorem~\ref{thm:MukaiflopAquiver} is a flop. 
\end{cor} 

\begin{proof}
In light of Theorem~\ref{thm:MukaiflopAquiver}, it remains to prove that the proper transform of the $f_\theta$-ample bundle $L_C(\theta)$ is $f_{\theta^\prime}$-antiample. Proposition~\ref{prop:tautrestrict} shows that $\det(\mathcal{R}_i)$ is the proper transform along $\varphi$ of $\det(\mathcal{R}_i^\prime)$ for all $i\in I$. In particular, the linearisation maps for the chambers $C$ and $C^\prime$ agree, i.e.\
\begin{equation}
    \label{eqn:LCLC'agree}
L_C(\eta) = L_{C^\prime}(\eta)\quad  \text{for all }\eta\in \Theta_v.
\end{equation}
 It follows that the proper transform of $L_C(\theta)$ is $L_{C^\prime}(\theta)$.
 
The ample bundle $L_0:=\mathcal{O}_{\mathfrak{M}_{\theta_0}}(1)$ on $\mf{M}_{\theta_0}$ satisfies $L_{C^\prime}(\theta_0) = f_{\theta^\prime}^*(L_0)$ by \eqref{eqn:pullbackO(1)}. It is possible to choose $\theta\in C$ and $\theta^\prime\in C^\prime$ such that $\theta_0 = \frac{1}{2}(\theta+\theta')$, giving 
\[
L_{C^\prime}(\theta)\otimes L_{C^\prime}(\theta') = L_{C^\prime}(\theta + \theta') = L_{C^\prime}(2\theta_0) = f_{\theta}^*(L_0)^{2}.
\]
Proposition~\ref{prop:LCC} implies that the set of curve classes contracted by $f_{\theta^\prime}$ is non-empty. Since $L_{C^\prime}(\theta^\prime)$ and $f_{\theta^\prime}^*(L_0)$ have positive and zero degree respectively on all such curves, it follows that 
\[
L_{C^\prime}(\theta)^{-1} = L_{C^\prime}(\theta^\prime)\otimes  f_{\theta^\prime}^*(L_0)^{-2}
\]
has positive degree on all such curves, so its inverse $L_{C^\prime}(\theta)$ is $f_{\theta^\prime}$-antiample. 
\end{proof}
 
 \begin{proof}[Proof of Theorem~\ref{thm:simplifiedmainintro}]
 Choose any chamber $C$ in $F$ and define $L_F(\theta):= L_C(\theta)$ for $\theta\in \Theta_{v}$. To see that $L_F$ is well-defined, independent of the choice of $C$, let $C^\prime$ be a chamber in $F$ that lies adjacent to $C$. A key fact, established in equation \eqref{eqn:LCLC'agree}, is that the linearisation maps $L_C$ and $L_{C^\prime}$ agree. Repeating this successively for all internal walls in $F$ shows that $L_F$ is well-defined. Each $L_C$ is an isomorphism of vector spaces by Proposition~\ref{prop:LCC}, hence so is $L_F$.

  For each chamber $C$ in $F$, Proposition~\ref{prop:LCC} shows that the restriction $L_F\vert_C = L_C$ identifies $C$ with the ample cone $\Amp(\mf{M}_\theta/Y)$ for $\theta\in C$. In particular, the restriction of $L_F$ to the interior of $F$ respects the wall-and-chamber structure on $F$ and the Mori chamber decomposition on $\Mov(X/Y)$. It remains to show that $L_F$ sends every boundary wall of $F$ to a boundary wall of $\Mov(X/Y)$. For this, let $\theta_0\in \Theta_v$ be generic in a wall of a chamber $C$ that lies in one of the boundary walls of $F$. According to our classification, each boundary wall of $F$ is either real or imaginary, in which case the morphism $f\colon \mf{M}_\theta\to \mf{M}_{\theta_0}$ induced by the line bundle $L_F(\theta_0)$ is a divisorial contraction by Corollary~\ref{cor:imaginarywallconstraction} or Proposition~\ref{prop:realboundary} respectively. Therefore $L_F$ also identifies the boundary of $F$ with that of $\Mov(X/Y)$. This completes the proof.
    \end{proof}

This result immediately implies Corollary~\ref{cor:introCIconjecture}. More specifically, we have established the following: 

\begin{cor}
\label{cor:CIconjecture}
 Every projective crepant resolution $X^\prime$ of $Y$ is of the form $\mathfrak{M}_\theta$ for any $\theta$ in the unique chamber $C$  in $F$ satisfying $L_F(C) = \Amp(X^\prime/Y)$.
\end{cor}

Theorem~\ref{thm:simplifiedmainintro} also provides a new proof for several results from the literature:
\begin{enumerate}
    \item together with Theorem~\ref{thm:counting}, it follows that the number of non-isomorphic projective crepant resolutions of  $\C^{2n}/\Gamma_n$ is given by \eqref{eqn:count}. This agrees with the count of Bellamy~\cite[Equation~(1.B)]{BellamyCounting}.
    \item that $\Mov(X/Y)$ is a simplicial cone. This recovers a special case of \cite[Theorem~4.1]{AW14}. 
\end{enumerate}
 In addition, Theorem~\ref{thm:simplifiedmainintro} leads to a new, purely quiver-theoretic proof of the following result that is due originally to Andreatta--Wi\'{s}niewski~\cite[Theorem~3.2]{AW14} for $n=2$, and to Namikawa~\cite[Lemma~1, Lemma~6]{Namikawa3} for $n>2$. Recall that a divisor class $D$ is \emph{semi-ample} if $kD$ is basepoint-free for some $k\geq 1$.

 \begin{cor}[Andreatta--Wi\'{s}niewski, Namikawa]
 \label{cor:MDS}
 The variety $X=\Hilb^{[n]}(S)$ is a relative Mori Dream Space over $Y$. That is, 
 \begin{enumerate}
 \item[\one] the cone $\Nef(X/Y)$ is generated by finitely many semi-ample line bundles; and
 \item[\two] there are only finitely many small birational models $X=X_0, X_1, \dots, X_k$ of $X$ over $Y$, and 
 \begin{equation}
 \label{eqn:movdecomp}
\Mov(X/Y) = \bigcup_{0\leq i\leq k} \Nef(X_i/Y),
\end{equation}
 where each cone in this description is generated by finitely many semi-ample line bundles.
 \end{enumerate}
 \end{cor}
 \begin{proof}
 The closure $\overline{C}$ of each chamber $C$ in $\Theta_{v}$ is finitely generated because the hyperplane arrangement $\mc{A}_{v}$ is finite, hence so is the cone $\Nef(\mathfrak{M}_\theta/Y)$ for any $\theta\in C$ by Proposition~\ref{prop:LCC}.  Each class in the interior of $\Nef(\mathfrak{M}_\theta/Y)$ is ample and hence semi-ample. For a class $D$  in the boundary of $\Nef(\mathfrak{M}_\theta/Y)$, Proposition~\ref{prop:LCC} implies that $\eta:=L_C^{-1}(D)\in \overline{C}\setminus C$. After multiplying by $\ell>0$ if necessary, equation \eqref{eqn:pullbackO(1)} shows that $L_C(\ell\eta) = g^*(\mc{O}_{\mathfrak{M}_{\ell\eta}}(1))$ for the morphism $g\colon \mathfrak{M}_\theta\to \mathfrak{M}_{\ell\eta}$ obtained by variation of GIT quotient. In particular,  we have that $\ell D=g^*(\mc{O}_{\mathfrak{M}_{\ell\eta}}(1))$ is basepoint-free, so $D$ is semi-ample. This proves part \one\ and, given Corollary~\ref{cor:CIconjecture}, also establishes the final statement of part \two. 
 To obtain the description \eqref{eqn:movdecomp} of $\Mov(X/Y)$, apply the isomorphism $L_F$ to the description
 \[
 F =  \bigcup_{C\subset F}  \overline{C},
 \]
 bearing in mind that $L_F(F) = \Mov(X/Y)$ by Theorem~\ref{thm:simplifiedmainintro} and $L_F(\overline{C}) = L_C(\overline{C}) = \Nef(\mf{M}_\theta/Y)$ for any $\theta\in C$ by Proposition~\ref{prop:LCC}. It remains to note that each birational model $X_i$ is of the form $\mf{M}_\theta$ for $\theta\in L_C^{-1}(\Amp(X_i/Y))$ by Corollary~\ref{cor:CIconjecture}.
 \end{proof}

 \subsection{The movable cone}
Combining Theorem~\ref{thm:simplifiedmainintro} with Lemma~\ref{lem:Fwalls} leads immediately to the  description of the Mori chamber decomposition of the movable cone $\Mov(X/Y)$ given in Theorem \ref{thm:movableintro}. 
 
 \begin{cor}
 The intersection of $\Mov(X/Y)$ with the affine hyperplane $\{L_F(\theta) \mid \theta(\delta)=1\}$ is isomorphic to the decomposition of the fundamental chamber of the $(n-1)$-extended Catalan hyperplane arrangement of $\Phi$ studied in \cite{Athanasiadis04, PS00}. Moreover, the wall-and-chamber decomposition of $\Mov(X/Y)$ is determined completely by this hyperplane arrangement.
 \end{cor}
 \begin{proof}
  For the first statement, it is enough by Theorem~\ref{thm:simplifiedmainintro} to prove that the slice $\Lambda:=\{\theta\in \Theta_v \mid \theta(\delta)=1\}$ in the wall-and-chamber structure of $F$ is isomorphic to the decomposition of the fundamental chamber of the $(n-1)$-extended Catalan hyperplane arrangement of $\Phi$. This is precisely the content of the proof of Theorem~\ref{thm:counting}. The second statement follows from the fact that every chamber in $F$ intersects the slice $\Lambda$.
 \end{proof}

  The result of Theorem~\ref{thm:movableintro} in the special case where $n=2$ and $\Phi$ is of type $A_r$ is due originally to Andreatta--Wi\'{s}niewski~\cite[Theorem~1.1]{AW14}. We now show how to recover their description from ours.
  
  \begin{example}
  \label{exa:n=2Ar}
  Let $n=2$ and suppose that $\Phi$ is of type $A_r$. Let $e_0,\dots e_r$ denote the standard basis of $\Q^{r+1}$. Consider the surjective linear map $T\colon (\Z\oplus R(\Gamma))\otimes_\Z \Q\to \Q^{r+1}$ given by 
\begin{equation}
 \label{eqn:Trhoi}
 T(\rho_i) = \left\{\begin{array}{cr}  -2e_0 & \text{for }i=\infty, \\ e_0-e_1-\dots - e_r & \text{for }i=0, \\
 e_i & \text{ for }1\leq i\leq r.
 \end{array}\right.
 \end{equation}
The kernel of $T$ is spanned by the vector $v=\rho_\infty+2(\rho_0+\cdots + \rho_r)$, so $T$ induces an isomorphism between $\Theta_{v}$ and $\Hom(\Q^{r+1},\Q)$. For $\theta\in \Theta_{v}$, we have $\theta(\delta)\geq 0$ iff $\theta(\rho_\infty)\leq 0$, so $T$ identifies $F=\langle -\rho_\infty,\rho_0,\rho_1,\dots,\rho_r\rangle^\vee$ with the cone $\langle e_0,e_1,\dots, e_r\rangle^\vee$. Moreover, if we write $\alpha_{i,j}:= \rho_i+\rho_{i+1}+\dots + \rho_j$ for  $1\leq i<j\leq r$, then the normal vector $\delta-\alpha_{i,j}$ to each hyperplane in $\mathcal{A}_v$ passing through $F$ is identified with  
 \[
 T(\delta - \alpha_{i,j}) = T\left(-\frac{\rho_\infty}{2} - (\rho_i+\dots +\rho_j)\right) = e_0 - (e_i+\dots + e_j).
 \]
 Therefore the Mori chamber decomposition of $\Mov(X/Y)$ from Theorem~\ref{thm:movableintro} coincides with the decomposition of the cone $\langle e_0,e_1,\dots, e_r\rangle^\vee$ obtained by cutting with the hyperplanes $\big(e_0 - (e_i+\dots + e_j)\big)^\perp$ for $1\leq i<j\leq r$. Thus we recover the result of Andreatta--Wi\'{s}niewski~\cite[Theorem~1.1]{AW14}. 
 \end{example}
 
 \begin{example}
 For $n=4$ and $\Phi$ of type $A_2$, consider again Example~\ref{exa:A2n=4}. The affine slice $\{L_F(\theta) \mid \theta(\delta)=1\}$ of $\Mov(X/Y)$ is obtained by applying the isomorphism $L_F$ to the slice of $F$ shown in Figure~\ref{fig:A2n=4a}; and similarly for the transverse slice of $\Mov(X/Y)$ shown in Figure~\ref{fig:A2n=4b}. 
 \end{example}

 \section{Reflection functors and the Namikawa Weyl group}
 Our results thus far, and specifically Theorem~\ref{thm:simplifiedmainintro}, give an understanding of the quiver varieties $\mf{M}_\theta$ for generic parameters $\theta$ that lie in the simplicial cone $F$. We now use the fact that $F$ is a fundamental domain for the action of the Namikawa Weyl group $W$ to study the moduli spaces $\mf{M}_\theta$ for any generic $\theta\in \Theta_v$.

 \subsection{Reflection functors}
 \label{sec:reflectionfunctors}
 Reflection functors were introduced by Nakajima~\cite{NakajimaWeyl}, but were also studied independently by Lusztig~\cite{LusztigWeylgroup}, Crawley-Boevey--Holland~\cite{CrawleyBoeveyHolland}, and Maffei~\cite{Maffei}.   

Let $C \subset \Theta_{v}$ be a chamber and let $\theta\in C$. Recall that for $1\leq i\leq r$, we write $s_i\colon \Theta_v\to\Theta_v$ for reflection in the hyperplane $\rho_i^\perp$. By Proposition~\ref{prop:W}\three, the action of the Namikawa Weyl group $W$ permutes the chambers in $\Theta_v$, so $s_i(C)$ is a chamber containing the parameter $s_i(\theta)$. The following result is a special case of a general result of Losev~\cite[Lemma~6.4.2]{Losev}.

\begin{lem}[Losev]
For $1\leq i\leq r$, the reflection functor $s_i$ induces an isomorphism $S_i \colon \mf{M}_{s_i(\theta)} \rightarrow \mf{M}_{\theta}$ of schemes over $Y=\mf{M}_0$. 
\end{lem}

Recall that for $j\in I$, we let $\mc{R}_j$ denote the corresponding tautological bundle on the fine moduli space $\mf{M}_\theta$. For simplicity, we write $\mc{R}_j^\prime$ for the corresponding tautological bundle on $\mf{M}_{s_i(\theta)}$.

\begin{lem}\label{lem:bundlesreflect}
	Let $1\leq i\leq r$ and $\theta\in C$. The reflection isomorphism $S_i \colon \mf{M}_{s_i(\theta)} \stackrel{\sim}{\longrightarrow} \mf{M}_{\theta}$ satisfies
	\begin{equation}
	    \label{eqn:siLCeta}
	S_i^* \left(\bigotimes_{j \in I} \det (\mc{R}_j)^{\otimes \eta_j} \right) \cong \bigotimes_{j \in I} \det (\mc{R}'_j)^{\otimes s_i(\eta)_j}
		\end{equation}
	for all $\eta\in \Theta_v$. In particular, the linearisation maps for $C$ and $s_i(C)$ fit into a commutative diagram
	 \begin{equation*}
\begin{tikzcd}
  C \ar[r,"s_i"]\ar[d,"{L_C}"] & s_i(C) \ar[d,"{L_{s_i(C)}}"] \\
 \Amp(\mf{M}_\theta/Y)\ar[r,"S_i^*"] & \Amp(\mf{M}_{s_i(\theta)}/Y).
 \end{tikzcd}
\end{equation*}
\end{lem}

\begin{proof}
	For simplicity, we follow the set-up described in \cite[\S 3]{Maffei}. We let $Z := Z_i^{\theta} \git \, G_{i,v}$, where $Z_i^{\theta}$ and $G_{i,v}$ are defined in \cite[\S 3]{Maffei}; note that $\theta$ is denoted $m$ there. We may assume without loss of generality that $\theta_i > 0$. Then there are explicit isomorphisms
	$$
    \begin{tikzcd}
	\mf{M}_{s_i(\theta)} &  Z \ar[l,"\sim"',"p^\prime"] \ar[r,"\sim","p"'] & \mf{M}_{\theta}.
	\end{tikzcd}
    $$
	Moreover, statements in and preceding \cite[Definition~27]{Maffei} imply that for indices $j \neq i$ there is an isomorphism $p^* (\mc{R}_{j}) \cong (p')^* (\mc{R}_{j}')$, while for the index $i$ there is a short exact sequence 
	$$
	0 \longrightarrow (p')^* (\mc{R}_{i}') \longrightarrow \bigoplus_{\head(a) = i} (p')^* (\mc{R}_{\tail(a)}') \longrightarrow p^* (\mc{R}_{i}) \longrightarrow 0
	$$
	of bundles on $Z$, where the sum is over arrows $a \in H$; here we use the isomorphisms $p^* (\mc{R}_{\tail(a)}) \cong (p')^* (\mc{R}_{\tail(a)}')$ for arrows with $\head(a)=i$ and the fact that the McKay quiver $Q$ has no loops. Taking determinants, we have $p^* (\det \mc{R}_{j}) \cong (p')^* (\det \mc{R}_{j}')$ for indices $j \neq i$, while for the index $i$ we have  
	$$
	p^* (\det \mc{R}_{i}) \cong (p')^* \left( (\det \mc{R}_{i}')^{\otimes -1}\otimes \bigotimes_{\head(a) = i} \det \mc{R}_{\tail(a)}' \right).
	$$
	Since $S_i := p \circ (p')^{-1}$, the left hand side of equation \eqref{eqn:siLCeta} is therefore
	\begin{equation}
	    \label{eqn:LHS7.1}
	\left( (\det \mc{R}_{i}')^{\otimes -\eta_i}\otimes \bigotimes_{\head(a) = i} \det (\mc{R}_{\tail(a)}')^{\otimes \eta_i} \right)\otimes \bigotimes_{j\neq i} \det(\mc{R}_j^\prime)^{\otimes \eta_j}. 
		\end{equation}
	Recall from equation \eqref{eqn:sieta} that $s_i(\eta)_j = \eta_j - c_{i,j}\eta_i$ for any $\eta\in \Theta_v$, so the right hand side of \eqref{eqn:siLCeta} is 
 \[
  \bigotimes_{j \in I} \det (\mc{R}'_j)^{\otimes (\eta_j - c_{i,j}\eta_i)}  \cong  \det (\mc{R}'_i)^{\otimes - \eta_i}\otimes \bigotimes_{\head(a) = i} \det(\mc{R}_{\tail(a)}')^{\otimes (\eta_j + \eta_i)}\otimes \bigotimes_{j\in J}\det (\mc{R}'_j)^{\otimes \eta_j}, 
 \]
 where $J = \{j\in I \mid j\neq i, j\neq \tail(a) \text{ for some }a\in Q_1 \text{ with }\head(a)=i\}$. This equals \eqref{eqn:LHS7.1} as required. It now follows directly from the definition that $ S_i^*(L_C(\eta)) \cong L_{s_i(C)}(s_i(\eta))$ for all $\eta\in \Theta_v$. The final statement of the lemma follows by considering this isomorphism for $\eta\in C$ and using the equalities $L_C(C) = \Amp(\mf{M}_\theta/Y)$ from Proposition~\ref{prop:LCC} for the chambers $C$ and $s_i(C)$.
\end{proof}

 The orbit of the cone $F$ under the action of the subgroup of $W$ generated by the reflections $s_1,\dots, s_r$ covers the half-space $\{\theta\in \Theta_v \mid \theta(\delta)\geq 0\}$. By combining the results from this section with Theorem~\ref{thm:simplifiedmainintro}, it follows that we now understand the moduli spaces $\mf{M}_\theta$ and their tautological bundles for all generic parameters $\theta$ that lie in this halfspace.

\subsection{Crossing the hyperplane $\delta^\perp$}
\label{sec:graphinvolution}
Ideally, we would like to reflect at the vertex $0$ as well, but this is not possible since $v$ is not fixed by $s_0$. In order to study the moduli $\mf{M}_\theta$ for $\theta$ in the half-space $\{\theta\in \Theta_v \mid \theta(\delta)<0\}$, we instead use the fact that the preprojective algebra is isomorphic to its dual. 

 Let $\iota$ denote the involution on $\Irr(\Gamma)$  given by $\rho \mapsto \rho^*$. Since $V$ is self-dual, this is a graph automorphism. It is described as follows: if $\Gamma$ is of type $A_n$ ($n>1$), $D_n$ ($n$ odd) or $E_6$, then $\iota$ is the involution of $\Irr(\Gamma)$ induced by the order 2 symmetry of the McKay graph described in \cite[item (XI) in Planche~I-VII]{BourbakiLie}; otherwise, $\Gamma$ is of type $A_1$, $D_n$ ($n$ even), $E_7, E_8$ and $\iota$ is the identity. Let $H_{\mathrm{aff}} \subset H$ denote the set of oriented edges of the unframed affine Dynkin graph. Then $\iota$ is uniquely defined on $H_{\mathrm{aff}}$ by the rule 
$$
\tail(\iota(a)) = \iota(\head(a)), \quad \head(\iota(a)) = \iota(\tail(a)). 
$$
This defines an anti-involution of the path algebra associated to $H_{\mathrm{aff}}$. Since $\iota(a)^* = \iota(a^*)$, the anti-involution $\iota$ descends to an anti-involution $\iota \colon \Pi_{\mathrm{aff}} \stackrel{\sim}{\longrightarrow} \Pi^{\mathrm{op}}_{\mathrm{aff}}$ of the preprojective algebra corresponding to $H_{\mathrm{aff}}$. We extend $\iota$ to an anti-involution $\iota \colon \Pi \stackrel{\sim}{\longrightarrow} \Pi^{\mathrm{op}}$ of the preprojective algebra for the framed McKay graph as follows: set $\iota(\rho_{\infty}) = \rho_{\infty}$, and for the arrows $u \colon \rho_{\infty} \rightarrow \rho_0$ and $u^* \colon \rho_0 \rightarrow \rho_{\infty}$ set $\iota(u) = u^*$ and $\iota(u^*) = u$. On dimension vectors, we define $\iota(\alpha)$ by $\iota(\alpha)_i = \alpha_{\iota(i)}$, and on stability conditions, define $\iota(\theta)_i := \theta_{\iota(i)}$. 

Let $A:= \C[V] \rtimes \Gamma$ be the skew-group algebra. We define an analogous anti-involution $\nu \colon A \stackrel{\sim}{\longrightarrow} A^{\mathrm{op}}$ by setting $\nu(x) = x$ and $\nu(g) = g^{-1}$ for $x \in V^* \subset \C[V]$ and $g \in G$. For each $0\leq i\leq r$, choose an idempotent $f_i \in \C \Gamma$ such that $(\C \Gamma) f_i \cong \rho_i$.  

\begin{lem}
The idempotents $f_0,\dots,f_r$ can be chosen so that $\nu(f_i) = f_{\iota(i)}$. In particular, $f := f_0 + \cdots + f_r$ is invariant under $\nu$. 
\end{lem}

\begin{proof}
If $\iota(i) \neq i$, simply choose $f_{\iota(i)} = \nu(f_i)$. If $\iota(i) = i$, then $\nu$ restricts to an anti-involution of the block of $\C \Gamma$ corresponding to $\rho_i$. Since this block is just matrices of size $\delta_i$ over $\C$, it is straight-forward to check that one can choose a primitive idempotent fixed by $\nu$ (every anti-automorphism of a matrix algebra is conjugate to the transpose by the Skolem--Noether Theorem).
\end{proof}

Work of Crawley-Boevey--Holland~\cite[Theorem~3.4]{CrawleyBoeveyHolland} establishes an isomorphism $f A f\cong \Pi_{\mathrm{aff}}$, so $f$ defines a Morita equivalence $A \sim \Pi_{\mathrm{aff}}$. 

\begin{lem}\label{lem:fpsiphicompat}
The isomorphism $f A f \cong \Pi_{\mathrm{aff}}$ can be chosen to identify the anti-involution $\nu |_{f A f}$ with $\iota$. 
\end{lem}
\begin{proof}
The proof of this lemma is rather lengthy since we must show that the isomorphism in the ``Key Lemma'' \cite[Lemma 3.2]{CrawleyBoeveyHolland} can be made compatible with the anti-involutions. First, we note that the case where $\Gamma$ is of type $A$ can be checked explicitly by hand, so we assume that $\Gamma$ is of type $D$ or $E$. In this case, the definition of $\iota$ implies that there is no arrow $a$ in $H_{\mathrm{aff}}$ such that $a \colon i \rightarrow \iota(i)$ for some vertex $i$, and that $\iota(a) \neq a$ for all $a$. In particular, we can choose an orientation $\Omega_{\mathrm{aff}} \subset H_{\mathrm{aff}}$ of the McKay graph so that $\epsilon(\iota(a)) = - \epsilon(a)$ for all arrows $a$.

Next, as in \cite{CrawleyBoeveyHolland}, we let $C_1 = V \otimes_{\C} \C \Gamma$ be the $\Gamma$-bimodule with diagonal left action and right action on the right factor only. Let $C_2 = C_1 \o_{\Gamma} C_1$. Then, 
$$
f C_1 f = \bigoplus_{i,j \in Q_0} f_j C_1 f_i,
$$
with $\dim f_j C_1 f_i = \Hom_\Gamma(\rho_j,V\otimes \rho_i) = 1$ if there is an arrow in $H_{\mathrm{aff}}$ from $i$ to $j$, and zero otherwise. The ``Key Lemma'' of \cite{CrawleyBoeveyHolland}, rephrased as \cite[Lemma~3.3]{CrawleyBoeveyHolland}, exhibits a basis $\theta_a$ of $f_{\head(a)} C_1 f_{\tail(a)}$ such that 
$$
\sum_{\stackrel{a \in Q_1}{\tail(a) = i}} \epsilon(a) \theta_{a^*} \theta_a = \delta_i f_i (xy - yx)
$$
in $C_2$, where $x,y$ is the standard basis of $V$ and where $\delta_i = \dim_{\C} \rho_i$. Now, Lemma~\ref{lem:fpsiphicompat} and the fact that $\nu$ is an anti-involution implies that $\nu(\theta_a) = m_a \theta_{\iota(a)}$ for some $m_a \in \C$ with $m_{\iota(a)} = m_a^{-1}$. Next, 
$$
f_i C_2 f_i = \bigoplus_{j \neq i} f_i C_1 f_j C_1 f_i
$$
and we can (uniquely) decompose $\delta_i f_i (xy - yx) = \sum_{j \neq i} d_{i,j}$, with $d_{i,j} \in f_i C_1 f_j C_1 f_i$. Since 
$$
\nu(\delta_i f_i (xy - yx)) = - \delta_{\iota(i)} f_{\iota(i)} (xy - yx)
$$
we have $\nu(d_{i,j}) = - d_{\iota(i),\iota(j)}$. Since there is at most one arrow from $i$ to $j$ in $H_{\mathrm{aff}}$, we deduce that either $d_{i,j} = \epsilon(a) \theta_{a^*} \theta_{a}$ or $d_{i,j} = 0$. This implies that 
$$
- \epsilon(\iota(a)) \theta_{\iota(a)} \theta_{\iota(a)^*} = - d_{\iota(i),\iota(j)} = \epsilon(a) \nu(\theta_{a}) \nu (\theta_{a^*}) =  \epsilon(a) m_a m_{a^*}\theta_{\iota(a)} \theta_{\iota(a)^*},
$$
 and hence $m_{a^*} = m_a^{-1}$. 

The fact that $\iota(a)\neq a$ for all $a$ implies that either 
\begin{enumerate}
\item[(i)] $a^* = \iota(a)$ (when $\iota(\tail(a)) = \tail(a)$ and $\iota(\head(a)) = \head(a)$), in which case we let $t_a$ be a square root of $m_a$ and set $t_{\iota(a)} = t^{-1}_a$; or
\item[(ii)] $a,a^*,\iota(a),\iota(a)^*$ are all distinct, in which case let $t_a$ again be a square root of $m_a$ and set 
$t_a = t_{\iota(a)^*} = t_{a^*}^{-1} = t_{\iota(a)}^{-1}$. 
\end{enumerate}
Finally, define $\widetilde{\theta}_a := \frac{1}{t_a}\theta_a$. Then 
$$
\nu(\widetilde{\theta}_a ) = \frac{1}{t_a}\nu(\theta_a) = \frac{1}{t_a}m_a \theta_{\iota(a)} = \widetilde{\theta}_{\iota(a)},
$$
and $\theta_a \theta_{a^*} = \widetilde{\theta}_a \widetilde{\theta}_{a^*}$. Combining this with the proof of \cite[Theorem~3.4]{CrawleyBoeveyHolland}, we obtain a well-defined isomorphism $\Pi_{\mathrm{aff}} \rightarrow f A f$ sending $\rho_i$ to $f_i$ and $a$ to $\widetilde{\theta}_a$ that is compatible with the anti-involutions $\iota$ and $\nu |_{f A f}$. 
\end{proof}

Let $\Lmod{\Pi_{\mathrm{aff}}}$ and $\Lmod{A}$ denote the categories of finite dimensional left $\Pi_{\mathrm{aff}}$-modules and $A$-modules respectively. There is a contravariant equivalence \[
D \colon \Lmod{\Pi_{\mathrm{aff}}} \longrightarrow \Lmod{\Pi_{\mathrm{aff}}}
\]
 satisfying $D(M) = M^*=\Hom_{\C}(M,\C)$, where the $\Pi_{\mathrm{aff}}$-module structure satisfies $(a \cdot \lambda)(m) = \lambda(\iota(a) \cdot m)$ for $a \in \Pi_{\mathrm{aff}}, m \in M$ and $\lambda \in M^*$. Similarly, there is a second contravariant equivalence $\mathbb{D}\colon \Lmod{A} \to \Lmod{A}$, and Lemma~\ref{lem:fpsiphicompat} implies that $f \circ \mathbb{D} \cong D \circ f$ as functors $\Lmod{A} \rightarrow \Lmod{\Pi_{\mathrm{aff}}}$.

\begin{lem}\label{lem:Ddualityonfunctions}
 Let $\theta \in \Theta_{v}$. The functor $D$ induces an isomorphism $D^* \colon \C[\mu^{-1}(0)]^{\theta} \stackrel{\sim}{\longrightarrow} \C[\mu^{-1}(0)]^{- \iota(\theta)}$, where semi-invariants are taken with respect to the actions of the groups $G(v)$ and $G(\iota(v))$ on the domain and codomain of $D^*$.
\end{lem}
\begin{proof}
Let $M = (V_i,\psi_a \mid i \in I, a \in H)$ be a point in $\mu^{-1}(0)$. Then $D(M)= (V_i^*,\psi_{\iota(a)}^T)$. We now define $\iota \colon G(v) \stackrel{\sim}{\longrightarrow} G(\iota(v))$ by $\iota(g)_i = g_{\iota(i)}$. Then $D(g \cdot M) = (\iota(g)^{-1})^T D(M)$. Hence, if $f \in \C[\mu^{-1}(0)]^{\theta}$ then 
\begin{align*}
(g \cdot D^*(f))(M) & = D^*(f)(g^{-1} \cdot M) \\
 & = f(D(g^{-1} \cdot m)) = f(\iota(g)^T \cdot D(M)) \\
 & = ((\iota(g)^{-1})^T \cdot f)(D(M)) = (- \iota(\theta))(g) f(D(M)), 
\end{align*}
as required. 
\end{proof}

\begin{prop}
There is an isomorphism $D \colon \mf{M}_{- \iota(\theta)} \rightarrow \mf{M}_{\theta}$ of schemes over $Y=\mf{M}_0$. 
\end{prop}
\begin{proof}
Since $\iota(v) = v$, Lemma~\ref{lem:Ddualityonfunctions} gives an isomorphism $D \colon \mf{M}_{-\iota(\theta)} \rightarrow \mf{M}_{\theta}$ of $\C$-schemes. It remains to check that the isomorphism $D\colon \mf{M}_0 \rightarrow \mf{M}_0$ is the identity. We can do this by looking at closed points. Thus, it is enough to show that if $M$ is a semi-simple $\Pi$-module of dimension $v$ then $D(M) \cong M$. By Theorem~\ref{thm:canondecomp}\four\ and Proposition~\ref{prop:canonicaldecomp}, we have $M = M_{\infty} \oplus M_1 \oplus \cdots \oplus M_n$, where $\dim M_{\infty} = \rho_{\infty}$ and $\dim M_i = \delta$ for all $i \ge 1$. If $L_i$ is the simple $\Pi$-module of dimension $\rho_i$, then $D(L_i) \cong L_{\iota(i)}$ and hence $D(M_{\infty}) \cong M_{\infty}$. Therefore, we need only show that if $M$ is semi-simple of dimension $\delta$, then $D(M) \cong M$. Note that the action of $\Pi$ on $M$ now factors through $\Pi_{\mathrm{aff}}$, so we may replace the former by the latter. Since simple $\Pi_{\mathrm{aff}}$-modules have dimension vector corresponding to positive affine roots, there are two cases to consider: either $M \cong \bigoplus_{0\leq i \leq r} L_i^{\oplus \delta_i}$; or $M$ is simple. In the former case, $D(L_i) \cong L_{\iota(i)}$ and $\delta_i = \delta_{\iota(i)}$ implies that $D(M) \cong M$ as required. Otherwise, we may assume that $M$ is simple. Under the Morita equivalence $A \sim \Pi_{\mathrm{aff}}$, we have $M = f N$, where $N$ is a simple $A$-module such that $N |_{\Gamma} \cong \C \Gamma$. Since $\Gamma$ acts freely on $V \smallsetminus \{ 0 \}$, $N$ is uniquely defined up to isomorphism by the maximal ideal $\mathrm{ann}_Z N$, where $Z := Z(A) = \C[V]^{\Gamma}$. But, if $z \in Z$, $\lambda \in N^*$ and $n \in N$, then 
$$
(z \cdot \lambda)(n) = \lambda(\nu(z) n) = \lambda(zn) = 0
$$
because $\nu$ is the identity on $\C[V]^{\Gamma}$. Therefore, $\mathrm{ann}_Z N = \mathrm{ann}_Z \mathbb{D}(N)$ and hence $N \cong \mathbb{D}(N)$. 
\end{proof}

We now compute what happens to the tautological bundles under $D$. For $i\in I$, write $\mc{R}_i$ and $\mc{R}_i^\prime$ for the corresponding tautological bundles on $\mf{M}_{\theta}$ and $\mf{M}_{-\iota(\theta)}$ respectively.

\begin{lem}
\label{lem:Dbundles}
For all $i\in I$, we have $D^*(\det \mc{R}_i) \cong (\det \mc{R}_{\iota(i)}')^{-1}$ on $\mf{M}_{- \iota(\theta)}$.
\end{lem}

\begin{proof}
Choose a non-zero semi-invariant $f \in \C[\mu^{-1}(0)]^{-m \theta}$ for some $m > 0$, and let $U \subset \mf{M}_{\theta}$ be the affine open subset that is complementary to the zero-locus of $f$. Then 
$$
\Gamma(U,\det \mc{R}_i) = \C \left[ \frac{h}{f^l} \  \Big| \ h \in \C[\mu^{-1}(0)]^{-\chi_i - l m \theta}\right],
$$
and hence, by Lemma \ref{lem:Ddualityonfunctions}, we have
$$
\Gamma(D^{-1}(U),D^*(\det \mc{R}_i)) = \C \left[ \frac{g}{D^*(f)^l} \  \Big| \ g \in \C[\mu^{-1}(0)]^{\chi_{\iota(i)} + l m \iota(\theta)}\right].
$$
This equals $\Gamma(D^{-1}(U),(\det \mc{R}_{\iota(i)}')^{-1})$ since $\chi_{\iota(i)} + l m \iota(\theta) = -(-\chi_{\iota(i)}) - l m (-\iota(\theta))$.  
\end{proof}

\begin{remark}
\label{rem:Gammatrivial}
 In fact all of the results of this section are valid even when $\Gamma$ is trivial if we set $\iota$ to be the identity. This gives rise to an isomorphism $D\colon \mf{M}_{-\theta}\rightarrow \mf{M}_{\theta}$ over the base $\Sym^n(\C^2)$. In this case, there are only two GIT chambers and it is well known that $\mf{M}_{\theta}$ is the Hilbert scheme of $n$-points on $\C^2$.
 \end{remark}

\subsection{The main results}
\label{sec:lastsection}
 We can now prove the main result which gives an understanding of the moduli spaces $\mf{M}_\theta$ for any generic $\theta\in \Theta_v$.  It is convenient to first establish a compatibility result for the linearisation maps associated to the chambers $C$ and $w(C)$ for any $w\in W$; here, $w(C)$ is a chamber by Proposition~\ref{prop:W}\three.
 
 \begin{lem}
 \label{lem:penultimate}
 For any chamber $C\subset \Theta_v$, we have that $L_{w(C)}(w(\theta)) = L_C(\theta)$ for all $w\in W$ and $\theta\in \overline{C}$.
 \end{lem}
 \begin{proof}
  The birational map $\psi_\theta\colon X\dashrightarrow \mf{M}_\theta$ of schemes over $\mf{M}_0$ from section~\ref{sec:birationalgeometry} is unique; it's determined by the linear system $\vert L_{C_-}(\ell\theta)\vert$ for sufficiently large $\ell$, where $C_-$ is the unique chamber in $F$ satisfying $X\cong\mf{M}_\theta$ for $\theta\in C_-$ (see Theorem~\ref{thm:Kuznetsov}). For clarity, in the course of this proof we choose not to suppress making reference to these maps when identifying nef cones and line bundles, so in fact our goal is to prove that
  \begin{equation}
      \label{eqn:LwC}
  \psi_{w(\theta)}^*L_{w(C)}(w(\theta)) = \psi_\theta^*L_C(\theta)
  \end{equation}
  for all $w\in W$ and $\theta\in \overline{C}$. 
  
  If $w_0$ is the longest element in $W_{\Gamma}$, with respect to the set of simple reflections $\{ s_1, \ds, s_r \}$, then we deduce from Bourbaki~\cite[item (XI) in Planche~I-VII]{BourbakiLie} that $w_0$ acts on the hyperplane in $R(\Gamma)$ spanned by $\rho_1, \dots,\rho_r$ as multiplication by $-\iota$; here $\iota$ is the involution of the Dynkin diagram introduced at the beginning of section~\ref{sec:graphinvolution}. The element $s_{\delta}$ fixes all vectors in this hyperplane. Therefore to show that $s_{\delta} w_0 = w_0 s_{\delta}$ acts as $- \iota$ on $R(\Gamma)$ it suffices to check that $ w_0 s_{\delta}(\rho_0) = - \rho_0$. Let $\beta \in \Phi^+$ be the highest root. Then 
  $$
   w_0 s_{\delta}(\rho_0) =  w_0(\rho_0 - 2\delta) = w_0(-\delta - \beta) = - \delta + \beta = -\rho_0
   $$
   since $w_0(\beta) = - \beta$. By definition, $s_{\delta} w_0$ also acts on $\Theta_{v}$ as multiplication by $-\iota$. Since $W$ is generated by $\{ s_1, \ds, s_r \}$ and $s_{\delta} w_0$, it suffices to check that equation \eqref{eqn:LwC} holds for $w=s_i$ for $1\leq i\leq r$, and for $s_{\delta} w_0$. In the first case, for $1\leq i\leq r$ we have $\psi_\theta = S_i\circ \psi_{s_i(\theta)}$, so
  \[
  \psi_\theta^*L_C(\theta) = \psi_{s_i(\theta)}^*\big(S_i^*(L_C(\theta)\big) = \psi_{s_i(\theta)}^*L_{s_i(C)}(s_i(\theta))
  \]
   by Lemma~\ref{lem:bundlesreflect}. In the second case, we have $\psi_\theta = D\circ \psi_{-\iota(\theta)}$, so
  \[
  \psi_\theta^*L_C(\theta) = \psi_{-\iota(\theta)}^*\big(D^*(L_C(\theta)\big) = \psi_{-\iota(\theta)}^*L_{-\iota(C)}(-\iota(\theta))
  \]
   by Lemma~\ref{lem:Dbundles}. Therefore equation \eqref{eqn:LwC} holds as required.
 \end{proof}

We are finally in a position to prove the strong form of our main result.

\begin{proof}[Proof of Theorem~\ref{thm:mainintro}]
 For \one, define $L$ by setting $L(\theta):= L_C(\theta)$ for $\theta\in \Theta_v$, where $C$ is any chamber satisfying $\theta\in \overline{C}$. To see that $L$ is well-defined, we need only show that for adjacent chambers $C, C^\prime\subset \Theta_{v}$, the maps $L_C, L_{C^\prime}$ agree on the separating wall, i.e.\ that $L_C(\theta_0) = L_{C^\prime}(\theta_0)$ for all $\theta_0\in \overline{C}\cap \overline{C^\prime}$ (we do not assume $\theta_0$ is generic in the wall). Since $F$ is a fundamental domain for $W$, we may assume without loss of generality by Lemma~\ref{lem:penultimate} that $C \subset F$. There are three cases to consider. First, if $C^\prime$ also lies in $F$, then the proof of Theorem~\ref{thm:simplifiedmainintro} shows that $L_C(\eta) = L_{C^\prime}(\eta)$ for all $\eta\in \Theta_v$.  Second, if the wall separating $C$ from $C^\prime$ is a real boundary wall of $F$, then it is contained in a hyperplane $\rho_i^\perp$ for some $1 \le i \le r$. In this case, $C^\prime=s_i(C)$ and $\theta_0 = s_i(\theta_0)$ for all $\theta_0 \in \overline{C}\cap \overline{C'}$, giving $L_{C'}(\theta_0) = L_{s_i(C)}(s_i \theta) = L_C(\theta)$, as required. Finally, if the wall separating $C$ from $C^\prime$ is an imaginary boundary wall, then $C' = s_{\delta}(C)$ and, as above, we deduce that $L_{C'}(\theta_0) = L_C(\theta_0)$. Therefore $L$ is a well-defined, piecewise-linear map. 

For \two, let $C$ be any chamber such that $\theta\in \overline{C}$. Then $w(\theta)\in \overline{w(C)}$ and hence Lemma~\ref{lem:penultimate} gives
 \[
 L(w(\theta)) = L_{w(C)}(w(\theta)) = L_C(\theta) = L(\theta)
 \]
 which establishes $W$-invariance. We have $L\vert_F = L_F$ by construction, so $L(F) = \Mov(X/Y)$ by Theorem~\ref{thm:simplifiedmainintro}, and since $F$ is a fundamental domain for $W$, the $W$-invariance of $L$ now implies that $L(\Theta_v) = \Mov(X/Y)$. For \three, observe first that $L\vert_F$ is compatible with the chamber decompositions of $\Theta_v$ and $\Mov(X/Y)$ by Theorem~\ref{thm:simplifiedmainintro}. Part \two\ now implies that $L\vert_{w(F)}$ is compatible with the chamber decompositions for each $w\in W$, and the statement of part \three\ follows from Proposition~\ref{prop:W}\two. Part \four\ follows from Proposition~\ref{prop:LCC}, because the restriction of $L$ to $C$ equals $L_C$.
 \end{proof}

 \begin{cor}
 Let $C, C^\prime\subset \Theta_{v}$ be chambers and let $\theta \in C$,  $\theta^\prime \in C^\prime$. Then $\mathfrak{M}_\theta\cong \mathfrak{M}_{\theta^\prime}$ as schemes over $Y$ if and only if there exists $w \in W$ such that $w(C) = C^\prime$. 
 \end{cor}
 \begin{proof}
  If $\mathfrak{M}_\theta\cong \mathfrak{M}_{\theta^\prime}$ as schemes over $Y$ then $L(C) = \Amp(\mathfrak{M}_{\theta^\prime}/Y) = L(C^\prime)$ by Theorem~\ref{thm:mainintro}\four. Since $F$ is a fundamental domain for the action of $W$ on $\Theta_{v}$, there exists $w_1, w_2\in W$ such that $w_1(C), w_2(C^\prime)\subset F$. Theorem~\ref{thm:mainintro}\one\ implies that
  \[
  L(w_1(C)) = L(C)=L(C^\prime) = L(w_2(C^\prime)).
  \]
   The map $L\vert_F=L_F$ identifies the chamber decomposition of $F$ with that of $\Mov(X/Y)$ by Theorem~\ref{thm:simplifiedmainintro}, so $w_1(C)=w_2(C^\prime)$ and hence $w:=w_2^{-1}w_1$ satisfies $w(C)=C^\prime$. For the converse, if there exists $w\in W$ such that $w(\theta)\in C^\prime$, then Theorem~\ref{thm:mainintro}\one\ implies that $L(\theta) = L(w(\theta))\in L(C^\prime) = \Amp(\mathfrak{M}_{\theta^\prime}/Y)$ and hence $\mathfrak{M}_\theta\cong \mathfrak{M}_{\theta^\prime}$ as schemes over $Y$ by Theorem~\ref{thm:mainintro}\four. 
 \end{proof}

 For completeness, we now present the analogues of Proposition~\ref{prop:LCC} and Theorem~\ref{thm:mainintro} in the degenerate case when $n=1$. Note that $X\cong S$ and $Y\cong \C^2/\Gamma$. 
 
\begin{prop}
\label{prop:n=1}
 Let $n=1$. 
 \begin{enumerate}
     \item[\one] For each chamber $C$ in $\Theta_v$, the linearisation map $L_C\colon \Theta_v\to N^1(S/(\C^2/\Gamma))$ is surjective, the kernel is spanned by $(-1,1,0,\dots,0)$, and it satisfies $L_C(C) = \Amp(\mf{M}_\theta/Y)$ for $\theta\in C$. 
     \item[\two] These maps glue to give a piecewise-linear, continuous map $L\colon \Theta_{v}\longrightarrow N^1\big(S/(\C^2/\Gamma))$ that is invariant with respect to the action of $W$ on $\Theta_v$, and whose image is $\Nef(S/(\C^2/\Gamma))$. 
 \end{enumerate} 
\end{prop}
\begin{proof}
 For \one, suppose first that $C$ lies in the half-space $\{\theta \in \Theta_v\mid \theta(\delta)>0\}$, so each $\theta\in C$ satisfies $\theta_\infty<0$. If $a\in Q_1$ is the unique arrow with tail at $\rho_\infty$, then the relation $a^*a = 0$ in $\Pi$ ensures that any $\theta$-stable point $(V_i,\psi_a \mid i \in I, a \in Q_1)$ in $\mu^{-1}(0)$ satisfies $\psi_a\neq 0$ and $\psi_{a^*}=0$. Therefore, we have a nowhere-zero morphism $\mc{R}_{\rho_\infty}\to \mc{R}_{\rho_0}$ which is necessarily an isomorphism. It follows that $\kappa:= (-1,1,0,\dots,0)$ lies in the kernel of $L_C$. Gonzalez-Sprinberg--Verdier~\cite{GSV83} implies that the line bundles $\det(\mc{R}_i)$ for $1\leq i\leq r$ provide an integral basis of $N^1(S/(\C^2/\Gamma))$, so $L_C$ is surjective and hence $\kappa$ spans the kernel of $L_C$. It remains to note that the image of $C$ in $\Theta_v/\langle \kappa\rangle\cong \delta^\perp$ is a Weyl chamber in the decomposition associated to $\Phi$, so the final statement from \one\ follows from Kronheimer~\cite{Kronheimer} and Lemma~\ref{lem:nequal1chamberdecomp}. The case where $C$ lies in the half-space $\{\theta \in \Theta_v\mid \theta(\delta)>0\}$ is similar. For part \two, the proof of Theorem~\ref{thm:mainintro} carries over verbatim, bearing in mind that $F$ is the closure of a unique chamber since $n=1$.
 \end{proof}

\appendix
\section{Variation of GIT for quiver varieties}

In the appendix we describe the properties of quiver varieties under variation of GIT that are required in the main body of the article. We adopt once again the notation and assumptions of section~\ref{sec:Nakasec}, so that $(\bv,\bw)$ are a pair of dimension vectors for a fixed graph with vertex set $\{0,1,\dots,r\}$ and $v$ is the dimension vector of the corresponding framed doubled quiver $Q=(I,Q_1)$ for $I=\{\infty,0,1,\dots,r\}$.

\subsection{} Recall that a representation type $\tau$ of $v$ is a tuple $(n_0, \beta^{(0)}; \ds ; n_k, \beta^{(k)})$, where $v = \sum_{i = 0}^k n_i \beta^{(i)}$, $\beta^{(i)} \in \Sigma_{\theta}$ and if $\beta^{(i)} = \beta^{(j)}$ for some $i \neq j$, then $p(\beta^{(i)}) > 0$, i.e. $\beta^{(i)}$ is imaginary. Then $\mf{M}_{\theta}(\bv,\bw)$ admits a finite stratification 
$$
\mf{M}_{\theta}(\bv,\bw) = \bigsqcup_{\tau} \mf{M}_{\theta}(\bv,\bw)_{\tau}
$$
by smooth locally closed subvarieties (each one being symplectic), where the union is over all representation types of $v$. When $\theta = 0$, there is always the ``minimal representation type'' $\boldsymbol{0} = (v_{\infty},e_{\infty};v_0,e_0;\ds ;v_r,e_r)$ of $v$, where $\{e_\infty,e_0,e_1,\ds ,e_r\}$ is the integer basis of $\Z^I$. This corresponds to the unique closed stratum $\mf{M}_{0}(\bv,\bw)_{\boldsymbol{0}}$ in $\mf{M}_{0}(\bv,\bw)$. In all that follows, we write $\mf{M}_{\theta} := \mf{M}_{\theta}(\bv,\bw)$. 

\begin{thm}
\label{thm:birationalimage}
Let $\theta \ge \theta_0 \in \Theta_{v}$ such that $\mf{M}_{\theta} \neq \emptyset$. There exists a unique representation type $\tau$ such that:
\begin{enumerate}
\item[\one] the morphism $f\colon \mf{M}_{\theta} \rightarrow \mf{M}_{\theta_0}$ obtained by variation of GIT quotient satisfies $\mathrm{Im} \, f = \overline{\mf{M}_{\theta_0,\tau}}$; and
\item[\two] the resulting morphism $f \colon \mf{M}_{\theta} \rightarrow \overline{\mf{M}_{\theta_0, \tau}}$ is birational and semi-small. 
\end{enumerate}
\end{thm}

\begin{remark}
The above theorem was shown by Nakajima \cite{Nak1994,NakajimaBranching} in the case where $\mf{M}_{\theta}$ is smooth. We give a different proof that does not rely on the topological arguments of \textit{loc.\ cit.}
\end{remark}

\subsection{Proof} The difficult part of Theorem \ref{thm:birationalimage} is to show that $f$ is semi-small. We reduce this statement to the following, whose proof relies on a result of Bozec. Recall that a root $\alpha$ is \textit{anisotropic} if $p(\alpha) > 1$. 

\begin{prop}\label{prop:keyssresult}
Assume that $v$ is an anisotropic root and $\theta \in \Theta_{v}$ sufficiently general and satisfies $v \in \Sigma_{\theta}$. If $f : \mf{M}_{\theta} \rightarrow \mf{M}_{0}$ is the corresponding projective morphism then 
$$
2 \dim f^{-1}(0) \le \dim \mf{M}_{\theta} - \dim \mf{M}_{0,\boldsymbol{0}}.
$$
\end{prop}

\begin{proof}
We begin by establishing  
\begin{equation}\label{eq:showsssemistable}
2 \dim \big(f^{-1}(0) \cap \mf{M}_{\theta}^s\big) \le \dim \mf{M}_{\theta} - \dim \mf{M}_{0,\boldsymbol{0}}.
\end{equation}
Abusing notation, let us write $\Omega$ for a quiver whose double $\overline{\Omega}$ equals the framed doubled quiver $Q$. Let $N$ be the number of loops in $\Omega$, so that $Q$ has $2N$ loops and $\dim \mf{M}_{0,\boldsymbol{0}} = 2N$ (one has the freedom to assign to each loop in $Q$ a scalar). Let $L \subset \Omega_1$ be the set of loops in $\Omega$, and let $\Lambda \subset \Rep(Q,v)$ denote the closed subset of \textit{seminilpotent} representations, as defined by Bozec~\cite{BozecCrystal}. Then $\Lambda$ is isotropic by \cite[Lemma 1.2]{BozecCrystal}, so $\dim \Lambda \le \frac{1}{2} \dim \Rep(Q,v)$. We define 
$$
\Lambda_0 = \{ x \in  \Lambda \ | \ \Tr(a^*) = 0, \ \forall \ a \in L \}.
$$
Then taking trace defines an isomorphism $ \Lambda \cong \Lambda_0 \times \C^{N}$. We note that $f^{-1}(0) \cap \mf{M}_{\theta}^s \subset  \Lambda_0^s \, \git \, G(v)$. Also note that $PG(v)$ acts freely on $\Lambda_0^s$. Thus,
\begin{align*}
2 \dim \big(f^{-1}(0) \cap \mf{M}_{\theta}^s\big) & \le 2 \dim \Lambda_0^s \, \git \, G \ \le 2 \ \dim \Lambda_0^s - 2 \dim PG(v) \\
& = 2 \dim \Lambda - 2 \dim PG(v) - \dim \mf{M}_{0,\boldsymbol{0}} \\
& \le \dim \Rep(Q,v) - 2 \dim PG(v) - \dim \mf{M}_{0,\boldsymbol{0}}  \\
& =  \dim \mf{M}_{\theta} - \dim \mf{M}_{0,\boldsymbol{0}},
\end{align*}
demonstrating inequality \eqref{eq:showsssemistable}. 

Next, write $v = d \alpha$, where $d \in \Z_{> 0}$ and $\alpha$ an indivisible dimension vector. We recall from \cite[Theorem 2.2]{BelSchQuiver} that $\alpha$ is an anisotropic root, and so too is every multiple $m \alpha$ of it. Our assumption that $\theta$ is sufficiently general is introduced to ensure that $\theta(\beta) \neq 0$ for all positive roots $\beta \le v$ that are \textit{not} a multiple of $\alpha$. It then follows from \cite[Theorem 2.2]{BelSchQuiver} that either (a) $\alpha \in \Sigma_{\theta}$, or (b) every proper multiple of $\alpha$ belongs to $\Sigma_{\theta}$ but $\alpha$ itself does not. Then the only representation types of $v$ are 
$$
\nu = (n_0,\nu_0 \alpha; \ds, ; n_k, \nu_k \alpha),
$$
the ``weighted partitions'' of $d$, as described in \cite[\S 6.1]{BelSchQuiver}, where we only allow weighted partitions with $\nu_i > 1$ for all $i$ if we are in case (b). Since this stratification is finite, it suffices to show that
$$
2 \dim \big(f^{-1}(0) \cap \mf{M}_{\theta,\nu}\big) \le \dim \mf{M}_{\theta} - \dim \mf{M}_{0,\boldsymbol{0}}
$$
 for every such $\nu$. For any root $\gamma \le v$ satisfying $\theta(\gamma) = 0$, we write $\mf{M}_{\theta}(\gamma)$ for the quiver variety associated to $\gamma$. 
 Each closed point of $\mf{M}_{\theta,\nu}$ has the form $\bigoplus_{0\leq i\leq k} M_i^{\oplus n_i}$, where $M_i \in \mf{M}_{\theta}(\nu_i \alpha)^s$ and $M_i \not\cong M_j$ for $i \neq j$. If $\mc{V}_i$ denotes the flat family of $\theta$-stable $\Pi$-modules of dimension vector $\nu_i\alpha$ on $ \mf{M}_{\theta}(\nu_i \alpha)^s$ for $0\leq i\leq k$, then the fibre of the vector bundle $\bigoplus_{0\leq i\leq k} \mc{V}_i^{\oplus n_i}$ over the closed point $\bigoplus_{0\leq i\leq k} M_i^{\oplus n_i}$ in the variety $\prod_{0\leq i\leq k} \mf{M}_{\theta}(\nu_i \alpha)^s$ is $\theta$-semistable, because each $M_i$ is $\theta$-stable. The universal property of the coarse moduli space $\mf{M}_\theta$ then gives a morphism
 \[
 \xi \colon \prod_{0\leq i\leq k} \mf{M}_{\theta}(\nu_i \alpha)^s \longrightarrow \mf{M}_{\theta}
 \]
 whose image is contained in the closure of $\mf{M}_{\theta,\nu}$. In particular, there is an open subset $U\subseteq \prod_{0\leq i\leq k} \mf{M}_{\theta}(\nu_i \alpha)^s$ such that $\xi(U)=\mf{M}_{\theta,\nu}$. For each $0\leq i\leq k$, we write $f_{\nu_i} \colon \mf{M}_{\theta}(\nu_i\alpha) \rightarrow \mf{M}_{0}(\nu_i\alpha)$ for the projective morphism obtained by variation of GIT quotient, and let $M_i\in \mf{M}_\theta(\nu_i\alpha)^s$ satisfy $(M_0,\dots,M_k)\in U$. After identifying the closed point $0\in \mf{M}_0$ with $(0,\dots,0)\in \prod_i\mf{M}_0(\nu_i\alpha)$, we have that
 \[
 0=f\big(\xi(M_0,\dots,M_k)\big) = f\big(M_0^{\oplus n_0}\oplus \cdots \oplus M_k^{\oplus n_k}\big) = f_{\nu_0}(M_0)^{\oplus n_0}\oplus \cdots \oplus f_{\nu_k}(M_k)^{\oplus n_k}
 \]
 if and only if $f_{\nu_i}(M_i)=0\in \mf{M}_0(\nu_i\alpha)$ for all $0\leq i\leq k$. Therefore the locus $\xi^{-1}\big(f^{-1}(0) \cap \mf{M}_{\theta,\nu})$ coincides with $U \cap \left( \prod_i f_{\nu_i}^{-1}(0)\cap \mf{M}_{\theta}(\nu_i \alpha)^s \right)$. Now, $U$ is dense in $\prod_{0\leq i\leq k} \mf{M}_{\theta}(\nu_i \alpha)^s$, and $\xi$ has finite fibres by a dimension count, so we deduce that 
\begin{equation}
    \label{eqn:fibredims}
\dim \big(f^{-1}(0) \cap \mf{M}_{\theta,\nu}) = \sum_i \dim \left( f_{\nu_i}^{-1}(0)\cap \mf{M}_{\theta}(\nu_i \alpha)^s\right).
\end{equation}
Inequality \eqref{eq:showsssemistable}, applied to each $v = \nu_i \alpha$, gives  
$$
2 \dim \big(f^{-1}_{\nu_i}(0) \cap \mf{M}_{\theta}(\nu_i \alpha)^s\big) \le \dim \mf{M}_{\theta}(\nu_i \alpha) - \dim \mf{M}_{0}(\nu_i \alpha)_{\boldsymbol{0}}
$$
for all $0\leq i\leq k$. Combining this with \eqref{eqn:fibredims} gives
$$
2\dim \big(f^{-1}(0) \cap \mf{M}_{\theta,\nu}\big) \le \sum_i \big(\dim \mf{M}_{\theta}(\nu_i \alpha) - \dim \mf{M}_{0}(\nu_i \alpha)_{\boldsymbol{0}}\big).
$$
Since $\dim \mf{M}_{0}(m\alpha)_{\boldsymbol{0}} = 2N$ is independent of $m$, and since $\sum_i \dim \mf{M}_{\theta}(\nu_i \alpha) = \dim \mf{M}_{\theta,\nu} \le \dim \mf{M}_{\theta}$, we have 
$$
2\dim \big(f^{-1}(0) \cap \mf{M}_{\theta,\nu}\big) \le \dim \mf{M}_{\theta} - \dim \mf{M}_{\theta,\boldsymbol{0}},
$$
which implies the statement of the proposition.
\end{proof}

\begin{proof}[Proof of Theorem \ref{thm:birationalimage}]
For \one, it was shown in \cite{BelSchQuiver} that $\mf{M}_{\theta}$ and $\mf{M}_{\theta_0}$ are irreducible symplectic varieties, whose leaves are precisely the strata $\mf{M}_{\theta,\tau}$ and $\mf{M}_{\theta_0,\tau}$ respectively. As noted in \cite[Lemma 2.4]{BelSchQuiver}, the morphism $f$ is a projective Poisson morphism. Therefore $\mathrm{Im} \, f $ is a closed, irreducible Poisson subvariety of $\mf{M}_{\theta_0}$, so there must exist a representation type $\tau$ such that $\mathrm{Im} \, f = \overline{\mf{M}_{\theta_0,\tau}}$. 

For \two, we show first that $f$ is semi-small. It suffices to show that 
$$
2 \dim f^{-1}(x) \le \dim \mf{M}_{\theta} -  \dim \mf{M}_{\theta_0,\eta}
$$
for each $x \in \mf{M}_{\theta_0,\eta}$ with $\mf{M}_{\theta_0,\eta} \subset \overline{\mf{M}_{\theta_0,\tau}}$, where $\tau$ is specified as above. Applying Theorem \ref{thm:commdiagram22} to the Ext-graph associated to $x$, we may assume that $\theta_0 = 0$, $\eta = \boldsymbol{0}$ and $x = 0$ in $\mf{M}_{0,\boldsymbol{0}}$. Thus, we are reduced to showing that $2 \dim f^{-1}(0) \le \dim \mf{M}_{\theta} -  \dim \mf{M}_{0,\boldsymbol{0}}$. We begin by assuming that $v$ has trivial canonical decomposition, i.e. $v \in \Sigma_{\theta}$. There is a trichotomy here: either
\begin{itemize}
\item[(a)] $v$ is a real root, in which case $\mf{M}_{\theta}$ is a point (and the statement is vacuous),
\item[(b)] $v$ is an isotropic imaginary root, in which case it follows from \cite[Lemma 4.1]{BelSchQuiver} that $\mf{M}_{\theta}$ is a partial resolution of the Kleinian singularity $\mf{M}_{0}$ (and hence $f$ is semi-small); or 
\item[(c)] $v$ is anisotropic. 
\end{itemize}
We need only deal with case (c). In this case, choose some $\theta' \in \Theta_{v}$ such that such that $\theta' \ge \theta$ and $\theta'(\beta) \neq 0$ for all positive roots $\beta < v$ that are not a multiple of $v$. Then we have a commutative diagram 
$$
\begin{tikzcd}
\mf{M}_{\theta'} \ar[r,"{f'}"] \ar[rr,"{h}"',dashed,bend right=20] & \mf{M}_{\theta} \ar[r,"f"] & \mf{M}_{0}.
\end{tikzcd}
$$
Since $v \in \Sigma_{\theta}$, we have $v \in \Sigma_{\theta'}$ and part \one\ applied to $h$ implies that $f'$ is surjective because every $\theta$-stable representation is $\theta'$-stable (and $\theta$-stable representations exist by Theorem~\ref{thm:canondecomp}). The dimension of $\mf{M}_{\theta}$ and $\mf{M}_{\theta'}$ both equal $2 p(v)$. Thus, it suffices to show that 
$$
2 \dim h^{-1}(0) \le \dim \mf{M}_{\theta'} - \dim \mf{M}_{0,\boldsymbol{0}}.
$$
But this is precisely the statement of Proposition \ref{prop:keyssresult}. 

Next we consider the general situation. As in the proof of Proposition \ref{prop:keyssresult}, if $\gamma \le v$ is a root with $\theta(\gamma) = 0$, then we write $\mf{M}_{\theta}(\gamma)$ for the quiver variety associated to $\gamma$. If (after grouping together like terms) $v = m_0 \gamma^{(0)} + \cdots + m_{\ell} \gamma^{(\ell)}$ is the canonical decomposition of $v$ in $\Sigma_{\theta}$, then \cite[Theorem 1.4]{BelSchQuiver} says that there is an isomorphism 
$$
\prod_{0\leq i\leq \ell} \Sym^{m_i} \mf{M}_{\theta}\big(\gamma^{(i)}\big) \stackrel{\sim}{\longrightarrow} \mf{M}_{\theta}.
$$
This gives two projective morphisms, $f \colon \mf{M}_{\theta} \rightarrow \mf{M}_0$ and
$$
\prod_i\Sym^{m_i}(f_i) \colon \mf{M}_{\theta} \rightarrow \prod_{i} \Sym^{m_i} \mf{M}_{0}\big(\gamma^{(i)}\big)
$$
 which both factor through the affinisation map $a \colon \mf{M}_{\theta} \rightarrow \mf{M}_{\theta}^{\mathrm{aff}} := \Spec \Gamma(\mf{M}_{\theta},\mc{O})$.  The induced morphisms $\mf{M}_{\theta}^{\mathrm{aff}} \rightarrow \mf{M}_0$ and $\mf{M}_{\theta}^{\mathrm{aff}} \rightarrow \prod_{i} \Sym^{m_i} \mf{M}_{0}\left(\gamma^{(i)}\right)$ are closed immersions of affine cones, so  
\begin{equation}\label{eq:ainfinv}
f^{-1}(0) = a^{-1}(0) = \left(\prod_i\Sym^{m_i}(f_i)\right)^{-1}(0,\dots,0). 
\end{equation}
As in the first part of the proof of Proposition \ref{prop:keyssresult}, we assume that $v$ is sincere, i.e. $\Supp\, v = Q$ and that $Q$ has $2N$ loops, so that $\dim \mf{M}_{0,\boldsymbol{0}} = 2N$. Let $2N_i$ denote the number of loops appearing in the subquiver $\Supp \, \gamma^{(i)}$ of $Q$ so that $\dim \mf{M}_{0,\boldsymbol{0}}(\gamma^{(i)}) = 2N_i$. Then, equality \eqref{eq:ainfinv} implies that  
\begin{align*}
2 \dim f^{-1}(0) & = 2 \sum_{i} \dim \Sym^{m_i}(f_i)^{-1}(0) =  2 \sum_{i} m_i \dim f_i^{-1}(0)\\
& \le \sum_{i} m_i \left(\dim \mf{M}_{\theta}(\gamma^{(i)}) - \mf{M}_{0}(\gamma^{(i)})_{\boldsymbol{0}} \right) \\
& = \dim \mf{M}_{\theta} - \sum_{i} m_i \dim \mf{M}_{0,\boldsymbol{0}}(\gamma^{(i)}) = \dim \mf{M}_{\theta} - \sum_{i} 2m_i N_i.
\end{align*}
The fact that $v = m_0 \gamma^{(0)} + \cdots + m_{\ell} \gamma^{(\ell)}$ implies that each loop in $Q$ must appear in the support of at least one of the $\gamma^{(i)}$. Hence $\sum_{i = 0}^{\ell} N_i \ge N$. This means that $\sum_{i = 0}^{\ell} 2m_i N_i \ge 2 N$ and hence 
$$
2 \dim f^{-1}(0) \le \dim \mf{M}_{\theta} - \dim \mf{M}_{0,\boldsymbol{0}}
$$
as required. 

Finally, we must show that $f$ is birational onto its image. Since $f$ is semi-small, it is generically finite. Therefore it suffices to show that a generic fibre is connected. Let $x \in \mf{M}_{\theta_0,\tau}$ be generic. Passing to the Ext-graph at $x$, it suffices to show that $f^{-1}(0)$ is connected when $f : \mf{M}_{\theta} \rightarrow \mf{M}_{0}$ has the property that $\mathrm{Im} \, f = \mf{M}_{0,\boldsymbol{0}}$. As in the proof of Proposition \ref{prop:keyssresult}, let $L \subset \Omega_1$ denote the loops in $\Omega$. Then taking trace of each $a,a^*$ for $a \in L$ defines an isomorphism $\mf{M}_{\theta} \cong \mf{M}_{\theta}^0 \times \C^{2N}$, where $\mf{M}_{\theta}^0$ is the subvariety of all representation with $\Tr(a) = \Tr(a^*) = 0$ for $a \in L$. Since $\mf{M}_{0,\boldsymbol{0}} = \{ 0 \} \times \C^{2N}$, with $\{ 0 \} = \mf{M}_{0,\boldsymbol{0}}^0$, we see in this case that $\mf{M}_{\theta}^0 = f^{-1}(0)$ and $f$ is a trivial fiber bundle over $\C^{2N}$. The quiver variety $\mf{M}_{\theta}$ is irreducible. Therefore, $f^{-1}(0) = \mf{M}_{\theta}^0$ is also irreducible as required. Note that we have actually shown that $\mf{M}_0^0$ is a single point in this case and the morphism $f$ is a closed embedding, i.e. $\mf{M}_{\theta} \cong \mf{M}_{0,\boldsymbol{0}}$.    
\end{proof}

\bibliographystyle{plain}

\def\cprime{$'$} \def\cprime{$'$} \def\cprime{$'$} \def\cprime{$'$}
  \def\cprime{$'$} \def\cprime{$'$} \def\cprime{$'$} \def\cprime{$'$}
  \def\cprime{$'$} \def\cprime{$'$} \def\cprime{$'$} \def\cprime{$'$}
  \def\cprime{$'$} \def\cprime{$'$}


\end{document}